\documentclass[10pt,reqno,a4paper]{amsart}
\usepackage{wavelet}
\usepackage{graphicx}
\usepackage{color}
\usepackage{epsfig}

\usepackage{amsthm}
\usepackage{algorithm}
\usepackage{algorithmic}
\usepackage{verbatim}

\newtheorem{theorem}{Theorem}
\newtheorem{proposition}{Proposition}
\newtheorem{corollary}{Corollary}
\newtheorem{lemma}{Lemma}

\newcommand{\ip}[2]{\langle #1, #2\rangle}

\newcommand{\cI}{\mathcal{I}}

\newcommand{\da}{\mathsf{A}}
\newcommand{\db}{\mathsf{B}}
\newcommand{\exchg}{\mathsf{E}}

\newcommand{\CSH}{\mathsf{CSH}}
\newcommand{\dd}{\mathsf{D}}
\newcommand{\falpha}{\boldsymbol{\alpha}}
\newcommand{\fbeta}{\boldsymbol{\beta}}
\newcommand{\fgamma}{\boldsymbol{\gamma}}
\newcommand{\fomega}{\boldsymbol{\omega}}
\newcommand{\fzeta}{\boldsymbol{\zeta}}
\newcommand{\fmu}{\boldsymbol{\mu}}
\newcommand{\fTheta}{\boldsymbol{\Theta}}
\newcommand{\fGamma}{\boldsymbol{\Gamma}}
\newcommand{\fupsilon}{\boldsymbol{\upsilon}}
\newcommand{\FASS}{\mathsf{FAS}}
\numberwithin{equation}{section}
\newcommand{\nDC}[1]{\mathscr{D}(\mathbb{R}^{#1})} 
\newcommand{\nSCpr}[1]{\mathscr{S}'(\mathbb{R}^{#1})} 

\begin{document}

\title[Smooth Affine Shear Tight Frames]{Smooth Affine Shear Tight Frames with MRA Structure}

\author{Bin Han}
\address{(B.H.) Dept. of Math. \& Stat. Sci.,
University of Alberta, Edmonton, AB, Canada T6G 2G1.
 Phone: +1 7804924289, Fax:+1 7804926826}
\email{bhan@math.ualberta.ca}
\urladdr{http://www.ualberta.ca/$\sim$bhan}
\thanks{B.H.  was supported in part by the Natural Sciences and Engineering
Research Council of Canada (NSERC Canada) under Grant RGP 228051.}

\author{Xiaosheng Zhuang}
\address{(X.Z.) Dept. of Math., City University of Hong Kong, Tat Chee Avenue, Hong Kong.
Phone: +852 34425942, Fax:+852 34420250}
\email{xzhuang7@cityu.edu.hk}
\urladdr{
http://www6.cityu.edu.hk/ma/people/profile/zhuangxs.htm}


\begin{abstract}
Finding efficient representations is one of the most challenging and
heavily sought problems in mathematics. Representation using
shearlets recently receives a lot of attention due to their
desirable properties in both theory and applications. Using the
framework of frequency-based affine systems as developed in
\cite{Han:2012:FWS}, in this paper we introduce and systematically
study affine shear tight frames which include all known shearlet
tight frames as special cases. Our results in this paper will
resolve several key questions on shearlets. We provide a complete
characterization for an affine shear tight frame and then use it to
obtain smooth affine shear tight frames with all their generators in
the Schwarz class. Though multiresolution analysis (MRA) is the
foundation and key feature of wavelet analysis for fast numerical
implementation of a wavelet transform, all the known shearlets so
far do not possess any MRA structure and filter banks. In order to
study affine shear tight frames with MRA structure, following the
lines developed in \cite{Han:2012:FWS}, we introduce the notion of a
sequence of affine shear tight frames and then we provide a complete
characterization for it. Based on our characterizations, we present
two different approaches, i.e., non-stationary and quasi-stationary,
for the construction of sequences of affine shear tight frames with
MRA structure such that all their generators are smooth (in the
Schwarz class) and they have underlying filter banks. Consequently,
their associated transforms can be efficiently implemented using
filter banks similarly as a fast wavelet transform does.
\end{abstract}



\maketitle

\bigskip

\pagenumbering{arabic}


\section{Introduction and Motivation}
\label{sec:introduction} In the era of information, everyday and
everywhere, huge amount of information is acquired, processed,
stored, and transmitted in the form of high-dimensional digital data
through Internet, TVs, cell phones, satellites, and various other
modern communication technologies. One of the main goals in today's
scientific research is on  the efficient representation and
extraction of information in high-dimensional data. It is well known
that high-dimensional data usually exhibit anisotropic phenomena due
to data clustering of various types of structures. For example,
cosmological data normally consist of many morphological distinct
objects concentrated near lower-dimensional structures such as
points (stars), filaments, and sheets (nebulae). The anisotropic
features of high-dimensional data thus encode a large portion of
significant information. Mathematical representation systems that
are capable of capturing such anisotropic features are therefore
undoubtedly the key for the efficient representation of
high-dimensional data.

During the past decade, directional multiscale representation
systems become more and more popular due to their abilities of
resolving anisotropic features in high-dimensional data
\cite{DoVetterli, Elad.Starck:MCA, GuoLabate:2011, KingKutZhuang,
KutShahZhuang, DualTreeCWT}. Our focus of this paper is on
investigation and construction of a new type of directional
multiscale representation systems: \emph{affine shear tight frames}.
Such a type of directional multiscale representation systems has
many desirable properties including directionality,  multiresolution
analysis, smooth generators, etc..  Moreover, the affine shear
systems  have an underlying filter banks associated with the
directional affine (wavelet) systems as considered in
\cite{Han:2012:FWS}.

Before proceeding further,
let us first introduce some necessary notation and definitions.
For a function $f\in L_1(\R^d)$, the Fourier transform of $f$ is  defined to be
\[
\widehat f(\xi)  = \ft f(\xi) := \int_{\R^d} f(x) e^{-ix\cdot\xi}dx,\quad \xi\in\R^d,
\]
which can be extended to square-integrable functions in $L_2(\R^d)$ and tempered distributions.
Note that the Plancherel identity holds in $\dLp{2}$:
\[
\ip{f}{g} = \frac{1}{(2\pi)^d}\ip{\widehat f}{\widehat g},\quad f, g\in L_2(\R^d),
\]
where $\ip{f}{g}:=\int_{\R^d} f(x)\overline{g(x)}dx$. Let $U$ be a $d\times d$ real-valued  invertible matrix. Define $\|f\|^2_2:=\la f, f\ra$ and
\[
f_{U;\vk,\vn}(x):=|\det U|^{1/2} f(Ux-\vk)\cdot e^{-i\vn\cdot Ux}, \qquad\vk,\vn, x\in\R^d.
\]
Here $U$, $\vk$, and $\vn$ refer to dilation, translation, and
modulation, respectively. We shall adopt the convention that
$f_{U;\vk}:=f_{U;\vk,0}$ and  $f_{\vk,\vn}:=f_{I_d;\vk,\vn}$ with
$I_d$ being the $d\times d$ identity matrix. It is trivial to check
that $\|f_{U;\vk,\vn}\|_2 = \|f\|_2$ and
$\ip{f_{U;\vk,\vn}}{g_{U;\vk,\vn}}=\ip{f}{g}$.

Following the standard notation, we denote by $\dDC$ the linear
space of all compactly supported $C^\infty$ (test) functions with
the usual topology, and $\dDCpr$ the linear space of all
distributions; that is, $\dDCpr$ is the dual space of $\dDC$.
 $\dSC$ denotes the Schwarz class of  all rapidly decreasing functions on $\R^d$ and $\dSCpr$ is its dual space of tempered distributions on $\R^d$. Note that
an element $f\in\dSC$ has the property that its derivatives of any
orders belong to $C^\infty(\R^d)$ and have polynomial decay of
arbitrary orders. Moreover, the Fourier transform $\ft:
\dSCpr\rightarrow\dSCpr$ is well-defined, one-to-one, and onto. By
$\dlLp{p}, 1\le p<\infty$ we denote the linear space of all
(Lebesgue) measurable functions $f$ on $\dR$ such that
$\int_X|f(x)|^pdx<\infty$ for every compact set  $X\subseteq\R^d$.

Though all the discussion and results in this paper can be carried
over to any high dimensions $\R^d$ with $d\ge 2$, for simplicity of
presentation, we restrict ourselves to the two-dimensional case
only, which is the most important case in the area of directional
multiscale representations. We shall use the following matrices
throughout this paper.
\begin{equation}\label{def:basic_matrices}
\exchg := \left[
\begin{matrix}
0 & 1\\
1 & 0\\
\end{matrix}
\right],\quad \sh^\tau:= \left[
\begin{matrix}
1 & \tau\\
0 & 1\\
\end{matrix}
\right],\quad
\sh_\tau:=\left[
\begin{matrix}
1 & 0\\
\tau & 1\\
\end{matrix}
\right],\quad
\da_\lambda:=
\left[
\begin{matrix}
\lambda^2 & 0\\
0 & \lambda\\
\end{matrix}
\right],\quad \db_\lambda:=(\da_\lambda)^{-\tp}= \left[
\begin{matrix}
\lambda^{-2} & 0\\
0 & \lambda^{-1}\\
\end{matrix}
\right],
\end{equation}
where $\tau\in\R$ and $\lambda>1$. $\sh_\tau$ and $\sh^\tau$ are the
shear operations while $\da_\lambda$ is the dilation matrix. Define
$\N_0:=\N\cup\{0\}$ and $\delta:\Z^d\rightarrow\R$ denotes the Dirac
sequence such that  $\delta(0)=1$ and $\delta(\vk)=0$ for all
$\vk\in\Z^d\backslash\{0\}$. We shall use boldface letters to
indicate functions in the frequency domain; that is, $\fphi$ usually
means $\fphi=\ft \varphi$ for some function $\varphi\in L_2(\R^d)$.
A  frequency-based  affine shear system is obtained by applying
shear, dilation, and translation (modulation) to generators at
different scales. Note that $\ff(\sh_\tau\cdot)$ could be highly
tilted when $\tau$ is very large for a compactly supported function
$\ff$. To balance the shear operation, one usually considers
cone-adapted systems \cite{GKL06,GuoLabate:2007, GuoLabate:2012,
Lim:cptsht}. A cone-adapted system usually consists of three parts:
one subsystem covers the low frequency region, one subsystem covers
the horizontal cone $\{\xi\in\R^2: |\xi_2/\xi_1|\le 1\}$, and one
subsystem covers the vertical cone $\{\xi\in\R^2: |\xi_1/\xi_2|\le
1\}$. The vertical-cone subsystem could be constructed to be the
`flipped' version of the horizontal-cone subsystem. More precisely,
a function $\fphi\in L_2^{loc}(\R^2)$ will serve as the generator
for the low frequency region, a function $\fpsi\in L_2^{loc}(\R^2)$
will generate an affine system covering certain region of the
horizontal cone, and $\{\fpsi^{j,\ell}\in L_2^{loc}(\R^2):
|\ell|=r_j+1,\ldots,s_j\}$ are generators at scale $j$ that will
generate elements along the seamlines (diagonal directions
$\{\xi\in\R^2: |\xi_2/\xi_1|=\pm1\}$ ) to serve the purpose of
tightness of the system. Note that $\fpsi^{j,\ell},
|\ell|=r_j+1,\ldots,s_j$ may not come from a single generator.
Define $\fPsi_j$ to be
\begin{equation}\label{def:Psi_j}
\fPsi_j:=\{\fpsi(\sh_\ell\cdot): \ell=-r_j,\ldots,r_j\}\cup \{\fpsi^{j,\ell}(\sh_\ell\cdot): |\ell|=r_j+1,\ldots,s_j\},
\end{equation}
where $r_j$ and $s_j$ are nonnegative integers.
A \emph{ frequency-based  affine shear system} is then defined to be
\begin{equation}\label{def:FAS}
\FASS(\fphi;\{\fPsi_j\}_{j=0}^\infty)= \{\fphi_{0,\vk}: \vk\in\Z^2\}
\cup\{\fh_{\db_\lambda^j;0,\vk}: \vk\in\Z^2,
\fh\in\fPsi_j\}_{j=0}^\infty \cup\{\fh_{\db_\lambda^j\exchg;0,\vk}:
\vk\in\Z^2, \fh\in\fPsi_j\}_{j=0}^\infty.
\end{equation}
For a function $f$ on $\R^2$, observe that $f_{\exchg;0}(x,y) =
f(y,x)$; that is, $f_{\exchg;0}$ is the `flipped' version of $f$
along the line $y=x$. Note that the system
$\{\fh_{\db_\lambda^j;0,\vk}: \vk\in\Z^2, \fh\in\fPsi_j\}$ is for
the high frequency region at scale $j$ with respect to the
horizontal cone, while its `flipped' version
$\{\fh_{\db_\lambda^j\exchg;0,\vk}: \vk\in\Z^2, \fh\in\fPsi_j\}$ is
for the high frequency region at scale $j$ with respect to the
vertical cone.

Suppose $\fphi,\fpsi,\fpsi^{j,\ell}\in \nSCpr{2}$. Then there are tempered distributions $\varphi,\psi, \psi^{j,\ell}\in \nSCpr{2}$
such that $\ft \varphi =\fphi$, $\ft\psi = \fpsi$, and $\ft\psi^{j,\ell}=\fpsi^{j,\ell}$. Let
\[
\Psi_j:=\{\psi(\sh^\ell\cdot): \ell=-r_j,\ldots,r_j\}\cup
\{\psi^{j,\ell}(\sh^\ell\cdot): |\ell|=r_j+1,\ldots,s_j\}.
\]
Note that $\widehat{ f_{U;\vk,0}} = \widehat f_{U^{-\tp};0,\vk}$.
Then the system defined as in  \eqref{def:FAS} in the spatial domain
is equivalent to
\begin{equation}\label{def:tFAS}
\AS(\varphi;\{\Psi_j\}_{j=0}^\infty)=
\{\varphi(\cdot-{\vk}): \vk\in\Z^2\}
\cup\{h_{\da_\lambda^j;\vk}, h_{\da_\lambda^j\exchg;\vk} \setsp \vk\in\Z^2, h\in\Psi_j\}_{j=0}^\infty.
\end{equation}
In other words, $\FASS(\fphi;\{\fPsi_j\}_{j=0}^\infty)$ is the image
of $\AS(\varphi;\{\Psi_j\}_{j=0}^\infty)$ under the Fourier
transform. Though within the framework of tempered distributions the
frequency-based approach and the space-domain approach are
equivalent to each other, as argued in \cite{Han:2012:FWS} it is
often easier to deal with a frequency-based system. Therefore, for
simplicity of presentation, in the rest of this paper we shall
mainly discuss the frequency-based affine shear system
$\FASS(\fphi;\{\fPsi_j\}_{j=0}^\infty)$. It is  trivial to see that
a  frequency-based affine shear system
$\FASS(\fphi;\{\fPsi_j\}_{j=0}^\infty)$ defined in \eqref{def:FAS}
is a special case of  the nonhomogeneous affine (wavelet) systems
discussed in \cite[Section~3]{Han:2012:FWS}.

\subsection{Related work}
\label{subsec:relatedwork}

In 1D, it is well known that wavelet representation systems provide
optimally sparse representation for  functions $f\in L_2(\R)$ that
are smooth except for finitely many discontinuity `jumps'
\cite{Daub:book, Mallat:book}. In high dimensions, wavelet
representation systems could be obtained by using tensor product of
1D wavelet representation systems. However, tensor product
representation systems usually lack  directionality since they only
favor certain  directions such as horizontal and vertical
directions. The limitation of directionality selectivity is one of
the main reasons that  the tensor product wavelets fail to provide
optimally sparse approximation for 2D smooth functions with
singularities along a closed smooth curve (anisotropic features)
\cite{Donoho01}.
 In order to achieve flexible directionality selectivity, additional operation other than dilation and translation is needed.

Directional framelets \cite{Antoine, Han:97,  Han:2012:FWS},
directly built from the frequency plane,  achieve  directionality by
separating the frequency plane into annulus at different scales and
further splitting each annulus into different wedge shapes. More
precisely, in the frequency domain, considering the polar coordinate
$(r,\theta)$ (i.e., $(x,y)=(r\cos\theta, r\sin\theta)$), one first
constructs a pair $\{\feta(r),\fzeta(r)\}$ of 1D scaling and wavelet
functions such that
$|\feta|^2+\sum_{j\in\N_0}|\fzeta(2^{-j}\cdot)|^2=1$. Then, a 2D
scaling function $\fphi$ in the frequency domain is defined to be
$\fphi(r,\theta) := \feta(r)$, while the 2D radial wavelet function
$\fpsi$ is defined to be $\fpsi(r,\theta):=\fzeta(r)$. The function
$\fpsi(2^{-j}\cdot)$ is supported on an annulus $\{(r,\theta):
2^jc_1\le r\le 2^jc_2, \theta\in[0,2\pi)\}$. Obviously, $\fpsi$ is
an isotropic function. But directionality can be easily achieved by
splitting $\fpsi$ in the angular direction $\theta$ with a smooth
partition of unity $\alpha_{j,\ell}(\theta)$ for $[0,2\pi)$:
$\sum_{\ell=1}^ {s_j}|\alpha_{j,\ell}(\theta)|^2 = 1$,
$\theta\in[0,2\pi)$. Generators at scale $j$ is then given by
$\fpsi^{j,\ell}(r,\theta) = \fzeta(r)\alpha_{j,\ell}(\theta),
\ell=1,\ldots, s_j$. The directional framelet systems are then
obtained by applying isotropic dilation $\dn:=2^{-1}I_2$ and
translation to the generators, which result in wavelet atoms of the
form $\fpsi^{j,\ell}_{\dn^j;0,\vk}$ and the whole system is a tight
frame for $L_2(\R^2)$ with all its generators in the Schwarz class.

Although directional framelets can easily achieve directionality,
yet they still use the isotropic dilation matrices. The system is
thus too `dense' to provide optimally sparse approximation for 2D
$C^2$ functions with singularity along a close $C^2$ curve. By using
parabolic dilation $\da=\diag(2,\sqrt{2})$ instead of an isotropic
dilation, the curvelets introduced in  \cite{CD} not only can
achieve  directionality selectivity, but also provide optimally
sparse approximation for  2D $C^2$ functions away from a close $C^2$
curve; see \cite{CD, GuoLabate:2007, KutLemvigLim, KutLim} for more
details on the optimally sparse approximation. The curvelet atom is
of the form $\fpsi^{j,\ell}_{\da^{-j} R_{\theta_{j,\ell}};0,\vk}$
with $R_{\theta_{j,\ell}}$ being a rotation operation determined by
the angle $\theta_{j,\ell}$. In other words, each generator
$\fpsi^{j,\ell}$ is attached with a dilation matrix
$\dn_{j,\ell}:=\da^{-j} R_{\theta_{j,\ell}}$ that is determined by
both scaling and rotation.

The curvelets use parabolic scaling and rotation and can  achieve
both directionality and  optimally sparse approximation. However,
the rotation operation $R_\theta$ destroys the preservation of
integer lattice since $R_\theta\Z^2$ is not necessarily an integer
lattice, yet the integer lattice preservation is a very much desired
property in applications. Shearlets, introduced in \cite{GLLWW04,
GKL06, GuoLabate:2007}, replace rotation $R_\theta$ by shear
$\sh_\ell$. The shear operator not only preserves the integer
lattice $\sh_\ell\Z^2=\Z^2$, but also enables a shearlet system with
only a few generators; that is, $\fpsi^{j,\ell}$ could come from the
shear versions of  several generators (even one single generator in
the case of non-cone-adapted shearlets \cite{GLLWW06:2006}). Let
$\da_h:=\diag(4,2)$ and $\da_v:=\diag(2,4)$. A \emph{cone-adapted
shearlet system} in \cite{GKL06, GuoLabate:2007} is generated by
three generators $\varphi,\psi^h, \psi^v:=\psi^h(\exchg\cdot)$
through shear, parabolic scaling, and translation:
\begin{equation}\label{def:cone-adapted-shearlet}
\begin{aligned}
\CSH(\varphi;\{\psi^h,\psi^v\})=&\{\varphi(\cdot-\vk): \vk\in\Z^2\}
\\&\cup\{2^{3j/2}\psi^h(\sh^\ell \da_h^j\cdot-\vk): \ell=-2^j,\ldots,2^j, \vk\in\Z^2\}_{j=0}^\infty
\\&\cup\{2^{3j/2}\psi^v(\sh_\ell \da_v^j\cdot-\vk):  \ell=-2^j,\ldots,2^j, \vk\in\Z^2\}_{j=0}^\infty.
\end{aligned}
\end{equation}
It is obvious that the above shearlet system is indeed a special
case of the affine shear systems defined as in \eqref{def:tFAS} by
noting that
$2^{3j/2}\psi(\sh^\ell\da_h^j\cdot-\vk)=2^{3j/2}\psi(\sh^\ell(\da_h^j\cdot-\sh^{-\ell}\vk))=\widetilde\psi_{\da_h^j;\sh^{-\ell}\vk}$
with $\widetilde\psi:=\psi(\sh^\ell\cdot)$. The system defined above
as in \eqref{def:cone-adapted-shearlet} is in general not a tight
frame for $L_2(\R^2)$. In the case of bandlimited generators,  such
a system can be modified into a tight frame for $L_2(\R^2)$ by using
projection techniques \cite{GuoLabate:2007}, which cut the seamline
elements $\psi^h(\sh^\ell \da_h^j\cdot-\vk), \psi^v(\sh_\ell
\da_v^j\cdot-\vk)$ with $\ell=\pm 2^j$  into half pieces and then
restrict them strictly in  each cone. Such projection techniques
will result in non-smooth shearlets along the seamlines:
$\psi^{h,\pm}(\sh^{\pm 2^j} \da_h^j\cdot-\vk), \psi^{v,\pm}(\sh_{\pm
2^j} \da_v^j\cdot-\vk)$.

The non-smoothness of the seamline elements breaks down the
arguments in the proof of the optimally sparse approximation for 2D
$C^2$ functions with singularities along a close $C^2$  curve in
\cite{GuoLabate:2007}, in which at least twice differentiability is
needed for the shearlet atoms. Guo and Labate in
\cite{GuoLabate:2012} proposed another type of shearlet-like
construction. The idea is still the frequency splitting; but this
time for the rectangular strip from the Fourier transform $\fphi$ of
the Meyer 2D  tensor product  scaling function. The splitting is
applied to
$\fpsi^j:=\sqrt{|\fphi(2^{-2j-2}\cdot)|^2-|\fphi(2^{-2j}\cdot)|^2}$.
A gluing procedure is applied to the two pieces along the seamlines
coming from different cones. With appropriate construction, the
gluing procedure is smooth and the system in \cite{GuoLabate:2012}
consists of smooth shearlet-like atoms. However, due to the
inconsistency of two cones, a different dilation matrix is needed
for the glued shearlet-like atom. We shall discuss the connections
of such systems to our affine shear systems  in more details in
Section~\ref{sec:construction}.

Though there are several constructions of various shearlets
available in the literature
\cite{GKL06,GuoLabate:2007,GuoLabate:2012,Lim:cptsht}, several key
problems remain unresolved. In particular, the following three
issues:

\begin{itemize}
\item[\rm{Q1)}]Existence of smooth shearlets. The cone-adapted shearlet system is obtained by applying shear, parabolic scaling,
 and translation to a few generators. To achieve tightness of the system, the shearlet atoms along
the seamlines need to be cut into half pieces. One way to achieve
smoothness is by using the gluing procedure as in
\cite{GuoLabate:2012}. However, the system no longer has a full
shear structure and is not affine-like. Are there shear tight frames
using one or a few generators?

\item[\rm{Q2)}] Shearlets with MRA structure.  The cone-adapted
shearlets achieve directionality by using a parabolic dilation
$\pA_\lambda$ (in fact it essentially uses two parabolic dilations:
$\pA_\lambda=\diag(\lambda^2, \lambda)$ for horizontal cone, and
$\exchg\pA_\lambda \exchg=\diag(\lambda, \lambda^2)$ for the
vertical cone) and the shear matrix in \eqref{def:FAS} while try to
keep the generators $\fpsi$ at all scales to be the same. In
essence, directionality is achieved in a shear system by using
infinitely many dilation matrices so that the initial direction of
the generator $\fpsi$ is dilated and sheared to other directions.
This is the main difficulty to build a shear system having a
multiresolution structure where only a single dilation matrix is
employed. It is shown in \cite{Houska} that there is no tranditional
shearlet MRA $\{\V_j\}_{j\in\Z}$ with  scaling function $\varphi$
having nice decay property, where
$\V_j=\{\varphi_{\sh^\ell\da_\lambda^j;\vk}: \vk\in\Z^2, \ell\in
I_j\}$. In this case, the space $\V_j$ uses many (possibly
infinitely many) dilation matrices.  Are there MRA structures in
certain setting for a shearlet system?

\item[\rm{Q3)}]  Filter bank association. Once we have an MRA for a shear system, it is then natural to ask whether there also exists
an associated filter bank system for the shear system.
\cite{HanKutShen,KutSauer} have studied the filter bank system with
shear operation directly in the discrete setting and provide
characterization for such a shear filter bank system. However, it is
still not clear whether a filter bank system exists and can be
naturally induced from the constructed shear system.
\end{itemize}

\subsection{Our contributions}
\label{subsec:contribution} In this paper, since shear operation has
many nice properties in both theory (optimally sparse approximation,
rich group structures, etc., see \cite{KutLabate, KutLemvigLim}) and
applications (edge detection, inpainting, separation, etc., see
\cite{GuoLabate:2009, GuoLabate:2011, KingKutZhuang,
shearlets:book}), we shall focus on the construction of directional
multiscale representation systems with shear operation: affine shear
systems. Along the way, we will focus on the above issues as
discussed in Q1 -- Q3.

For  smoothness, we  show that by using one inner smooth generator
$\fpsi$ and only a few smooth boundary generators $\fpsi^{j,\ell}$
(8 boundary generators  in total for each scale $j$ and they are
actually generated by only 2 generators through shear and `flip' for
the non-stationary construction), we can indeed construct smooth
affine shear tight frames. In addition, in this paper, we study
sequences of affine shear systems. We show that a sequence of affine
shear tight frames naturally induces an MRA structure. We would like
to point out here that almost all existing approaches
\cite{GKL06,GuoLabate:2007,GuoLabate:2012} study only one shear
system, while it is of fundamental importance to investigate
sequences of shear systems as promoted in \cite{Han:2012:FWS}.

We propose two approaches for the construction of sequences of
smooth affine shear tight frames. One is non-stationary construction
and the other is quasi-stationary construction. The  function
$\fphi^j$ for the non-stationary construction is different at
different scale $j$, while the quasi-stationary construction is with
a fixed scaling function $\fphi^j\equiv\fphi$. These two approaches
actually share the similar  idea  of frequency splitting as that for
the construction of directional framelets: at scale $j$, a smooth 2D
wavelet function
$\fomega^j=(|\fphi^{j+1}(\lambda^{-2}\cdot)|^2-|\fphi^j|^2)^{1/2}$
is constructed in the frequency domain; then a smooth partition of
unity $\fgamma_{j,\ell}, \ell=1,\ldots, s_j$ for $\R^2$ such that
$\sum_{\ell=1}^{s_j}|\fgamma_{j,\ell}|^2 =1$ is created using shear
operations for two cones instead of rotation for the case of
directional framelets or curvelets; eventually, generators
$\fpsi^{j,\ell}$ at scale $j$ are obtained by applying
$\fgamma_{j,\ell}$ to $\fomega^j$.

By carefully designing the function $\fomega^j$, we show that we can
indeed generate a smooth affine  shear tight frame (or a sequence of
affine shear tight frames), which contains a subsystem (or a
sequence of subsystems) that is generated by only one generator. In
fact, for the non-stationary case, we will see that
$\fpsi^{j,\ell}=\fpsi$ for all $\ell$ except those $\ell$ with
respect to  seamline elements (8 in total and they can be generated
by only 2 elements). In other words, the shear operations in the
non-stationary construction can reach arbitrarily close to the
seamlines. For the quasi-stationary construction, we will see that
$\fpsi^{j,\ell}=\fpsi$ for a total number of $\ell$ that is
proportional to $\lambda^j$. In this case, the shear operators in
each cone are restricted inside an area with a fixed opening angle.
We shall discuss these two types of constructions in
Section~\ref{sec:construction} with more details.

The non-stationary construction and quasi-stationary construction
induce two types of MRA structure: non-stationary MRA and stationary
MRA. Both of these two types of MRAs are the traditional wavelet MRA
in the sense that the space $\V_j$ is generated by the function
($\fphi$ or $\fphi^j$) using a fixed dilation matrix $\dm =\lambda^2
I_2$. On the other hand, the space $\W_j$ is generated by $\fpsi$
and $\fpsi^{j,\ell}$ using many dilation matrices determined by
shears and parabolic scalings. We  show that such types of
constructions have a very close relation with the directional
framelets developed in \cite{Han:97, Han:2012:FWS}. By a simple
modification, we show that the construction of directional framelets
developed in \cite{Han:97, Han:2012:FWS} using tensor product on the
polar coordinate can be easily adapted to the setting of Cartesian
coordinate under the cone-adapted setting. For the directional
framelets, it is natural and easy to build a directional tight frame
with MRA structure and with an underlying filter bank. We show that
certain affine shear tight frames can be regarded as a subsystem of
certain directional framelets. Therefore, such affine shear tight
frames have an inherited MRA structure and filter banks from the
corresponding directional framelets. This observation implies that
the transform of such affine shear tight frames can be implemented
through the filter banks of their corresponding directional
framelets.

\subsection{Contents}
\label{subsec:contents} The structure of this paper is as follows.
In Section~\ref{sec:DFAS}, we shall provide a characterization of a
frequency-based affine shear system to be a tight frame in
$L_2(\R^2)$. Based on the characterization,  simple characterization
conditions can be obtained  for  frequency-based affine shear
systems with nonnegative generators. Then, we shall present a toy
example of  frequency-based affine shear tight frames with
characteristic function generators. In Section~\ref{sec:DFAS-J},
since sequences of frequency-based affine shear systems play a very
important role in our study of the MRA structure of affine shear
systems, we shall  characterize  a sequence of frequency-based
affine shear systems to be a sequence of affine shear  tight frames
for $L_2(\R^2)$. Correspondingly, simple characterization conditions
on sequences of frequency-based affine shear tight frames with
nonnegative generators shall be given. Based on the characterization
results, in Section~\ref{sec:construction}, we provide details for
the construction of smooth frequency-based affine shear tight
frames. Two approaches shall be introduced, one is the
non-stationary construction and the other is the quasi-stationary
construction. The connection of our construction of frequency-based
affine shear systems to other existing shear systems shall also be
addressed. In Section~\ref{sec:filterBank}, we shall investigate the
relation between our frequency-based affine shear systems and the
directional framelets in \cite{Han:97, Han:2012:FWS}. By modifying
the generators for directional framelets, we shall construct
cone-adapted directional framelets, with which a natural filter bank
 is associated. We shall show that for $\da_\lambda$ with
$\lambda>1$ being an integer,  a frequency-based affine shear tight
frame is in fact a subsystem of a cone-adapted directional framelet
and therefore a frequency-based affine shear system has also an
inherited filter bank. Some extension and discussion shall be given
in the last section.

\section{Frequency-based  Affine Shear Tight Frames}
\label{sec:DFAS}
Affine systems and their properties have been studied by many
researchers, e.g., see \cite{ChuiHeStockler,DaubHanRonShen,Han:97,
Han:2012:FWS,RonShen97}. In this section we characterize
frequency-based affine shear tight frames. Based on the
characterization, we show that very simple characterization
conditions could be obtained for frequency-based affine shear tight
frames with nonnegative generators. To prepare our study of smooth
affine shear tight frames in later sections,  we shall present a toy
example of frequency-based affine shear tight frame at the end of
this section.

For $\FASS(\fphi; \{\fPsi_j\}_{j=0}^\infty)$  given as in
\eqref{def:FAS} with $\fPsi_j$ being given as in \eqref{def:Psi_j},
we define the following functions:
\begin{equation}\label{def:IPhiIPsi}
\begin{aligned}
\cI_{\fphi}^{\vk}(\xi)
&:=\overline{\fphi(\xi)}{\fphi(\xi+2\pi\vk)},\quad\vk\in\Z^2,\quad\xi\in\R^2,
\\
\cI_{\fPsi_j}^{\vk}(\xi)
&:=\sum_{\ell=-s_j}^{s_j}\overline{\fpsi^{j,\ell}(\sh_\ell\xi)}{\fpsi^{j,\ell}(\sh_\ell(\xi+2\pi\vk))},
\quad\vk\in\Z^2, \quad\xi\in\R^2; \quad\fpsi^{j,\ell}:=\fpsi\mbox{
for } |\ell|\le r_j,
\\
\cI_{\fphi}^{\vk}(\xi)&\;=\cI_{\fPsi_j}^{\vk}(\xi):=0,
\quad\vk\in\R^2\backslash\Z^2, \quad\xi\in\R^2.
\end{aligned}
\end{equation}
We say that $\FASS(\fphi; \{\fPsi_j\}_{j=0}^\infty)$ is a \emph{frequency-based  affine shear tight frame}
for $L_2(\R^2)$ if all generators $\{\fphi\}\cup\{\fPsi_j\}_{j=0}^\infty\subseteq L_2(\R^2)$ and
\begin{equation}\label{def:tight}
\begin{aligned}
(2\pi)^2\|\ff\|_2^2=\sum_{\vk \in\Z^2} |\ip{\ff}{\fphi_{0,\vk}}|^2
&+\sum_{j=0}^\infty\sum_{\fh\in\fPsi_j}\sum_{\vk \in\Z^2}
(|\ip{\ff}{\fh_{\db_\lambda^{j};0,\vk}}|^2+|\ip{\ff}{\fh_{\db_\lambda^{j}\exchg;0,\vk}}|^2)\quad
\forall\ff\in L_2(\R^2).
\end{aligned}
\end{equation}
Using Plancherel identity, \eqref{def:tight} is equivalent to saying that $\AS(\varphi; \{\Psi_j\}_{j=0}^\infty)$ is a tight frame for $L_2(\R^2)$:
\begin{equation}\label{def:tight-t}
\begin{aligned}
\|f\|_2^2=\sum_{\vk \in\Z^2} |\ip{f}{\varphi(\cdot-\vk)}|^2
&+\sum_{j=0}^\infty\sum_{h\in\Psi_j}\sum_{\vk \in\Z^2}
(|\ip{f}{h_{\da_\lambda^{j};\vk}}|^2+|\ip{f}{h_{\da_\lambda^{j}\exchg;\vk}}|^2)\quad
\forall f\in L_2(\R^2),
\end{aligned}
\end{equation}
 where $\ft \varphi  =\fphi$ and $\{\ft h = \fh: h\in\Psi_j\}=\fPsi_j$ for $j\in\N_0$.

\subsection{Characterization of  frequency-based  affine shear tight frames}
We next  characterize the system in \eqref{def:FAS} to be a
frequency-based  affine shear tight  frame. We have the following
characterization.

\begin{theorem}\label{thm:main}
Let $\da_\lambda, \db_\lambda, \sh_\ell, \exchg$ be  defined as in
\eqref{def:basic_matrices} with $\lambda>1$ and let  $\FASS(\fphi;
\{\fPsi_j\}_{j=0}^\infty)$ be defined as in \eqref{def:FAS} with
$\{\fphi\}\cup\{\fPsi_j\}_{j=0}^\infty\subseteq L_2^{loc}(\R^2)$.
Define
$\Lambda:=\cup_{j=0}^{\infty}([\da_\lambda^{j}\Z^2]\cup[\exchg\da_\lambda^{j}\Z^2])$.
Then $\FASS(\fphi; \{\fPsi_j\}_{j=0}^\infty)$  is a  frequency-based
affine  shear tight frame for $L_2(\R^2)$ if and only if
\begin{align}\label{eq:k0}
\cI_{\fphi}^{\bf0}(\xi)+\sum_{j=0}^{\infty
}[\cI_{\fPsi_j}^{\bf0}(\db_\lambda^j\xi)+\cI_{\fPsi_j}^{\bf0}(\db_\lambda^j\exchg\xi)] = 1, \quad a.e.\; \xi\in\R^2
\end{align}
and
\begin{align}\label{eq:k}
\cI_{\fphi}^{\vk}(\xi)
+\sum_{j=0}^{\infty}
[
\cI_{\fPsi_j}^{\db_\lambda^j\vk}(\db_\lambda^j\xi)
+
\cI_{\fPsi_j}^{\db_\lambda^j\exchg\vk}(\db_\lambda^j\exchg\xi)
]
=0,\quad a.e. \;\xi\in\R^2,\vk\in\Lambda\backslash\{0\},
\end{align}
where the sum in \eqref{eq:k0} converges absolutely and the infinite
sum in \eqref{eq:k} is finite for almost every $\xi\in\R^2$.
\end{theorem}

\begin{proof}
The proof essentially follows the lines developed in
\cite[Theorem~11]{Han:2012:FWS}.  For $\ff\in\nDC{2}$, define
\begin{equation}\label{def:S(f)}
\begin{aligned}
S^J(\ff):=
\sum_{\vk \in\Z^2}
|\ip{\ff}{\fphi_{0,\vk}}|^2
&+\sum_{j=0}^J\sum_{\ell=-r_j}^{r_j}\sum_{\vk \in\Z^2}
(|\ip{\ff}{\fpsi_{\sh_\ell\db_\lambda^j;0,\vk}}|^2+|\ip{\ff}{\fpsi_{\sh_\ell\db_\lambda^j\exchg;0,\vk}}|^2)\\
&
+\sum_{j=0}^J\sum_{|\ell|=r_j+1}^{s_j}\sum_{\vk \in\Z^2}
(|\ip{\ff}{\fpsi^{j,\ell}_{\sh_\ell\db_\lambda^j;0,\vk}}|^2+|\ip{\ff}{\fpsi^{j,\ell}_{\sh_\ell\db_\lambda^j\exchg;0,\vk}}|^2).
\end{aligned}
\end{equation}
 Using \cite[Lemma~10]{Han:2012:FWS}, we have
\begin{equation}\label{def:S(f)Short}
S^J(\ff) =
(2\pi)^2\int_{\R^2}\sum_{\vk\in\Lambda}\ff(\xi)\overline{\ff(\xi+2\pi\vk)}
\Big(\cI_\fphi^\vk(\xi)
+\sum_{j=0}^J\Big[\cI_{\fPsi_j}^{\db_\lambda^j\vk}(\db_\lambda^j\xi)
+\cI_{\fPsi_j}^{\db_\lambda^j\exchg\vk}(\db_\lambda^j\exchg\xi)\Big]\Big)d\xi.
\end{equation}
Since $\lambda>1$, we have
 $B_r(0)\cap \Lambda$ is finite for any ball $B_r(0)$ with radius $r>0$. Hence, $\Lambda$ has no accumulation point. Moreover, we have
\begin{align}\label{cond:LambdaFinte}
\{j\in \Z \mid j\ge 0, \;\db_\lambda^j \vk\in\Z^2 \mbox{ or }
\db_\lambda^{j}\exchg\vk \in\Z^2 \} \mbox{ is a fnite set for every
} \vk\in\Lambda\backslash\{0\}.
\end{align}
In fact, let $\vk\in\Lambda\backslash\{0\}$. Then there exist
$j_0\in\Z$ and $\vk_0\in\Z^2$ such that $\vk = \da_\lambda^{j_0}
\vk_0$ or $\vk = \exchg \da_\lambda^{j_0} \vk_0$. Then , for
$\vk=\da_\lambda^{j_0}\vk_0$, $\db_\lambda^j\vk = \da^{j_0-j}\vk_0$
or $\db_\lambda^j\exchg\vk =
(\db_\lambda^j\exchg\da_\lambda^{j_0})\vk_0$. Since $\da_\lambda$ is
expansive and $j\ge0$, there are only finitely many $j$ such that
$\da_\lambda^{j_0-j}\vk_0$ or
$(\db_\lambda^j\exchg\da_\lambda^{j_0})\vk_0$ is in $\Z^2$. The same
is true for $\vk=\exchg\da_\lambda^{j_0}\vk_0$. Consequently,
\eqref{cond:LambdaFinte} holds.

On the one hand, since $\ff$ is compactly supported, there exists a
constant $c>0$ such that $\ff(\xi)\overline{\ff(\xi+2\pi\vk)}=0$ for
all $\xi\in\R^2$ and $|\vk|\ge c$.  On the other hand, $\Lambda\cap
B_c(0)$ is a finite set. Therefore, by \eqref{cond:LambdaFinte},
there exists a positive integer $J'$ such that
$\db_\lambda^j\vk\notin\Z^2$ and $\db_\lambda^j\exchg\vk\notin\Z^2$
for all $j\ge J'$ and $\vk\in\Lambda\backslash\{0\}$ with $|\vk|<c$.
Then for all $J>J'$, \eqref{def:S(f)Short} becomes
\begin{equation}\label{def:S(f)Short2}
\begin{aligned}
S^J(\ff) &= (2\pi)^2\int_{\R^2}|\ff(\xi)|^2 \Big(\cI_\fphi^0(\xi)
+\sum_{j=0}^J\Big[\cI_{\fPsi_j}^{0}(\db_\lambda^j\xi)
+\cI_{\fPsi_j}^{0}(\db_\lambda^j\exchg\xi)\Big]\Big)d\xi
\\&+(2\pi)^2\int_{\R^2}\sum_{\vk\in\Lambda\backslash\{0\}}\ff(\xi)\overline{\ff(\xi+2\pi\vk)}
\Big(\cI_\fphi^\vk(\xi)
+\sum_{j=0}^\infty\Big[\cI_{\fPsi_j}^{\db_\lambda^j\vk}(\db_\lambda^j\xi)
+\cI_{\fPsi_j}^{\db_\lambda^j\exchg\vk}(\db_\lambda^j\exchg\xi)\Big]\Big)d\xi.
\end{aligned}
\end{equation}
Since \eqref{eq:k} holds, the above equation \eqref{def:S(f)Short2} becomes
\[
S^J(\ff) = (2\pi)^2\int_{\R^2}|\ff(\xi)|^2 \Big(\cI_\fphi^0(\xi)
+\sum_{j=0}^J\Big[\cI_{\fPsi_j}^{0}(\db_\lambda^j\xi)
+\cI_{\fPsi_j}^{0}(\db_\lambda^j\exchg\xi)\Big]\Big)d\xi,\quad J>J'.
\]
Note that all $\cI_\fphi^0$ and $\cI_{\fPsi_j}^{0}$ are nonnegative functions. By
\eqref{eq:k0} and the Lebesgue Dominated Convergence Theorem, we have
\begin{equation}\label{def:S(f)short3}
\lim_{J\rightarrow\infty}S^J(\ff)= (2\pi)^2\int_{\R^2}|\ff(\xi)|^2
\lim_{J\to \infty} \Big(\cI_\fphi^0(\xi)
+\sum_{j=0}^J\Big[\cI_{\fPsi_j}^{0}(\db_\lambda^j\xi)
+\cI_{\fPsi_j}^{0}(\db_\lambda^j\exchg\xi)\Big]\Big)d\xi=
(2\pi)^2\|\ff\|_2^2, \quad \ff\in\nDC{2}.
\end{equation}
Therefore, by that $\nDC{2}$ is dense in $L_2(\R^2)$, we conclude
that $\{\fphi\}\cup\{\fPsi_j\}_{j=0}^\infty \subseteq L_2(\R^2)$ and
\eqref{def:FAS} is a  frequency-based  affine shear tight frame for
$L_2(\R^2)$.

Conversely, suppose  that \eqref{def:FAS} is a  frequency-based
affine shear tight frame for $L_2(\R^2)$.
 Consider
\begin{equation}\label{def:S(fg)}
\begin{aligned}
S^J(\ff,\fg):= \sum_{\vk \in\Z^2}
\ip{\ff}{\fphi_{0,\vk}}\ip{\fphi_{0,\vk}}{\fg}
&+\sum_{j=0}^J\sum_{\fh\in\fPsi_j}
(\ip{\ff}{\fh_{\db_\lambda^j;0,\vk}}\ip{\fh_{\db_\lambda^j;0,\vk}}{\fg}
+\ip{\ff}{\fh_{\db_\lambda^j\exchg;0,\vk}}\ip{\fh_{\db_\lambda^j\exchg;0,\vk}}{\fg}).
\end{aligned}
\end{equation}
Then by the polarization identity, that \eqref{def:FAS} is tight is
equivalent to that
\[
\lim_{J\rightarrow\infty}S^J(\ff,\fg)=(2\pi)^2\ip{\ff}{\fg}\quad
\forall \ff,\fg\in\nDC{2}.
\]
Similarly, by \cite[Lemma 10]{Han:2012:FWS}, we have
\begin{equation}\label{def:S(fg)Short2}
\begin{aligned}
S^J(\ff,\fg) &= (2\pi)^2\int_{\R^2}\ff(\xi)\overline{\fg(\xi)}
\Big(\cI_\fphi^0(\xi)
+\sum_{j=0}^J\Big[\cI_{\fPsi_j}^{0}(\db_\lambda^j\xi)
+\cI_{\fPsi_j}^{0}(\db_\lambda^j\exchg\xi)\Big]\Big)d\xi
\\&+(2\pi)^2\int_{\R^2}\sum_{\vk\in\Lambda\backslash\{0\}}\ff(\xi)\overline{\fg(\xi+2\pi\vk)}
\Big(\cI_\fphi^\vk(\xi)
+\sum_{j=0}^\infty\Big[\cI_{\fPsi_j}^{\db_\lambda^j\vk}(\db_\lambda^j\xi)
+\cI_{\fPsi_j}^{\db_\lambda^j\exchg\vk}(\db_\lambda^j\exchg\xi)\Big]\Big)d\xi.
\end{aligned}
\end{equation}
Note that the set $\Lambda$ is discrete and closed; that is, for any
$\vk\in\Lambda$,
$\varepsilon_\vk:=2\pi\cdot\inf_{y\in\Lambda}\|y-\vk\|_2/2>0$. For
any $\vk\in\Lambda\backslash\{0\}$ and $\xi_0\in\R^2$, considering
$\ff,\fg\in\nDC{2}$ such that $\supp\ff\subseteq
B_{\varepsilon_\vk}(\xi_0)$ and $\supp\fg\subseteq
B_{\varepsilon_\vk}(\xi_0-2\pi\vk)$, then we have
\begin{equation}\label{def:S(f)short4}
\begin{aligned}
(2\pi)^2\int_{\R^2}\ff(\xi)\overline{\fg(\xi+2\pi\vk)}
\Big(\cI_\fphi^\vk(\xi)
+\sum_{j=0}^\infty[\cI_{\fPsi_j}^{\db_\lambda^j\vk}(\db_\lambda^j\xi)
+\cI_{\fPsi_j}^{\db_\lambda^j\exchg\vk}(\db_\lambda^j\exchg\xi)]\Big)d\xi
=\lim_{J\rightarrow\infty} S^J(\ff,\fg)= (2\pi)^2\ip{\ff}{\fg}=0.
\end{aligned}
\end{equation}
Now from the above relation as in \eqref{def:S(f)short4} we deduce
that \eqref{eq:k} holds for almost  every $\xi\in
B_{\varepsilon_\vk}(\xi_0)$. Since $\xi_0$ is arbitrary,  we see
that \eqref{eq:k} must hold for almost every $\xi\in\R^2$.
Consequently, we conclude from \eqref{def:S(f)Short2} that
\be \label{to1}
\lim_{J\to \infty} (2\pi)^2 \int_{\R^2}|\ff(\xi)|^2
 \Big(\cI_\fphi^0(\xi)
+\sum_{j=0}^J\Big[\cI_{\fPsi_j}^{0}(\db_\lambda^j\xi)
+\cI_{\fPsi_j}^{0}(\db_\lambda^j\exchg\xi)\Big]\Big)d\xi=\lim_{J\to
\infty} S^J(\ff)= (2\pi)^2\|\ff\|_2^2
\ee
for all $\ff\in\nDC{2}$. Since all $\cI_\fphi^0$ and
$\cI_{\fPsi_j}^{0}$ are nonnegative functions, by Monotone
Convergence Theorem, it follows from \eqref{to1} that
\[
(2\pi)^2 \int_{\R^2}|\ff(\xi)|^2 \Big(\cI_\fphi^0(\xi)
+\sum_{j=0}^\infty \Big[\cI_{\fPsi_j}^{0}(\db_\lambda^j\xi)
+\cI_{\fPsi_j}^{0}(\db_\lambda^j\exchg\xi)\Big]\Big)d\xi=
(2\pi)^2\|\ff\|_2^2
\]
for all  $\ff\in\nDC{2}$, from which we deduce that \eqref{eq:k0}
must hold.
\end{proof}
When all generators $\fphi, \fpsi, \fpsi^{j,\ell}$ are nonnegative,
the characterization in  Theorem~\ref{thm:main} is simplified as follows.
\begin{corollary}\label{cor:main}
Let $\da_\lambda, \db_\lambda, \sh_\ell, \exchg$ be defined as in
\eqref{def:basic_matrices} with $\lambda>1$ and let $\FASS(\fphi;
\{\fPsi_j\}_{j=0}^\infty)$ be defined as in \eqref{def:FAS} with
$\{\fphi\}\cup\{\fPsi_j\}_{j=0}^\infty\subseteq L_2^{loc}(\R^2)$.
Suppose
\begin{equation}\label{eq:nonnegativeGenerators}
\fh(\xi) \ge0, \quad a.e.\, \xi\in \R^2,\; \forall \fh\in\{\fphi\}\cup\{\fPsi_j\}_{j=0}^\infty.
\end{equation}
 Then
$\FASS(\fphi; \{\fPsi_j\}_{j=0}^\infty)$  is
a  frequency-based   affine  shear tight frame for $L_2(\R^2)$  if and only if
\begin{equation}\label{cond:PartionOfUnity}
|\fphi(\xi)|^2+
\sum_{j=0}^{\infty}\left(\sum_{\ell=-r_j}^{r_j}\left[|\fpsi(\sh_\ell\db_\lambda^{j}\xi)|^2
+|\fpsi(\sh_\ell\db_\lambda^{j}\exchg\xi)|^2\right]
+\sum_{|\ell|=r_j+1}^{s_j}\left[|\fpsi^{j,\ell}(\sh_\ell\db_\lambda^{j}\xi)|^2
+|\fpsi^{j,\ell}(\sh_\ell\db_\lambda^{j}\exchg\xi)|^2\right]\right)
 = 1
\end{equation}
for $a.e.\; \xi\in\R^2$ and
\begin{equation}\label{cond:non-overlap}
\fh(\xi)\fh(\xi+2\pi \vk) = 0, \; a.e.\;\xi\in\R^2, \;
\forall\vk\in\Z^2\backslash\{0\},\mbox{ and } \forall
\fh\in\{\fphi\}\cup\{\fPsi_j\}_{j=0}^\infty.
\end{equation}
\end{corollary}
\begin{proof}
Obviously,   \eqref{eq:k0} is
equivalent to \eqref{cond:PartionOfUnity}. When all generators are nonnegative,
 \eqref{eq:k} is equivalent to $\fphi(\xi)\fphi(\xi+2\pi \vk)=0$,
 $\fpsi(\xi)\fpsi(\xi+2\pi \vk) = 0$, and $\fpsi^{j,\ell}(\xi)\fpsi^{j,\ell}(\xi+2\pi\vk)=0$
 for almost every $\xi\in\R^2$, $|\ell|=r_j+1,\ldots,s_j$,  and $\vk\in\Z^2\backslash\{0\}$, which is  \eqref{cond:non-overlap}.
\end{proof}

By Corollary~\ref{cor:main}, we see that when all generators are
nonnegative, condition \eqref{cond:PartionOfUnity} is essentially
saying that a partition of unity on the frequency plane is required
for the system $\FASS(\fphi; \{\fPsi_j\}_{j=0}^\infty)$ to be a
tight frame for $L_2(\R^2)$. Condition \eqref{cond:non-overlap} says
that each generator should not overlap with its $2\pi$-shifted
version. In summary, the characterization in Theorem~\ref{thm:main}
is simplified to a partition of unity condition and a
non-overlapping condition.

\subsection{A toy example of  frequency-based  affine shear tight  frames}
\label{subsec:toyExample} To prepare for our study of smooth affine
shear tight frames in later sections,  we next give a simple example
of frequency-based affine shear tight frames generated by
characteristic functions.

Let $\lambda>1$ and define
$\ell_{\lambda^j}:=\lfloor\lambda^j-1/2\rfloor+1$. Choose $\rho>0$
such that either $0<\rho\le1$ for any $\lambda>1$, or
$\lambda^2\le\rho\le 2$ for $1<\lambda\le \sqrt{2}$. Let
\[
\begin{aligned}
E&:=\{\xi\in \R^2: -1/2\le \xi_2/\xi_1 \le 1/2, |\xi_1|\in (\lambda^{-2}\rho\pi,\rho\pi]\},\\
E_{j,+}&:=\{\xi\in \R^2: -1/2\le \xi_2/\xi_1 \le \lambda^j-\ell_{\lambda^j} , |\xi_1|\in (\lambda^{-2}\rho\pi,\rho\pi]\},\\
E_{j,-}&:=\{\xi\in \R^2: -\lambda^j+\ell_{\lambda^j}\le \xi_2/\xi_1 \le 1/2, |\xi_1|\in (\lambda^{-2}\rho\pi,\rho\pi]\}.
\end{aligned}
\]
Define
\begin{equation}\label{def:shannonPhiPsi}
\fphi:=\chi_{[-\lambda^{-2}\rho\pi,\lambda^{-2}\rho\pi]^2},\quad \fpsi:=\chi_{E}, \quad \fpsi^{j,-\ell_{\lambda^j}}:=\chi_{E_{j,+}}, \quad \fpsi^{j,\ell_{\lambda^j}}:=\chi_{E_{j,-}}.
\end{equation}
Let
\begin{equation}\label{def:shannonPsij}
\fPsi_j:=\{\fpsi(\sh_\ell\cdot):
\ell=-\ell_{\lambda^j}+1,\ldots,\ell_{\lambda^j}-1\}\cup\{\fpsi^{j,\ell_{\lambda^j}}(\sh_{\ell_{\lambda^j}}\cdot),\fpsi^{j,-\ell_{\lambda^j}}(\sh_{{-\ell_{\lambda^j}}}\cdot)\}.
\end{equation}

\begin{figure}[h]
\includegraphics[width=2.4in]{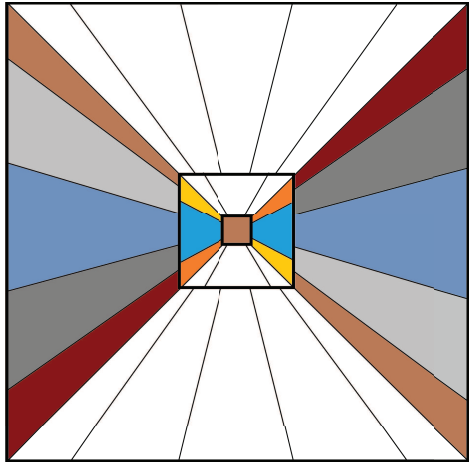}
\caption{Frequency tilings of $\FAS(\fphi;
\{\fPsi_j\}_{j=0}^\infty)$ generated by characteristic functions.
Inner rectangle:  $\fphi$.
 Middle rectangle: $\fpsi, \fpsi^{0,-1}(\sh_{-1}\cdot),
 \fpsi^{0,+1}(\sh_1\cdot)$ (color parts) and their flipped versions.
Outer rectangle: $\fpsi(\sh_\ell\db_2\cdot), \ell=-1,0,1,
\fpsi^{1,-2}(\sh_{-2}\db_2\cdot),\fpsi^{1,+2}(\sh_2\db_2\cdot)$
(color parts) and their flipped versions.}
\label{fig:shearletTilings}
\end{figure}By Corollary~\ref{cor:main}, we have the following result.
\begin{corollary}\label{cor:ShannonShealets}
Let $\da_\lambda, \db_\lambda, \sh_\ell, \exchg$ be defined as in
\eqref{def:basic_matrices} with $\lambda>1$ and let $\FASS(\fphi;
\{\fPsi_j\}_{j=0}^\infty)$ be defined as in \eqref{def:FAS}  with
$\fphi$ and $\fPsi_j$ being given as in \eqref{def:shannonPhiPsi}
and \eqref{def:shannonPsij}. Then
$\FASS(\fphi;\{\fPsi_j\}_{j=0}^\infty)$ is a  frequency-based affine
shear tight frame for $L_2(\R^2)$.
\end{corollary}
\begin{proof}
Note that for a fixed $j$, $\fpsi(\sh_\ell\db_\lambda^j\cdot)=\chi_{E_{j,\ell}}$ with
\[
E_{j,\ell} = \{\xi\in\R^2: \lambda^{-j}(-\ell-1/2)\le\xi_2/\xi_1\le \lambda^{-j}(-\ell+1/2), |\xi_1|\in(\lambda^{2j-2}\rho\pi,\lambda^{2j}\rho\pi]\}
\]
and $\fpsi^{j,-\ell_{\lambda^j}}(\sh_{-\ell_{\lambda^j}}\db_\lambda^j\cdot)=\chi_{E_{j,-\ell_{\lambda^j}}}, \fpsi^{j,\ell_{\lambda^j}}(\sh_{2^j}\db_\lambda^j\cdot)=\chi_{E_{j,\ell_{\lambda^j}}}$ with
\[
\begin{aligned}
E_{j,-\ell_{\lambda^j}} &= \{\xi\in\R^2: \lambda^{-j}(\ell_{\lambda^j}-1/2)\le\xi_2/\xi_1\le 1, |\xi_1|\in(\lambda^{2j-2}\rho\pi,\lambda^{2j}\rho\pi]\},\\
E_{j,\ell_{\lambda^j}} &= \{\xi\in\R^2: -1\le\xi_2/\xi_1\le -\lambda^{-j}(\ell_{\lambda^j}-1/2), |\xi_1|\in(\lambda^{2j-2}\rho\pi,\lambda^{2j}\rho\pi]\}.
\end{aligned}
\]
Thus, we have,
\[
\begin{aligned}
\cI_{\fPsi_j}^0(\db_\lambda^{j}\cdot)
& =\sum_{\ell=-\ell_{\lambda^j}-1}^{\ell_{\lambda^j}-1}|\fpsi(\sh_\ell\db_{\lambda}^{j}\cdot)|^2
   +|\fpsi^{j,-\ell_{\lambda^j}}(\sh_{-\ell_{\lambda^j}}\db_{\lambda}^{j}\cdot)|^2
   +|\fpsi^{j,\ell_{\lambda^j}}(\sh_{\ell_{\lambda^j}}\db_{\lambda}^{j}\cdot)|^2\\
& =\chi_{\cup_{\ell=-\ell_{\lambda^j}}^{\ell_{\lambda^j}}
E_{j,\ell}}
 =\chi_{\{\xi\in\R^2: -1\le\xi_2/\xi_1\le 1, |\xi_1|\in(\lambda^{2j-2}\rho\pi,\lambda^{2j}\rho\pi]\}}.
\end{aligned}
\]
Similarly, we have
\[
\begin{aligned}
\cI_{\fPsi_j}^0(\db_\lambda^{j}\exchg\cdot)&
&  =\chi_{\{\xi\in\R^2: -1\le\xi_1/\xi_2x\le 1,
|\xi_2|\in(\lambda^{2j-2}\rho\pi,\lambda^{2j}\rho\pi]\}}.
\end{aligned}
\]
Consequently, we obtain
\[
\cI_\fphi^0(\xi)+\sum_{j=0}^\infty[\cI_{\fPsi_j}^0(\db_\lambda^{j}\xi)+\cI_{\fPsi_j}^0(\db_\lambda^{j}\exchg\xi)] =1, \; a.e. \; \xi\in \R^2.
\]
Hence, \eqref{cond:PartionOfUnity} holds.

By our choice of $\rho$, we have $\rho\le\lambda^{2}$. Then
$-\lambda^{-2}\rho\pi+2\pi\ge \lambda^{-2}\rho\pi$, which implies
\[
(-\lambda^{-2}\rho\pi+2\pi,\lambda^{-2}\rho\pi+2\pi)\cap(-\lambda^{-2}\rho\pi,\lambda^{-2}\rho\pi)=\emptyset.
\]
Hence, $\fphi(\xi)\fphi(\xi+2\vk\pi)=0$, a.e. $\xi\in\R^2$ and
$\vk\in\Z^2\backslash\{0\}$. Similarly, the condition for $\fpsi$
that
\[
(-\rho\pi+2\pi,-\lambda^{-2}\rho\pi+2\pi)\cup (\lambda^{-2}\rho\pi+2\pi,\rho\pi+2\pi)]\cap[(-\rho\pi,-\lambda^{-2}\rho\pi)\cup(\lambda^{-2}\rho\pi,\rho\pi)]=\emptyset
\]
and
\[
(-\frac12\rho\pi+2\pi,\frac12\rho\pi+2\pi)\cap
(-\frac12\rho\pi,\frac12\rho\pi) =\emptyset,
\]
is equivalent to
\[
\rho\le\frac{2\lambda^2}{\lambda^2-1} \mbox{ and  } \rho\le1
\qquad\mbox{or}\qquad
\rho\le\frac{2\lambda^2}{\lambda^2-1} \mbox{ and  } \rho\ge\lambda^2 \mbox{ and  } \rho\le2,
\]
which is equivalent to our choice of $\rho$ and implies
$\fpsi(\xi)\fpsi(\xi+2\vk\pi)=0$, a.e. $\xi\in\R^2$ and
$\vk\in\Z^2\backslash\{0\}$. The case that
$\fpsi^{j,\pm\ell_{\lambda^j}}(\xi)\fpsi^{j,\pm\ell_{\lambda^j}}(\xi+2\vk\pi)=0$
a.e. $\xi\in\R^2$ and $\vk\in\Z^2\backslash\{0\}$ can be argued in
the same way. Hence, \eqref{cond:non-overlap} holds.

Note that all generators are nonnegative.
Therefore, by Corollary~\ref{cor:main}, $\FASS(\fphi;\{\fPsi_j\}_{j=0}^\infty)$  with $\fphi$ and $\fPsi_j$ being given as in \eqref{def:shannonPhiPsi} and
\eqref{def:shannonPsij} is a  frequency-based affine shear tight  frame for $L_2(\R^2)$.
\end{proof}
For an illustration of $\FASS(\fphi;\{\fPsi_j\}_{j=0}^\infty)$ with
$\lambda =2$, see Figure~\ref{fig:shearletTilings}. One of the main
goals of this paper is to construction smooth frequency-based affine
shear tight frames that in certain sense can be regarded as the
smoothened version of $\FASS(\fphi;\{\fPsi_j\}_{j=0}^\infty)$ as
discussed in Corollary~\ref{cor:ShannonShealets}.

\section{Sequences of  Frequency-based  Affine Shear Tight Frames}
\label{sec:DFAS-J}
Most current papers in the literature have investigated only one
single affine system. However, to have MRA structure, as argued in
\cite{Han:2012:FWS},  it is of fundamental importance to study a
sequence of affine systems. In order to study the MRA structure of
frequency-based affine shear systems, we next study sequences of
frequency-based affine shear systems. We first characterize a
sequence of frequency-based affine shear systems to be a sequence of
frequency-based affine shear tight frames for $L_2(\R^2)$. Then,
corresponding to Corollary~\ref{cor:main}, a simple characterization
shall be given for a sequence of frequency-based affine shear tight
frames with nonnegative generators.

We need an additional dilation matrix $\dm_\lambda:=\lambda^2I_2$
and we define $\dn_\lambda:=\dm_\lambda^{-\tp}=\lambda^{-2}I_2$. Let
$J$ be an integer and $\fphi^j, \fpsi, \fpsi^{j,\ell},
|\ell|=r_j+1,\ldots, s_j, j\ge J_0$ be functions in
$L_2^{loc}(\R^2)$. Let $\fPsi_j$ be defined as in \eqref{def:Psi_j}
and $\da_\lambda, \db_\lambda, \sh_\ell, \exchg$ be defined as
before. A \emph{ frequency-based  affine shear system}
$\FAS_J(\fphi^J; \{\fPsi_j\}_{j=J}^\infty)$ is then defined to be
\begin{equation}\label{def:FAS_J}
\begin{aligned}
\FAS_J(\fphi^J;
\{\fPsi_j\}_{j=J}^\infty):=&\{\fphi^J_{\dn_{\lambda}^{J};0,\vk}:
\vk\in\Z^2\}\cup\{\fh_{ \db_\lambda^j;0,\vk},
\fh_{\db_\lambda^j\exchg;0,\vk}: \vk\in\Z^2,
\fh\in\fPsi_j\}_{j=J}^\infty.
\end{aligned}
\end{equation}
We have the following characterization for a sequence of affine
shear systems $\FAS_J(\fphi^J;\{\fPsi_j\}_{j=J}^\infty)$, $J\ge J_0$
to be a sequence of affine shear tight frames for $L_2(\R^2)$.
\begin{theorem}\label{thm:main2}
Let $\dm_\lambda, \dn_{\lambda}, \da_\lambda, \db_\lambda, \sh_\ell,
\exchg$ be defined as before and $J_0$ be an integer. Let
$\FAS_J(\fphi^J; \{\fPsi_j\}_{j=J}^\infty)$ be defined as in
\eqref{def:FAS_J} with all generators
$\{\fphi^J\}\cup\{\fPsi_j\}_{j=J}^\infty\subseteq L_2^{loc}(\R^2)$
for all $J\ge J_0$. Then the following statements are equivalent to
each other.
\begin{enumerate}
\item[\rm{(1)}]  $\FAS_J(\fphi^J; \{\fPsi_j\}_{j=J}^\infty)$ is a  frequency-based  affine shear tight frame for $L_2(\R^2)$, i.e., all generators are from $L_2(\R^2)$ and for all $\ff\in L_2(\R^2)$,
\begin{equation}\label{def:tight_J}
\begin{aligned}
(2\pi)^2\|\ff\|_2^2=\sum_{\vk \in\Z^2}
|\ip{\ff}{\fphi^J_{\dn_{\lambda}^J;0,\vk}}|^2
&+\sum_{j=J}^\infty\sum_{\fh\in\fPsi_j}\sum_{\vk \in\Z^2}
(|\ip{\ff}{\fh_{\db_\lambda^{j};0,\vk}}|^2+|\ip{\ff}{\fh_{\db_\lambda^{j}\exchg;0,\vk}}|^2)
\end{aligned}
\end{equation}
for every integer $J\ge J_0$.

\item[\rm{(2)}]
 The following identities hold:
\begin{align}\label{eq:fPhiTof}
\lim_{j\rightarrow\infty} \sum_{\vk \in\Z^2}
|\ip{\ff}{\fphi^{j}_{\dn_{\lambda}^j;0,\vk}}|^2=
(2\pi)^2\|\ff\|_2^2\quad \forall \ff\in\mathscr{D}(\R^2)
\end{align}
and for all $j\ge J_0$,
\begin{equation}\label{eq:fPhiAddfPsi}
\begin{aligned}
\sum_{\vk \in\Z^2}
|\ip{\ff}{\fphi^{j+1}_{\dn_{\lambda}^{j+1};0,\vk}}|^2=&\sum_{\vk \in\Z^2}
|\ip{\ff}{\fphi^{j}_{\dn_{\lambda}^j;0,\vk}}|^2
+\sum_{\ell=-r_j}^{r_j}\sum_{\vk \in\Z^2}
(|\ip{\ff}{\fpsi_{\sh_\ell\db_\lambda^j;0,\vk}}|^2+|\ip{\ff}{\fpsi_{\sh_\ell\db_\lambda^j\exchg;0,\vk}}|^2)\\
&+\sum_{|\ell|=r_j+1}^{s_j}\sum_{\vk \in\Z^2}
(|\ip{\ff}{\fpsi^{j,\ell}_{\sh_\ell\db_\lambda^j;0,\vk}}|^2+|\ip{\ff}{\fpsi^{j,\ell}_{\sh_\ell\db_\lambda^j\exchg;0,\vk}}|^2)\quad\forall\ff\in\nDC{2}.
\end{aligned}
\end{equation}

\item[\rm{(3)}] The following identities hold:
\begin{align}\label{eq:SimpleCond3}
\lim_{j\rightarrow\infty}\ip{|\fphi^{j}(\dn_{\lambda}^j\cdot)|^2}{\fh}=
\ip{ 1}{\fh} \quad \forall \fh\in\nDC{2}
\end{align}
and for all integers $j\ge J_0$,
\begin{align}\label{eq2:k}
\cI_{\fphi^j}^{\dn_{\lambda}^j\vk}(\dn_{\lambda}^j\xi)
+\left(\cI_{\fPsi_j}^{\db_\lambda^j\vk}(\db_\lambda^j\xi)+\cI_{\fPsi_j}^{\db_\lambda^j\exchg\vk}(\db_\lambda^j\exchg\xi)\right)
=\cI_{\fphi^{j+1}}^{\dn_{\lambda}^{j+1}\vk}(\dn_{\lambda}^{j+1}\xi)\quad \end{align}
for a.e. $\xi\in\R^2$, $\vk\in([\dm_\lambda^{j}\Z^2]\cup[\dm_{\lambda}^{(j+1)}\Z^2]\cup[\da_\lambda^j\Z^2]\cup[\exchg\da_\lambda^j\Z^2])$,
where $\cI_{\fphi^j}^\vk,\cI_{\fPsi_j}^\vk$ are similarly defined as in \eqref{def:IPhiIPsi}.
\end{enumerate}
\end{theorem}
\begin{proof}
(1)$\Rightarrow$(2). Considering \eqref{def:tight_J} with two
consecutive $J$ and $J+1$ with $J\ge J_0$, the difference gives
\eqref{eq:fPhiAddfPsi}. Now by \eqref{eq:fPhiAddfPsi}, it is easy to
deduce that,
\begin{equation}\label{eq:fPhiAddfPsi_J}
\begin{aligned}
\sum_{\vk \in\Z^2}
|\ip{\ff}{\fphi^{J'}_{\dn_{\lambda}^{J'};0,\vk}}|^2&=\sum_{\vk
\in\Z^2} |\ip{\ff}{\fphi^{J}_{\dn_{\lambda}^J;0,\vk}}|^2
+\sum_{j=J}^{J'-1}\Bigg(\sum_{\ell=-r_j}^{r_j}\sum_{\vk \in\Z^2}
(|\ip{\ff}{\fpsi_{\sh_\ell\db_\lambda^j;0,\vk}}|^2+|\ip{\ff}{\fpsi_{\sh_\ell\db_\lambda^j\exchg;0,\vk}}|^2)
\\&+\sum_{|\ell|=r_j+1}^{s_j}\sum_{\vk \in\Z^2}
(|\ip{\ff}{\fpsi^{j,\ell}_{\sh_\ell\db_\lambda^j;0,\vk}}|^2+|\ip{\ff}{\fpsi^{j,\ell}_{\sh_\ell\db_\lambda^j\exchg;0,\vk}}|^2)\Bigg)\quad
\forall J'\ge J.
\end{aligned}
\end{equation}
Therefore, by letting $J'\rightarrow\infty$, we see that \eqref{eq:fPhiTof} holds.

(2)$\Rightarrow$(1). By \eqref{eq:fPhiAddfPsi}, we deduce that
\eqref{eq:fPhiAddfPsi_J} holds. By letting $J'\rightarrow\infty$ and
in view of \eqref{eq:fPhiTof}, we conclude that \eqref{def:tight_J}
holds.

(2)$\Leftrightarrow$(3). By \cite[Lemma~10]{Han:2012:FWS}, we can show that \eqref{eq:fPhiAddfPsi} is equivalent to
\begin{equation}\label{eq2:k_J}
\int_{\R^2}\sum_{\vk\in\Lambda_j}\ff(\xi)\overline{\ff(\xi+2\pi
\vk)}\left(
\Big[\cI_{\fphi^j}^{\dn_{\lambda}^j\vk}(\dn_{\lambda}^j\xi)
+\cI_{\fPsi_j}^{\db_\lambda^j\vk}(\db_\lambda^j\xi)+\cI_{\fPsi_j}^{\db_\lambda^j\exchg\vk}(\db_\lambda^j\exchg\xi)\Big]
-\cI_{\fphi^{j+1}}^{\dn_{\lambda}^{j+1}\vk}(\dn_{\lambda}^{j+1}\xi)\right)d\xi=0,
\end{equation}
where
$\Lambda_j=[\dm_{\lambda}^{j}\Z^2]\cup[\dm_{\lambda}^{j+1}\Z^2]\cup[\da_\lambda^j\Z^2]\cup[\exchg\da_\lambda^j\Z^2]$.
Since $\dm_{\lambda}=\lambda^2 I_2$ and $\da_\lambda =
\diag(\lambda^2,\lambda)$ with $\lambda>1$, we see that the lattice
$\Lambda_j$ is discrete. By the same argument as in the proof of
Theorem~\ref{thm:main}, we see that \eqref{eq2:k_J} is equivalent to
\eqref{eq2:k}.

By \cite[Lemma~10]{Han:2012:FWS}, we see that \eqref{eq:fPhiTof} is equivalent to
\begin{equation}\label{eq3:k_J}
\lim_{j\rightarrow\infty}\int_{\R^2}\sum_{\vk\in
[\dm_{\lambda}^j\Z^2]}\ff(\xi)\overline{\ff(\xi+2\pi \vk)}
\cI_{\fphi^j}^{\dn_{\lambda}^j\vk}(\dn_{\lambda}^j\xi)=(2\pi)^2\|\ff\|_2^2\quad
\forall \ff\in\nDC{2}.
\end{equation}
Since $\ff\in\mathscr{D}(\R^2)$  is compactly supported, there
exists $c>0$ such that $\ff(\xi)\overline{\ff(\xi+2\pi \vk)}=0$ for
all $\xi\in\R^2$ and $|\vk|>c$. Also, by our assumption on
$\dm_{\lambda}$ and $\da_\lambda$, it is easy to show that
\begin{equation}\label{cond:LambdaFinte_J}
\{j\in\Z^2: j\ge J_0, [\dn_{\lambda}^jB_c(0)]\cap\Z^2\neq\{0\}\} \quad\mbox{is a finite set for every}\quad c\in[1,\infty).
\end{equation}
Hence, there exists $J''\ge J_0$ such that $\ff(\xi)\overline{\ff(\xi+2\pi\vk)}\cI_{\fphi^j}^{\dn_{\lambda}^j\vk}(\dn_{\lambda}^j\xi) = 0$ for
all $\xi\in\R^2$, $\vk\in[\dm_{\lambda}^j\Z^2]\backslash\{0\}$, and $j\ge J''$. Consequently, for $j\ge J''$, \eqref{eq3:k_J} becomes
\begin{equation}\label{eq3:k0_J}
\lim_{j\rightarrow\infty}\int_{\R^2}|\ff(\xi)|^2
\cI_{\fphi^j}^{0}(\dn_{\lambda}^j\xi)=(2\pi)^2\|\ff\|_2^2 \quad
\forall \ff\in\mathscr{D}(\R^2),
\end{equation}
which is equivalent to \eqref{eq:SimpleCond3}.
\end{proof}

If all elements  $\fphi^j,\fpsi,\fpsi^{j,\ell}$ are nonnegative, we
have the following simple characterization; also see \cite[ Cor.
18]{Han:2012:FWS}.

\begin{corollary}\label{cor:positive}
Let $\dm_\lambda, \dn_{\lambda}, \da_\lambda, \db_\lambda, \sh_\ell,
\exchg$ be defined as before and $J_0$ be an integer. Let
$\FAS_J(\fphi^J; \{\fPsi_j\}_{j=J}^\infty)$ be defined as in
\eqref{def:FAS_J} with all generators
$\{\fphi^J\}\cup\{\fPsi_j\}_{j=J}^\infty\subseteq L_2^{loc}(\R^2)$
for all $J\ge J_0$. Suppose that
\begin{equation}\label{cond:positive_J}
\fh\ge0 \mbox{ for all } \fh\in\{\fphi^j,\fpsi,\fpsi^{j,\ell}: j\ge J_0, |\ell|=r_j+1,\ldots,s_j\}.
\end{equation}
Then, for all integers $J\ge J_0$,
$\FWS_J(\fphi^J;\{\fPsi_j\}_{j=J}^\infty)$ is a  frequency-based
affine shear tight frame for $L_2(\R^2)$  if and only if
\begin{align}\label{eq:SimpleCond1}
\fh(\xi)\fh(\xi+2\pi \vk)=0, \; a.e., \xi\in\R^2,
\vk\in\Z^2\backslash\{0\},\mbox{ and
}\;\fh\in\{\fphi^j,\fpsi,\fpsi^{j,\ell}: j\ge J_0,
|\ell|=r_j+1,\ldots,s_j\},
\end{align}
\begin{equation}
\label{eq:SimpleCond2}
\begin{aligned}
|\fphi^{j+1}(\dn_{\lambda}^{j+1}\xi)|^2=&|\fphi^j(\dn_{\lambda}^j\xi)|^2
+\sum_{\ell=-r_j}^{r_j}(|\fpsi(\sh_\ell\db_\lambda^j\xi)|^2+|\fpsi(\sh_\ell\db_\lambda^j\exchg\xi)|^2)
\\&+\sum_{|\ell|=r_j+1}^{s_j}(|\fpsi^{j,\ell}(\sh_\ell\db_\lambda^j\xi)|^2+|\fpsi^{j,\ell}(\sh_\ell\db_\lambda^j\exchg\xi)|^2)
,\quad a.e.,\xi\in\R^2, \; j\ge J_0,
\end{aligned}
\end{equation}
and \eqref{eq:SimpleCond3} holds.
\end{corollary}
\begin{proof}
When \eqref{cond:positive_J} holds, by item (3) of
Theorem~\ref{thm:main2}, for $\vk\in\Z^2\backslash\{0\}$,
\eqref{eq2:k} is equivalent to \eqref{eq:SimpleCond1}. For $\vk=0$,
\eqref{eq2:k} is equivalent to \eqref{eq:SimpleCond2}. Together with
the condition \eqref{eq:SimpleCond3} and by item (3) of
Theorem~\ref{thm:main2}, the claim follows from the equivalence
between item (1) and item (3) of Theorem~\ref{thm:main2}.
\end{proof}
The condition in \eqref{eq:SimpleCond3} can be further simplified as
in the following lemma.
\begin{lemma}\label{lemma:fphi}
Suppose that there exist two positive numbers $c$ and $C$ such that
\begin{equation}\label{eq:SimpleCond3Onfphi}
|\fphi^j(\xi)|\le C,\; a.e.\;\xi\in[-c,c]^2 \mbox{ and }\forall j\ge
J_0.
\end{equation}
Assume that
$\fg(\xi):=\lim_{j\rightarrow\infty}|\fphi^j(\dn_\lambda^j\xi)|^2$
exists for almost every $\xi\in\R^2$. Then \eqref{eq:SimpleCond3}
holds if and only if $\fg(\xi)= 1$, a.e. $\xi\in\R^2$.
\end{lemma}
\begin{proof}
Since $\fh\in\nDC{2}$ has compact support and
$\dn_\lambda^{-1}=\dm_\lambda$ is expansive, there exists $J\in\N$
such that
\[
|\fphi^j(\dn_\lambda^j\xi)|^2|\fh(\xi)|\le C^2|\fh(\xi)|\quad
\forall j\ge J,\; a.e.\; \xi\in\R^2.
\]
Since $\fh\in L_1(\R^2)$, by  Lebesgue Dominated Convergence
Theorem, we have
\[
\lim_{j\rightarrow\infty}\ip{|\fphi^j(\dn_\lambda^j\cdot)|^2}{\fh}=\ip{\lim_{j\rightarrow\infty}|\fphi^j(\dn_\lambda^j\cdot)|^2}{\fh}=
\ip{\fg}{\fh}.
\]
Now it is trivial to see that \eqref{eq:SimpleCond3} holds if and
only if $\ip{\fg}{\fh}=\ip{1}{\fh}$ for all $\fh\in\nDC{2}$, which
is equivalent to $\fg(\xi)=1$ for almost every $\xi\in\R^2$.
\end{proof}

Consider the toy example   in Corollary~\ref{cor:ShannonShealets}.
%
%
Now define $\fphi^j:=\fphi$ and $\fPsi_j:=\{\fpsi(\sh_\ell\cdot):
\ell=-\ell_{\lambda^j}+1,\ldots,\ell_{\lambda^j}-1\}\cup\{\fpsi^{j,\pm\ell_{\lambda^j}}\}$
with $\fpsi, \fpsi^{j,\pm\ell_{\lambda^j}}$ being constructed as in
Corollary~\ref{cor:ShannonShealets}. Then condition
\eqref{eq:SimpleCond3} holds by Lemma~\ref{lemma:fphi} since
$\fphi^j$ satisfies \eqref{eq:SimpleCond3Onfphi} and
$\fg(\xi)=\lim_{j\rightarrow\infty}|\fphi^j(\dn_\lambda^j\xi)|^2 =1$
a.e., $\xi\in\R^2$. Condition \eqref{eq:SimpleCond1} directly
follows from the proof of Corollary~\ref{cor:ShannonShealets}.
Condition \eqref{eq:SimpleCond2} holds by our construction.
Therefore, by Corollary~\ref{cor:positive},
$\FAS_J(\fphi^J;\{\fPsi_j\}_{j=J}^\infty)$ is a frequency-based
affine shear tight frame for $L_2(\R^2)$ for any integer $J\ge0$.

A sequence of frequency-based affine shear tight frames naturally
induces an MRA structure $\{\V_j\}_{j=J_0}^\infty$ with
$\V_j:=\overline{\rm span}\{\varphi^j(\dm_\lambda^j\cdot-\vk):
\vk\in\Z^2\}$, where $\varphi^j = \ft^{-1}\fphi^j$. But so far, the
toy example and its induced sequence of systems are not smooth since
all their generators are discontinuous.  In the next section, we
shall focus on the construction of  smooth  frequency-based  affine
shear tight frames for $L_2(\R^2)$ that have many desirable
properties. We shall show that not only our systems can achieve
smoothness of generators, but also they  have shear structure and
more importantly, an MRA structure could be deduced from such type
of systems.

\section{Construction of Smooth  Frequency-based  Affine Shear Tight Frames}
\label{sec:construction} In this section we shall provide two types
of constructions of smooth frequency-based affine shear tight
frames: one is non-stationary construction and the other is
quasi-stationary construction. Both these two types of constructions
use the idea of normalization. In essence, we first construct a
smooth frequency-based affine shear frame for $L_2(\R^2)$ and then a
normalization procedure is applied to such a frame. The
non-stationary construction uses different  functions $\fphi^j$ for
different scales $j$, while the quasi-stationary construction
employs a single function $\fphi$ for every scale. We first need
some auxiliary results and then provide details on the two types of
constructions.

\subsection{Auxiliary results}
We shall use a function $\nu\in C^\infty(\R)$ such that $\nu(x)=0$
for $x\le-1$, $\nu(x)=1$ for $x\ge1$, and $|\nu(x)|^2+|\nu(-x)|^2=1$
for all $x\in\R$. There are many choices of such functions. For
example, define $f(x):=e^{-1/x^2}$ for $x>0$ and $f(x):=0$ for
$x\le0$, and let $g(x):=\int_{-1}^xf(1+t)f(1-t)dt$. Define
\begin{equation}\label{def:nu}
\nu(x):=\frac{g(x)}{\sqrt{|g(x)|^2+|g(-x)|^2}},\quad x\in\R.
\end{equation}
Then $\nu\in C^\infty(\R)$ is a desired function. Using such a
function $\nu$, we now construct our building block
$\ff_{[a,b];\varepsilon_1,\varepsilon_2}$ with $a,b\in\R$,
$\varepsilon_1,\varepsilon_2>0$ and $\varepsilon_1+\varepsilon_2\le
b-a$ as follows.
\begin{equation}\label{def:fmu}
\begin{aligned}
\ff_{[a,b];\varepsilon_1,\varepsilon_2}(x)&=
\begin{cases}
\nu(\frac{x-a}{\varepsilon_1})&x<a+\varepsilon_1\\
1& a+\varepsilon_1\le x\le b-\varepsilon_2\\
\nu(\frac{-x+b}{\varepsilon_2})&x>b-\varepsilon_2.\\
\end{cases}
\end{aligned}
\end{equation}
Then, $\ff_{[a,b];\varepsilon_1,\varepsilon_2}\in C^\infty(\R)$ and $\supp \ff_{[a,b];\varepsilon_1,\varepsilon_2}=[a-\varepsilon_1,b+\varepsilon_2]$.

Define $\falpha_{\lambda,t,\rho}, \fbeta_{\lambda,t,\rho}$ with
$\lambda>1$, $0<t\le1$ ,and $0<\rho\le\lambda^2$ as follows:
\begin{equation}\label{def:falphafbeta}
\begin{aligned}
\falpha_{\lambda,t,\rho}(\xi):=\ff_{[a,b];\varepsilon_1,\varepsilon_2}(\xi)\quad\mbox{and}\quad
\fbeta_{\lambda,t,\rho}(\xi)
:=(|\falpha_{\lambda,t,\rho}(\lambda^{-2}\xi)|^2-|\falpha_{\lambda,t,\rho}(\xi)|^2)^{1/2},\;\;\xi\in\R,
\end{aligned}
\end{equation}
where
$[a,b]=[-\lambda^{-2}(1-t/2)\rho\pi,\lambda^{-2}(1-t/2)\rho\pi]$ and
$\varepsilon_1= \varepsilon_2=\lambda^{-2}t\rho\pi/2$. We have
$\supp\falpha_{\lambda,t,\rho}=[-\lambda^{-2}\rho\pi,\lambda^{-2}\rho\pi]$
and
$\supp\fbeta_{\lambda,t,\rho}=[-\rho\pi,-\lambda^{-2}(1-t)\rho\pi]\cup[\lambda^{-2}(1-t)\rho\pi,\rho\pi]$.
Furthermore, define a $2\pi$-periodic function
$\fmu_{\lambda,t,\rho}$ and $\fupsilon_{\lambda,t,\rho}$ as follows:
\begin{equation}\label{def:fmask}
\begin{aligned}
\fmu_{\lambda,t,\rho}(\xi)&:=
\begin{cases}
\frac{\falpha_{\lambda,t,\rho}(\lambda^2\xi)}{\falpha_{\lambda,t,\rho}(\xi)} &|\xi|\le\lambda^{-2}\rho\pi\\
0 & \lambda^{-2}\rho\pi<|\xi|\le \pi
\end{cases},\\
\fupsilon_{\lambda,t,\rho}(\xi)&:=
\begin{cases}
\frac{\fbeta_{\lambda,t,\rho}(\lambda^2\xi)}{\falpha_{\lambda,t,\rho}(\xi)} &\lambda^{-4}(1-t)\rho\pi\le|\xi|\le\lambda^{-2}\rho\pi\\
\fg_{\lambda, t, \rho}(\xi) & \xi\in[-\pi,\pi]\backslash\supp\fbeta_{\lambda,t,\rho}(\lambda^2\cdot),
\end{cases}
\end{aligned}
\end{equation}
where $\fg_{\lambda, t, \rho}$ is a function in $C^\infty(\T)$ such
that $\big[\frac{d^n}{d\xi^n}\fg_{\lambda, t,
\rho}(\xi)\big]\Big|_{\xi=\pm\lambda^{-2}\rho\pi} = \delta(n)$ for
all $n\in\N_0$. The purpose of $\fg_{\lambda,t,\rho}$ is to make the
function $\fupsilon_{\lambda,t,\rho}$ smooth. Such a $\fg_{\lambda,
t, \rho}$ exists. In fact, noting that
$\frac{\fbeta_{\lambda,t,\rho}(\lambda^2\xi)}{\falpha_{\lambda,t,\rho}(\xi)}
\equiv 1$ for $|\xi|\ge\lambda^{-4}\rho\pi$ and
$\frac{\fbeta_{\lambda,t,\rho}(\lambda^2\xi)}{\falpha_{\lambda,t,\rho}(\xi)}
\equiv 0$ for $|\xi|\le\lambda^{-4}(1-t)\rho\pi$, we can simply
define $\fg_{\lambda,t,\rho}$ to be $\fg_{\lambda, t,
\rho}(\xi):\equiv1 $ for $ \lambda^{-4}\rho\pi\le |\xi|\le\pi$  and
$\fg_{\lambda, t, \rho}(\xi) \equiv 0$ for $|\xi|\le
\lambda^{-4}(1-t)\rho\pi$. In this case, $\fg_{\lambda,t,\rho}$
extends periodically as a constant 1 near the boundary of $\T$. If
$\lambda^{-2}\rho<1$, then another way to make
$\fupsilon_{\lambda,t,\rho}(\xi)$ smooth is by defining
$\fg_{\lambda,t,\rho}$ to be $\fg_{\lambda, t, \rho}(\xi):\equiv1 $
for $\lambda^{-4}\rho\pi\le|\xi|\le\lambda^{-2}\rho\pi $,  and
$\fg_{\lambda, t, \rho}(\xi) \equiv 0$ for $|\xi|\le
\lambda^{-4}(1-t)\rho\pi$ or $\lambda^{-2}\rho_0\pi\le|\xi|\le\pi$
with $\rho_0$ being a positive constant such that
$\lambda^{-2}\rho<\lambda^{-2}\rho_0<1$, which can be achieved by
using smoothing kernel.
We have the following result.
\begin{proposition}
Let $\lambda>1$, $0<t\le 1$, and $0<\rho\le\lambda^2$. Let
$\falpha_{\lambda,t,\rho}$, $\fbeta_{\lambda,t,\rho}$, and
$\fmu_{\lambda,t,\rho}, \fupsilon_{\lambda,t,\rho}$ be defined as in
\eqref{def:falphafbeta} and \eqref{def:fmask}, respectively. Then
$\falpha_{\lambda,t,\rho}, \fbeta_{\lambda,t,\rho}\in C^\infty(\R)$
and $\fmu_{\lambda,t,\rho}, \fupsilon_{\lambda,t,\rho}\in
C^\infty(\T)$. Moreover,
\[
|\falpha_{\lambda,t,\rho}(\xi)|^2+|\fbeta_{\lambda,t,\rho}(\xi)|^2 = |\falpha_{\lambda,t,\rho}(\lambda^{-2}\xi)|^2, \quad \xi\in\R,
\]
and
\[
\falpha_{\lambda,t,\rho}(\lambda^2\xi) = \fmu_{\lambda,t,\rho}(\xi)
\falpha_{\lambda,t,\rho}(\xi), \quad
\fbeta_{\lambda,t,\rho}(\lambda^2\xi) =
\fupsilon_{\lambda,t,\rho}(\xi) \falpha_{\lambda,t,\rho}(\xi), \quad
\xi\in\R.
\]
\end{proposition}
\begin{proof}Explicitly, we have
\begin{equation}\label{def:falpha}
\begin{aligned}
\falpha_{\lambda,t,\rho}(\xi)&=
\begin{cases}
1& \mbox{if } |\xi|\leq \lambda^{-2}(1-t)\rho\pi;\\
\nu(\frac{-2\lambda^2|\xi|+(2-t)\rho\pi}{t\rho \pi})& \mbox{if } \lambda^{-2}(1-t)\rho\pi<|\xi|\le\lambda^{-2}\rho\pi;\\
0 & \mbox{otherwise}.
\end{cases}
\end{aligned}
\end{equation}
Hence, by the smoothness of $\nu$, we have
$\falpha_{\lambda,t,\rho}\in C^\infty(\R)$.

If $1-t\ge \lambda^{-2}$, by the definition,
$\fbeta_{\lambda,t,\rho}$ can be written as
$\fbeta_{\lambda,t,\rho}(\xi)=\ff_{[a,b],\varepsilon_1,\varepsilon_2}(\xi)
+\ff_{[a,b],\varepsilon_1,\varepsilon_2}(-\xi)$ with
$[a,b]=[\lambda^{-2}(2-t)\rho\pi/2,(2-t)\rho\pi/2]$ and
$\varepsilon_1=\lambda^{-2}t\rho\pi/2$, $\varepsilon_2=\rho\pi/2$;
that is,
\begin{equation}\label{def:fbeta1}
\begin{aligned}
\fbeta_{\lambda,t,\rho}(\xi)&=
\begin{cases}
\nu(\frac{2\lambda^2|\xi|-(2-t)\rho\pi}{t\rho \pi}) & \mbox{if }  \lambda^{-2}(1-t)\rho\pi\le|\xi|< \lambda^{-2}\rho\pi;\\
1 &\mbox{if } \lambda^{-2}\rho\pi\le|\xi|<(1-t)\rho\pi;\\
\nu(\frac{-2|\xi|+(2-t)\rho\pi}{t\rho \pi}) &\mbox{if }  (1-t)\rho\pi\le |\xi|\le \rho\pi;\\
0 & \mbox{otherwise.}
\end{cases}
\end{aligned}
\end{equation}
Again, by the smoothness of $\nu$, we have  $\fbeta_{\lambda,t,\rho}\in C^\infty(\R)$.

 If $0\le1-t<\lambda^{-2}$, then
$\fbeta_{\lambda,t,\rho}$ is given by
\begin{equation}\label{def:fbeta2}
\begin{aligned}
\fbeta_{\lambda,t,\rho}(\xi)&=
\begin{cases}
[(\nu(\frac{-2|\xi|+(2-t)\rho\pi}{t \rho\pi}))^2-(\nu(\frac{-2\lambda^2|\xi|+(2-t)\rho\pi}{t \rho\pi}))^2]^{1/2}&
 \mbox{if }  \lambda^{-2}(1-t)\rho\pi\le|\xi|\le\rho\pi;\\
0 & \mbox{otherwise.}
\end{cases}
\end{aligned}
\end{equation}
Note that $\widetilde\nu(\xi):=(\nu(\frac{-2|\xi|+(2-t)\rho\pi}{t
\rho\pi}))^2-(\nu(\frac{-2\lambda^2|\xi|+(2-t)\rho\pi}{t
\rho\pi}))^2>0$ for all $\xi$ such that
$|\xi|\in(\lambda^{-2}(1-t)\rho\pi,\rho\pi)$. Hence,
$\fbeta_{\lambda,t,\rho}(\xi)=\sqrt{\widetilde\nu(\xi)}$ is
infinitely differentiable for all
$|\xi|\in(\lambda^{-2}(1-t)\rho\pi,\rho\pi)$. For all other $\xi$
such that $|\xi|\notin(\lambda^{-2}(1-t)\rho\pi,\rho\pi)$, all the
derivatives of $\widetilde\nu(\xi)$ vanish. Then, using the Taylor
expansion for $\fbeta_{\lambda,t,\rho}=\sqrt{\widetilde\nu}$, we see
that all the derivatives of $\fbeta_{\lambda,t,\rho}$ vanish for all
$|\xi|\notin(\lambda^{-2}(1-t)\rho\pi,\rho\pi)$. Hence,
$\fbeta_{\lambda,t,\rho}\in C^\infty(\R)$.

Therefore, $\falpha_{\lambda,t,\rho}, \fbeta_{\lambda,t,\rho}\in
C^\infty(\R)$. By the definition of $\fbeta_{\lambda,t,\rho}$, we
have
$|\falpha_{\lambda,t,\rho}(\xi)|^2+|\fbeta_{\lambda,t,\rho}(\xi)|^2
= |\falpha_{\lambda,t,\rho}(\lambda^{-2}\xi)|^2$ for all $\xi\in\R$.

Similar to the cases of $\beta_{\lambda,t,\rho}$, if $1-t\ge
\lambda^{-2}$, then we have
\begin{equation}\label{def:fmask1}
\begin{aligned}
\fmu_{\lambda,t,\rho}(\xi)&=
\begin{cases}
\falpha_{\lambda,t,\rho}(\lambda^2\xi) &  |\xi|\le \lambda^{-4}\rho\pi\\
0 & \lambda^{-4}\rho\pi< |\xi|\le\pi
\end{cases},\\
\fupsilon_{\lambda,t,\rho}(\xi)&=
\begin{cases}
\nu(\frac{2\lambda^4|\xi|-(2-t)\rho\pi}{t\rho \pi}) & \lambda^{-4}(1-t)\rho\pi\le |\xi|\le \lambda^{-2}\rho\pi\\
\fg_{\lambda,t,\rho}(\xi)&  \xi\in[-\pi,\pi]\backslash\supp\fbeta_{\lambda,t,\rho}(\lambda^2\cdot).
\end{cases}
\end{aligned}
\end{equation}
In this case, obviously, $\fmu_{\lambda,t,\rho}\in C^\infty(\T)$.
Note that
$\big[\frac{d^n}{d\xi^n}\nu(\frac{-2\lambda^4|\xi|+(2-t)\rho\pi}{t\rho
\pi})\big]\Big|_{\xi=\pm\lambda^{-2}\rho\pi}=\delta(n)$.  By our
choice of $\fg_{\lambda,t,\rho}$, we see that
$\fupsilon_{\lambda,t,\rho}\in C^\infty(\T)$.

If $0\le 1-t< \lambda^{-2}$, then we have
\begin{equation}\label{def:fmask2}
\begin{aligned}
\fmu_{\lambda,t,\rho}(\xi)&=
\begin{cases}
\frac{\nu(\frac{-2\lambda^4|\xi|+(2-t)\rho\pi}{t \rho\pi})}{\nu(\frac{-2\lambda^2|\xi|+(2-t)\rho\pi}{t\rho \pi})}& |\xi|\le \lambda^{-4}\rho\pi\\
0 & \lambda^{-4}\rho\pi< |\xi|\le\pi
\end{cases}
\end{aligned}
\end{equation}
and
\begin{equation}\label{def:fmask3}
\begin{aligned}
\fupsilon_{\lambda,t,\rho}(\xi)&=
\begin{cases}
\frac{[(\nu(\frac{-2\lambda^2|\xi|+(2-t)\rho\pi}{t \rho\pi}))^2-(\nu(\frac{-2\lambda^4|\xi|+(2-t)\rho\pi}{t \rho\pi}))^2]^{1/2}}{\nu(\frac{-2\lambda^2|\xi|+(2-t)\rho\pi}{t\rho \pi})}&
 \lambda^{-4}(1-t)\rho\pi\le|\xi|\le\lambda^{-2}\rho\pi\\
\fg_{\lambda,t,\rho}(\xi)&  \xi\in[-\pi,\pi]\backslash\supp\fbeta_{\lambda,t,\rho}(\lambda^2\cdot).
\end{cases}
\end{aligned}
\end{equation}
Note that the function
$\nu(\frac{-2\lambda^2|\xi|+(2-t)\rho\pi}{t\rho \pi})$ is strictly
positive for all $\xi$ such that $|\xi|\le \lambda^{-4}\rho\pi$.
Hence, $\frac{1}{\nu(\frac{-2\lambda^2|\xi|+(2-t)\rho\pi}{t\rho
\pi})}$ is infinitely differentiable  for all $\xi$ such that
$|\xi|\le \lambda^{-4}\rho\pi$. Since $\nu$ is $C^\infty$,  the
product of $\nu(\frac{-2\lambda^4|\xi|+(2-t)\rho\pi}{t \rho\pi})$
and $\frac{1}{\nu(\frac{-2\lambda^2|\xi|+(2-t)\rho\pi}{t\rho \pi})}$
is infinitely differentiable  for all $\xi$ such that $|\xi|\le
\lambda^{-4}\rho\pi$.  For $\xi$ such that $
\lambda^{-4}\rho\pi\le|\xi|\le \pi$,  all the derivatives of
$\fmu_{\lambda,t,\rho}(\xi)$ vanish.  Consequently,
$\fmu_{\lambda,t,\rho}\in C^\infty(\T)$. Observe that
$[(\nu(\frac{-2\lambda^2|\xi|+(2-t)\rho\pi}{t
\rho\pi}))^2-(\nu(\frac{-2\lambda^4|\xi|+(2-t)\rho\pi}{t
\rho\pi}))^2]^{1/2}=\nu(\frac{-2\lambda^2|\xi|+(2-t)\rho\pi}{t
\rho\pi}))$ for
$\lambda^{-2}(1-t)\rho\pi\le|\xi|\le\lambda^{-2}\rho\pi$. By similar
arguments, we conclude that $\fupsilon_{\lambda,t,\rho}\in
C^\infty(\T)$.

Therefore, $\fmu_{\lambda,t,\rho}, \fupsilon_{\lambda,t,\rho}\in
C^\infty(\T)$. By their constructions, it is easy to check that
$\falpha_{\lambda,t,\rho}(\lambda^2\xi) = \fmu_{\lambda,t,\rho}(\xi)
\falpha_{\lambda,t,\rho}(\xi)$ and
$\fbeta_{\lambda,t,\rho}(\lambda^2\xi) =
\fupsilon_{\lambda,t,\rho}(\xi) \falpha_{\lambda,t,\rho}(\xi)$ for
$\xi\in\R$.
\end{proof}

The functions $\falpha_{\lambda,t,\rho}$ and
$\fbeta_{\lambda,t,\rho}$ shall be used for the horizontal
direction. We next define $\fgamma_\varepsilon$  for splitting
pieces along the vertical direction. In what follows, $\varepsilon$
shall be fixed as a constant such that $0< \varepsilon\le1/2$.
Define a function $\fgamma_{\varepsilon}$ to be
$\fgamma_{\varepsilon}:=\ff_{[-1/2,1/2],\varepsilon,\varepsilon}$;
that is,
\begin{equation}\label{def:gamma}
\fgamma_\varepsilon( x) =
\begin{cases}
1 & \mbox{if }|x|\le   1/2-\varepsilon;\\
{\nu(\frac{-|x|+1/2}{\varepsilon})} & \mbox{if } 1/2-\varepsilon \le |x|\le 1/2+\varepsilon;\\
0 & \mbox{otherwise}.
\end{cases}
\end{equation}
Then it is easy to check that $\gamma_\varepsilon\in C^\infty(\R)$ and $\sum_{\ell\in\Z}|\fgamma_\varepsilon(\cdot+ \ell)|^2 \equiv 1$.

For $\lambda\in\R$, define $\ell_\lambda:=\lfloor
\lambda-(1/2+\varepsilon)\rfloor+1=\lfloor
\lambda+(1/2-\varepsilon)\rfloor$.
Define the corner pieces
$\gamma_{\lambda,\varepsilon,\varepsilon_0}^\pm$ by
\begin{equation}\label{def:gamma:corner-smooth0}
\begin{aligned}
\fgamma_{\lambda,\varepsilon,\varepsilon_0}^+(\lambda x-\ell_\lambda) &:=
\begin{cases}
\fgamma_{\varepsilon}(\lambda x-\ell_\lambda)& \lambda^{-1}(\ell_{\lambda}-1/2-\varepsilon)\le x\le\lambda^{-1}(\ell_\lambda -1/2+\varepsilon)\\
\nu(1+\frac{\lambda^2}{\varepsilon_0}(1-x))& \lambda^{-1}(\ell_\lambda -1/2+\varepsilon)\le x \le 1+\frac{2\varepsilon_0}{\lambda^2},\\
\end{cases}
\\
\fgamma_{\lambda,\varepsilon,\varepsilon_0}^-(\lambda x+\ell_\lambda) &:=
\begin{cases}
\fgamma_{\varepsilon}(\lambda x+\ell_\lambda)& \lambda^{-1}(-\ell_{\lambda}+1/2-\varepsilon)\le x\le\lambda^{-1}(-\ell_\lambda +1/2+\varepsilon)\\
\nu(1+\frac{\lambda^2}{\varepsilon_0}(1+x))& -1-\frac{2\varepsilon_0}{\lambda^2}\le x\le \lambda^{-1}(-\ell_\lambda +1/2+\varepsilon).\\
\end{cases}
\end{aligned}
\end{equation}
That is,
\begin{equation}\label{def:gamma:corner-smooth}
\begin{aligned}
\fgamma_{\lambda,\varepsilon,\varepsilon_0}^+(x) &=
\begin{cases}
\fgamma_{\varepsilon}(x)& -1/2-\varepsilon\le x\le -1/2+\varepsilon\\
\nu(1+\frac{\lambda^2}{\varepsilon_0}-\frac{\lambda}{\varepsilon_0}(x+\ell_\lambda))& -1/2+\varepsilon\le x \le \lambda(1+2\varepsilon_0/\lambda^2)-\ell_\lambda,\\
\end{cases}
\\
\fgamma_{\lambda,\varepsilon,\varepsilon_0}^-(x) &=
\begin{cases}
\fgamma_{\varepsilon}(x)& 1/2-\varepsilon\le x\le 1/2+\varepsilon\\
\nu(1+\frac{\lambda^2}{\varepsilon_0}+\frac{\lambda}{\varepsilon_0}(x-\ell_\lambda))& -\lambda(1+\varepsilon_0/\lambda^2)+\ell_\lambda\le x \le1/2-\varepsilon.\\
\end{cases}
\end{aligned}
\end{equation}
Note that $\gamma_{\lambda,\varepsilon,\varepsilon_0}^\pm$ are
$C^\infty$ functions. Then, for $\lambda\ge 1$,
\begin{equation}\label{eq:gamma:non-station}
\sum_{\ell=-\ell_\lambda+1}^{\ell_\lambda-1}|\fgamma_\varepsilon(\lambda x+ \ell)|^2+|\fgamma_{\lambda,\varepsilon,\varepsilon_0}^\pm(\lambda x\mp \ell_\lambda)|^2 = 1\quad \forall |x|\le 1
\end{equation}
and
\begin{equation}\label{eq:gamma:station}
\sum_{\ell=-\ell_\lambda}^{\ell_\lambda}|\fgamma_\varepsilon(\lambda x+ \ell)|^2 = 1 \quad \forall |x|\le \frac{\ell_{\lambda}+1/2-\varepsilon}{\lambda}.
\end{equation}

Equations  \eqref{eq:gamma:non-station} and \eqref{eq:gamma:station}
will be used to construct two types of smooth frequency-based affine
shear tight frames. One is non-stationary construction with
$\fphi^j$ changing at different scales and the other is
quasi-stationary construction, in which case the  function $\fphi$
is fixed. We next discuss the details of these two types of
constructions.

\subsection{Non-stationary construction}
We first discuss the non-stationary construction. For such a type of
construction, the shear operations could reach arbitrarily close to
the seamlines when $j$ goes to infinity. The idea of constructing
such a smooth frequency-based  affine shear tight frame in the
non-stationary setting is simple. We first construct a
 frequency-based  affine shear frame from only a few generators and then apply normalization to   such a frame to obtain a tight frame.

More precisely, let $\lambda>1$, $0<t,\rho\le1$, and
$0<\varepsilon\le 1/2$. Let $\da_\lambda, \db_\lambda, \dm_\lambda,
\dn_\lambda, \falpha_{\lambda,t,\rho}$, $\fbeta_{\lambda,t,\rho}$,
and $\fgamma_\varepsilon$,
$\fgamma_{\lambda,\varepsilon,\varepsilon_0}$, $\ell_\lambda$ be
defined as before. Define
\begin{equation}\label{def:fetaZeta}
\begin{aligned}
\feta(\xi_1,\xi_2)&:=\falpha_{\lambda,t,\rho}(\xi_1)\fgamma_{\varepsilon}(\xi_2/\xi_1),\quad (\xi_1,\xi_2)\in\R^2,\\
\fzeta(\xi_1,\xi_2)&:=\fbeta_{\lambda,t,\rho}(\xi_1)\fgamma_{\varepsilon}(\xi_2/\xi_1),\quad
(\xi_1,\xi_2)\in\R^2\backslash\{0\},
\end{aligned}
\end{equation}
as well as the corner peices
\begin{equation}\label{def:fetaZetaCorner}
\begin{aligned}
\feta^{j,\pm\ell_{\lambda^j}}(\xi_1,\xi_2)&:=\falpha_{\lambda,t,\rho}(\xi_1)\fgamma_{\lambda^j,\varepsilon,\varepsilon_0}^\mp(\xi_2/\xi_1),\quad
(\xi_1,\xi_2)\in\R^2,\\
\fzeta^{j,\pm\ell_{\lambda^j}}(\xi_1,\xi_2)&:=\fbeta_{\lambda,t,\rho}(\xi_1)\fgamma_{\lambda^j,\varepsilon,\varepsilon_0}^\mp(\xi_2/\xi_1),\quad
(\xi_1,\xi_2)\in\R^2\backslash\{0\}.
\end{aligned}
\end{equation}
For $\xi=0$, $\fzeta(0):=0$ and
$\fzeta^{j,\pm\ell_{\lambda^j}}(0):=0$. Since the support of
$\fbeta_{\lambda,t,\rho}$ is away from the origin, we have $\fzeta,
\fzeta^{j,\pm\ell_{\lambda^j}}\in C^\infty(\R^2)$. Let
$\fphi(\xi):=\falpha_{\lambda,t,\rho}(\xi_1)\falpha_{\lambda,t,\rho}(\xi_2)$,
$\xi\in\R^2$. Then, $\fphi$ is also a function in $C^\infty(\R^2)$.
 For a nonnegative integer $J_0$, define
\begin{equation}\label{def:fTheta}
\begin{aligned}
\fTheta^{J_0}(\xi)
:=&|\fphi(\dn_\lambda^{J_0}\xi)|^2+\sum_{j=J_0}^\infty\sum_{\ell=-\ell_{\lambda^j}}^{\ell_{\lambda^j}}(|\fzeta^{j,\ell}(\sh_\ell\db_\lambda^j\xi)|^2+|\fzeta^{j,\ell}(\sh_\ell\db_\lambda^j\exchg\xi)|^2)
\end{aligned}
\end{equation}
for $\xi\in\R^2$, where $\feta^{j,\ell}:=\feta$ and
$\fzeta^{j,\ell}:=\fzeta$ for $|\ell|<\ell_{\lambda^j}$.
We have the following result concerning the function $\fTheta^{J_0}$.
\begin{proposition}\label{prop:fTheta}
Let $\lambda>1, 0<\varepsilon\le1/2, 0<t, \rho\le1$. Let $J_0$ be a
nonnegative integer and $\fTheta^{J_0}$ be defined as in
\eqref{def:fTheta}. Choose $\varepsilon_0>0$ such that
$\varepsilon_0< \frac12\lambda^{J_0-1}$.
 Then $\fTheta^{J_0}$ has the following properties.
\begin{itemize}
\item[\rm{(i)}]$\fTheta^{J_0}\in C^\infty(\R^2)$, $\fTheta^{J_0}=\fTheta^{J_0}(\exchg\cdot)$, and $0<\fTheta^{J_0} \le 2$.
\item[\rm{(ii)}] $
\fTheta^{J_0}(\xi)=\fTheta^{J_0}(\exchg\xi)= 1 \quad\forall
\xi\in\cup_{j=J_0+1}^\infty\cup_{\ell=-\ell_{\lambda^j}+2}^{\ell_{\lambda^j}-2}\supp\fzeta^{j,\ell}(\sh_\ell\db_\lambda^j\cdot)$.
\end{itemize}
\end{proposition}
\begin{proof} Since the generators $\fphi, \fzeta^{j,\ell}$ are compactly supported functions in $C^\infty(\R^2)$ and for any bounded open set $E\subseteq\R^2$,
$\fTheta^{J_0}(\xi)$ is the  summation of finitely many terms from
$\fphi, \fzeta^{j,\ell}$ for all $\xi\in E$,  the function
$\fTheta^{J_0}$ is thus also a function in $C^\infty(\R^2)$. By its
definition, it is obvious that $\fTheta^{J_0}(\exchg\cdot) =
\fTheta^{J_0}$.

For simplicity of presentation, we denote $\chi_{\{|\xi_2/\xi_1|\le
1\}}:=\chi_{\{\xi\in\R^2: |\xi_2/\xi_1|\le1\}}$ and similar notation
applies for others. For $\xi\in\{\xi\in\R^2:
\max\{|\xi_1|,|\xi_2|\}<\lambda^{2J_0}\rho\pi\}$, by the property of
$\fgamma_\varepsilon$ as in \eqref{eq:gamma:non-station}, we have
\[
\begin{aligned}
\fTheta^{J_0}(\xi) &\ge |\fphi(\dn_{\lambda}^{J_0}\xi)|^2 +
\sum_{\ell=-\ell_{\lambda^{J_0}}}^{\ell_{\lambda^{J_0}}}(|\fzeta^{J_0,\ell}(\sh_\ell\db_\lambda^{J_0}\xi)|^2
+|\fzeta^{J_0,\ell}(\sh_\ell\db_\lambda^{J_0}\exchg\xi)|^2)
\\&
\ge
|\falpha_{\lambda,t,\rho}(\lambda^{-2J_0}\xi_1)\falpha_{\lambda,t,\rho}(\lambda^{-2J_0}\xi_2)|^2+
|\fbeta_{\lambda,t,\rho}(\lambda^{-2J_0}\xi_1)|^2\chi_{\{
|\xi_2/\xi_1|\le 1\}}(\xi)\\&\quad+
|\fbeta_{\lambda,t,\rho}(\lambda^{-2J_0}\xi_2)|^2\chi_{\{
|\xi_2/\xi_1|> 1\}}(\xi)
\\&>0,
\end{aligned}
\]
and for
 $\xi\in\{\xi\in\R^2:
\max\{|\xi_1|,|\xi_2|\}>\lambda^{2J_0-2}\rho\pi\}$, we have
\[
\begin{aligned}
\fTheta^{J_0}(\xi) &=\sum_{j=J_0}^\infty
\sum_{\ell=-\ell_{\lambda^{j}}}^{\ell_{\lambda^{j}}}(|\fzeta^{j,\ell}(\sh_\ell\db_\lambda^{j}\xi)|^2+|\fzeta^{j,\ell}(\sh_\ell\db_\lambda^{j}\exchg\xi)|^2)
\\&\ge\sum_{j=J_0}^\infty\Big[
\fbeta_{\lambda,t,\rho}(\lambda^{-2j}\xi_1)\chi_{\{|\xi_2/\xi_1|\le1\}}(\xi)
+\fbeta_{\lambda,t,\rho}(\lambda^{-2j}\xi_2)\chi_{\{|\xi_2/\xi_1|>1\}}(\xi)\Big]
\\&>0.
\end{aligned}
\]
Consequently, $\fTheta^{J_0}>0$.

We next show that $\fTheta^{J_0}\le 2$. By the property of
$\fgamma_\varepsilon$ as in \eqref{eq:gamma:non-station}, we have
\[
\begin{aligned}
&\sum_{\ell=-\ell_{\lambda^{j}}}^{\ell_{\lambda^{j}}}|\feta(\sh_\ell\db_\lambda^j\xi)|^2\chi_{\{|\xi_2/\xi_1|\le1\}}(\xi)
=|\falpha_{\lambda,t,\rho}(\lambda^{-2j}\xi_1)|^2\chi_{\{|\xi_2/\xi_1|\le1\}}(\xi),\quad
\xi\neq0
\end{aligned}
\]
and similarly,
\[
\begin{aligned}
&\sum_{\ell=-\ell_{\lambda^j}}^{\ell_{\lambda^j}}|\fzeta(\sh_\ell\db_\lambda^j\xi)|^2\chi_{\{|\xi_2/\xi_1|\le1\}}(\xi)=
|\fbeta_{\lambda,t,\rho}(\lambda^{-2j}\xi_1)|^2\chi_{\{|\xi_2/\xi_1|\le1\}}(\xi),\quad\xi\neq0.
\end{aligned}
\]
Hence,
\[
\begin{aligned}
&\sum_{\ell=-\ell_{\lambda^j}}^{\ell_{\lambda^j}}(|\feta(\sh_\ell\db_\lambda^j\xi)|^2+|\fzeta(\sh_\ell\db_\lambda^j\xi)|^2)\chi_{\{|\xi_2/\xi_1|\le1\}}(\xi)
\\&=(|\falpha_{\lambda,t,\rho}(\lambda^{-2j}\xi_1)|^2
+|\fbeta_{\lambda,t,\rho}(\lambda^{-2j}\xi_1)|^2)\chi_{\{|\xi_2/\xi_1|\le1\}}(\xi)
\\&=
|\falpha_{\lambda,t,\rho}(\lambda^{-2(j+1)}\xi_1)|^2\chi_{\{|\xi_2/\xi_1|\le1\}}(\xi)
\\&=\sum_{\ell=-\ell_{\lambda^{j+1}}}^{\ell_{\lambda^{j+1}}}|\feta^{j+1,\ell}(\sh_\ell\db_\lambda^{j+1}\xi)|^2\chi_{\{|\xi_2/\xi_1|\le1\}}(\xi),\quad\xi\neq0.
\end{aligned}
\]
Therefore, we have
\[
\begin{aligned}
&\lim_{J\rightarrow\infty}
\Big(\sum_{\ell=-\ell_{\lambda^{J_0}}}^{\ell_{\lambda^{J_0}}}|\feta^{J_0,\ell}(\sh_\ell\db_\lambda^{J_0}\xi)|^2+\sum_{j=J_0}^{J-1}
\sum_{\ell=-\ell_{\lambda^j}}^{\ell_{\lambda^j}}
|\fzeta^{j,\ell}(\sh_\ell\db_\lambda^j\xi)|^2\Big)\chi_{\{|\xi_2/\xi_1|\le1\}}(\xi)
\\&=\lim_{J\rightarrow\infty}|\falpha_{\lambda,t,\rho}(\lambda^{-2J}\xi)|^2\chi_{\{|\xi_2/\xi_1|\le1\}}(\xi)
=\chi_{\{|\xi_2/\xi_1|\le1\}}(\xi), \quad\xi\neq0.
\end{aligned}
\]
Now, we define
\[
\begin{aligned}
\widetilde\fTheta^{J_0}(\xi):=&\sum_{\ell=-\ell_{\lambda^{J_0}}}^{
\ell_{\lambda^{J_0}}}(|\feta^{J_0,\ell}(\sh_\ell\db_\lambda^{{J_0}}\xi)|^2
+|\feta^{J_0,\ell}(\sh_\ell\db_\lambda^{{J_0}}\exchg\xi)|^2)
\\&
+\sum_{j={J_0}}^\infty\sum_{\ell=-\ell_{\lambda^j}}^{
\ell_{\lambda^j}}(|\fzeta^{j,\ell}(\sh_\ell\db_\lambda^j\xi)|^2+|\fzeta^{j,\ell}(\sh_\ell\db_\lambda^j\exchg\xi)|^2),
\quad \xi\neq0.
\end{aligned}
\]
Then,
\[
\begin{aligned}
\widetilde\fTheta^{J_0}(\xi)&=\sum_{\ell=-\ell_{\lambda^{J_0}}}^{
\ell_{\lambda^{J_0}}}(|\feta^{J_0,\ell}(\sh_\ell\db_\lambda^{{J_0}}\xi)|^2
+|\feta^{J_0,\ell}(\sh_\ell\db_\lambda^{{J_0}}\exchg\xi)|^2)(\chi_{\{|\xi_2/\xi_1|\le1\}}(\xi)+\chi_{\{|\xi_2/\xi_1|>1\}}(\xi))
\\&
+\sum_{j={J_0}}^\infty\sum_{\ell=-\ell_{\lambda^j}}^{ \ell_{\lambda^j}}(|\fzeta^{j,\ell}(\sh_\ell\db_\lambda^j\xi)|^2+|\fzeta^{j,\ell}(\sh_\ell\db_\lambda^j\exchg\xi)|^2)
(\chi_{\{|\xi_2/\xi_1|\le1\}}(\xi)+\chi_{\{|\xi_2/\xi_1|>1\}}(\xi))\\
&=
1+\Big(|\feta^{J_0,\pm\ell_{\lambda^{J_0}}}(\sh_{\pm\ell_{\lambda^{J_0}}}\db_\lambda^{{J_0}}\xi)|^2+
\sum_{j={J_0}}^\infty|\fzeta^{j,\pm\ell_{\lambda^j}}(\sh_{\pm\ell_{\lambda^{j}}}\db_\lambda^j\xi)|^2\Big)\chi_{\{|\xi_2/\xi_1|>1\}}(\xi)
\\&
+\Big(|\feta^{J_0,\pm\ell_{\lambda^{J_0}}}(\sh_{\pm\ell_{\lambda^{J_0}}}\db_\lambda^{{J_0}}\exchg\xi)|^2+
\sum_{j={J_0}}^\infty|\fzeta^{j,\pm\ell_{\lambda^j}}(\sh_{\pm\ell_{\lambda^{j}}}\db_\lambda^j\exchg\xi)|^2\Big)\chi_{\{|\xi_2/\xi_1|\le1\}}(\xi)
\\&=1+I(\xi)+I(\exchg\xi)-\sqrt{I(\xi)I(\exchg\xi)},\quad\xi\neq0,
\end{aligned}
\]
where
\[
I(\xi):=\Big(|\feta^{J_0,\pm\ell_{\lambda^{J_0}}}(\sh_{\pm\ell_{\lambda^{J_0}}}\db_\lambda^{{J_0}}\xi)|^2+
\sum_{j={J_0}}^\infty|\fzeta^{j,\pm\ell_{\lambda^j}}(\sh_{\pm\ell_{\lambda^{j}}}\db_\lambda^j\xi)|^2\Big)\chi_{\{|\xi_2/\xi_1|\ge1\}}(\xi),\quad\xi\neq0.
\]
 By the
construction of $\falpha_{\lambda,t,\rho}$ and
$\fbeta_{\lambda,t,\rho}$, we have $I\le 1$. Therefore, $1\le
\widetilde\fTheta^{J_0}\le 2$. Observe that $\fTheta^{J_0}(\xi)\le
\widetilde\fTheta^{J_0}(\xi)\le 2$ for $\xi\neq0$ and
$\fTheta^{J_0}(0)=1$, we conclude that item (i) holds.

We next show that item (ii) holds. We have
\[
\fTheta^{J_0}(\xi) =
\widetilde\fTheta^{J_0}(\xi)+|\fphi(\dn_\lambda^{J_0}\xi)|^2-
\sum_{\ell=-\ell_{\lambda^{J_0}}}^{
\ell_{\lambda^{J_0}}}(|\feta^{J_0,\ell}(\sh_\ell\db_\lambda^{{J_0}}\xi)|^2.
\]
Note that $\supp\fphi(\dn_\lambda^{J_0}\cdot)$ is inside the support
of $\sum_{\ell=-\ell_{\lambda^{J_0}}}^{
\ell_{\lambda^{J_0}}}(|\feta^{J_0,\ell}(\sh_\ell\db_\lambda^{{J_0}}\cdot)|^2$.
Hence, for $\xi$ outside the support of
$\sum_{\ell=-\ell_{\lambda^{J_0}}}^{
\ell_{\lambda^{J_0}}}(|\feta^{J_0,\ell}(\sh_\ell\db_\lambda^{{J_0}}\cdot)|^2$,
we have $\fTheta^{J_0}=\widetilde\fTheta^{J_0}$. By that
$\widetilde\fTheta^{J_0}=1+I+I(\exchg\cdot)-\sqrt{I\cdot
I(\exchg\cdot)}$, we hence only need to check the overlapping coming
from $I$ and $I(\exchg\cdot)$. In fact,
at scale $j$, the seamline element on the  horizontal cone  with
respect to $\ell=-\ell_{\lambda^j}$ has part of the piece
overlapping with the other cone. By the support of
$\fgamma^+_{\lambda^j,\varepsilon,\varepsilon_0}$, for this seamline
element, we have its support satisfying $\xi_2/\xi_1\le
1+\frac{2\varepsilon_0}{\lambda^{2j}}$. Moreover, by the support of
$\fbeta_{\lambda,t,\rho}$, this seamline element can only affect
other elements in the vertical cone with respect to scales $j_0=
j-1, j, j+1$. Now, the support of the vertical cone element
corresponding to scale $j_0$ and $\ell=- \ell_{\lambda^{j_0}}+s$
with $s$ being a nonnegative integer satisfying
\[
\lambda^{j_0}\xi_1/\xi_2- \ell_{\lambda^{j_0}}+s\le \frac12+\varepsilon,
\]
which implies $\xi_1/\xi_2\le \frac{1/2+\varepsilon+ \ell_{\lambda^{j_0}}-s}{\lambda^{j_0}}$. Consequently, the seamline element on the horizontal cone affecting
the elements in the vertical cone at scale $j_0$ means
\[
1+\frac{2\varepsilon_0}{\lambda^{2j}}\ge \frac{\lambda^{j_0}}{1/2+\varepsilon+ \ell_{\lambda^{j_0}}-s},
\]
which implies
\[
\begin{aligned}
s&\le\frac12+\varepsilon+\ell_{\lambda^{j_0}}-\frac{\lambda^{2j+j_0}}{\lambda^{2j}+2\varepsilon_0}
\\&\le\frac12+\varepsilon+(\lambda^{j_0}+\frac12-\varepsilon)-\frac{\lambda^{2j+j_0}}{\lambda^{2j}+2\varepsilon_0}
\\&\le1+\frac{2\varepsilon_0\lambda^{j_0}}{\lambda^{2j}+2\varepsilon_0}
\le1+\frac{2\varepsilon_0\lambda^{j+1}}{\lambda^{2j}+2\varepsilon_0}
\\&\le 1+\frac{2\varepsilon_0}{\lambda^{j-1}}
\le 1+\frac{2\varepsilon_0}{\lambda^{J_0-1}}<2
\end{aligned}
\]
since $\varepsilon_0<\frac{\lambda^{J_0-1}}{2}$. Hence, by that  $s$
is a nonnegative integer, we deduce that $s$ is either 0 or 1. By
symmetry, same result holds for seamline elements on vertical cone
affecting the horizontal cone. Therefore, we have
$\fTheta^{J_0}(\xi)\equiv\fTheta^{J_0}(\exchg\xi)\equiv 1$ for $\xi$
in the support of those
$\fzeta^{j,\ell}(\sh_\ell\db_\lambda^j\cdot)$ with $|\ell|<
\ell_{\lambda^j}-1$ and $j\ge J_0+1$. That is, item (ii) holds.
\end{proof}
The function $\fTheta^{J_0}$ is for the normalization of the frame
generated by $\fzeta^{j,\ell}$. We next define a function
$\fGamma^j$, which shall be used  for frequency splitting. The
function $\fGamma^j$ is defined to be
\begin{equation}\label{def:fGamma-non-station}
\begin{aligned}
\fGamma^j(\xi):=&\sum_{\ell=-\ell_{\lambda^j}+1}^{\ell_{\lambda^j}-1}
\Big[|\fgamma_\varepsilon(\lambda^j\xi_2/\xi_1+\ell)|^2+|\fgamma_\varepsilon(\lambda^j\xi_1/\xi_2+\ell)|^2\Big]
\\&+|\fgamma^\mp_{\lambda^j,\varepsilon,\varepsilon_0}(\lambda^j\xi_2/\xi_1\pm\ell_{\lambda^j})|^2
+|\fgamma^\mp_{\lambda^j,\varepsilon,\varepsilon_0}(\lambda^j\xi_1/\xi_2\pm\ell_{\lambda^j})|^2,\quad
\xi\neq0.
\end{aligned}
\end{equation}
 For $\fGamma^j$,
we  have the following result.
\begin{proposition}\label{prop:fGamma-non-station}
Let $\fGamma^j$ be defined as in \eqref{def:fGamma-non-station}.
Then $\fGamma^j$ is a function in $C^\infty(\R^2\backslash\{0\})$
and has the following properties.
\begin{itemize}
\item[{\rm(i)}]
  $1\le\fGamma_{j}(\xi)\le 2$, $\fGamma^j(\exchg\xi) = \fGamma^j(\xi)$, and  $\fGamma^j(t\xi)=\fGamma^j(\xi)$ for all $t\neq0$ and $\xi\neq0$.
\item[{\rm(ii)}] $\fGamma^j$ satisfies
\begin{equation}\label{eq:fGamma==1-non-station}
\fGamma^{j}(\xi)\equiv1, \; \xi\in\left\{\xi\in\R^2\backslash\{0\}:
\max\{|\xi_2/\xi_1|,|\xi_1/\xi_2|\}\le
\frac{\lambda^{2j}}{\lambda^{2j}+2\varepsilon_0}\right\}.
\end{equation}
\end{itemize}
\end{proposition}
\begin{proof}
Since $\fgamma_\varepsilon$ is a compactly supported function in
$C^\infty(\R)$ and the function $f(\xi):=\xi_2/\xi_1$ or
$\xi_1/\xi_2$ is infinitely differentiable for all
$\xi=(\xi_1,\xi_2)\in\R^2$ such that both $\xi_1\neq0$ and
$\xi_2\neq 0$, we see that by Taylor expansion $\fGamma^j$ is
infinitely differentiable for $\xi\in\R^2$ such that both
$\xi_1\neq0$ and $\xi_2\neq0$. For a fixed $\xi_1\neq0$,  we have
\[
\Gamma^j(\xi)= \sum_{\ell=-\ell_{\lambda^j}+1}^{\ell_{\lambda^j}-1}
|\fgamma_\varepsilon(\lambda^j\xi_2/\xi_1+\ell)|^2+|\fgamma^\mp_{\lambda^j,\varepsilon,\varepsilon_0}(\lambda^j\xi_2/\xi_1\pm\ell_{\lambda^j})|^2
\]
for $|\xi_2|$ small enough in view of the supports of
$\fgamma_\varepsilon$ and
$\fgamma_{\lambda^j,\varepsilon,\varepsilon_0}$, which implies that
$\fGamma^j(\xi)$ is infinitely differentiable at $(\xi_1,0)$ with
$\xi_1\neq0$. Similarly, we have $\fGamma^j(\xi)$ is infinitely
differentiable at $(0,\xi_2)$ with $\xi_2\neq0$.
 Hence, we have $\fGamma^j\in C^\infty(\R^2\backslash\{0\})$. By
its definition as in \eqref{def:fGamma-non-station},
$\fGamma^j(\exchg\cdot) =\fGamma^j$ and $\fGamma^j(t\xi)
=\fGamma^j(\xi)$ for $t\neq0$ and $\xi\neq0$.

By the property of $\fgamma_\varepsilon,
\fgamma_{\lambda,\varepsilon,\varepsilon_0}$ as in
\eqref{eq:gamma:non-station}, it is easily seen that
$1\le\fGamma^j(\xi)\le 2$ for $\xi\neq0$. Now to see that
\eqref{eq:fGamma==1-non-station} holds, we  notice that the seamline
element on the  horizontal cone with respect to
$\ell=-\ell_{\lambda^j}$ has part of the piece overlapping with the
other cone. By the support of
$\fgamma_{\lambda,\varepsilon,\varepsilon_0}^+$, for this seamline
element, we have $\xi_2/\xi_1\le
1+\frac{2\varepsilon_0}{\lambda^{2j}}$. Hence, those elements on the
vertical cone with support satisfying $|\xi_1/\xi_2|\le
\frac{\lambda^{2j}}{\lambda^{2j}+2\varepsilon_0}$ is not affected by
that seamline elements. By symmetry, same result holds for seamline
elements on vertical cone affecting the horizontal cone. Therefore,
\eqref{eq:fGamma==1-non-station} holds.
\end{proof}

Since $0< \fTheta^{J_0}\le 2$, we can take the square root of
$\fTheta^{J_0}$, which is still a smooth function. Moreover,
$1/\sqrt{\fTheta^{J_0}}$ is also a smooth function. Define
$\fphi^{J_0}:=\frac{\fphi}{\sqrt{\fTheta(\dm_\lambda^{J_0}\cdot)}}$
%
%
%
and
\begin{equation}\label{def:fG}
\begin{aligned}
\fomega_{\lambda,t,\rho}^{j}(\dn_\lambda^j\xi) :=&\frac{\left(
\sum_{\ell=-\ell_{\lambda^j}}^{\ell_{\lambda^j}}
(|\fzeta^{j,\ell}(\sh_\ell\db_\lambda^j\xi)|^2
+|\fzeta^{j,\ell}(\sh_\ell\db_\lambda^j\exchg\xi)|^2)\right)^{1/2} }
{\sqrt{\fTheta^{J_0}(\xi)}},\;\; j\ge J_0.
\end{aligned}
\end{equation}
Define $\fphi^{j+1}$ to be
\begin{equation}\label{def:fphij}
\fphi^{j+1}(\dn_{\lambda}^{j+1}\xi):=\left(|\fphi^{j}(\dn_{\lambda}^j\xi)|^2+|\fomega_{\lambda,t,\rho}^j(\dn_\lambda^j\xi)|^2\right)^{1/2}.
\end{equation}
Now, we split the function $\fomega_{\lambda,t,\rho}^j$ as follows.
Let $\dd_\lambda:=\diag(1,\lambda)$. For $\xi\neq0$, define
\begin{equation}\label{def:fpsijl}
\fpsi^{j,\ell}(\xi):=\fomega_{\lambda,t,\rho}^j(\dd_\lambda^{-j}\sh_{-\ell}\xi)\frac{\fgamma_\varepsilon(\xi_2/\xi_1)}{\fGamma^j((\sh_\ell\db_\lambda^j)^{-1}\xi)},
\;\ell=- \ell_{\lambda^j}+1,\ldots, \ell_{\lambda^j}-1,
\end{equation}
and
\begin{equation}\label{def:fpsijl-corner}
\fpsi^{j,\pm\ell_{\lambda^j}}(\xi)
:=\fomega_{\lambda,t,\rho}^j(\dd_\lambda^{-j}\sh_{\mp\ell_{\lambda^j}}\xi)
\frac{\fgamma^{\mp}_{\lambda^j,\varepsilon,\varepsilon_0}(\xi_2/\xi_1)}
{\fGamma^j((\sh_{\pm\ell_{\lambda^j}}\db_\lambda^j)^{-1}\xi)}.
\end{equation}
For $\xi=0$, we define $\fpsi^{j,\ell}(0):=0$. Since the support of
$\fomega^j_{\lambda,t,\rho}$ is away from the origin and in view of
the properties of $\fGamma^j$, we deduce that $\fpsi^{j,\ell}$ are
functions in $C^\infty(\R^2)$. Let
\begin{equation}
\label{def:fPsi-j-non-station}
\fPsi_j:=\{\fpsi^{j,\ell}(\sh_\ell\cdot): \ell=-
\ell_{\lambda_j},\ldots, \ell_{\lambda_j}\}
\end{equation}
with $\fpsi^{j,\ell}$ being given as in \eqref{def:fpsijl} and
\eqref{def:fpsijl-corner}. The  \emph{frequency-based  affine shear
system} $\FASS_J(\fphi^J;\{\fPsi_j\}_{j=J}^\infty)$ is then defined
as follows:
\begin{equation}
\begin{aligned}\label{def:DFAS-cone-shearlet-smooth-non-station}
\FASS_J(\fphi^J;\{\fPsi_j\}_{j=J}^\infty) :=
& \{\fphi^J_{\dn^J;0,\vk}:\vk\in\Z^2\}\cup\{\fh_{\db_\lambda^j;0,\vk},\fh_{\db_\lambda^j\exchg;0,\vk}: \vk\in\Z^2, \fh\in\fPsi_j\}_{j=J}^\infty.\\
\end{aligned}
\end{equation}
Note that, explicitly, we have,
\[
\begin{aligned}
\FASS_J(\fphi^J;\{\fPsi_j\}_{j=J}^\infty) = & \{\fphi^J_{\dn^J;0,\vk}:\vk\in\Z^2\}
\cup\{\fpsi^{j,\ell}_{\sh_\ell\db_\lambda^j;0,\vk},\fpsi^{j,\ell}_{\sh_\ell\db_\lambda^j\exchg;0,\vk}: \vk\in\Z^2, \ell=- \ell_{\lambda^j},\ldots, \ell_{\lambda^j}\}_{j=J}^\infty.
\end{aligned}
\]
With the property of $\fTheta^{J_0}$ as in item (ii) of
Proposition~\ref{prop:fTheta}, we can show that the system defined
as \eqref{def:DFAS-cone-shearlet-smooth-non-station} can have shear
structure for elements inside  each cone. Moreover, with the scale
$j$ going to infinity, the shear operation could reach the seamline
arbitrarily close. Indeed, we have the following result.
\begin{theorem}\label{thm:non-station}
Let $\lambda>1$, $0<\varepsilon\le1/2$, $0<t,\rho\le1$ such that
$1/\rho-1/2 -\varepsilon>0$. Let $J_0$ be a nonnegative integer.
Choose $\varepsilon_0>0$ such that
\[
\varepsilon_0<\min\left\{ \frac{\lambda^{J_0-1}}{2},
\lambda^{2J_0}(\frac{\lambda^2}{2\rho}-1/2),
(1/\rho-1/2-\varepsilon)\lambda^{J_0} \right\}.
\]
Then the system $\FASS_J(\fphi^J;\{\fPsi_j\}_{j=J}^\infty)$ defined
as in \eqref{def:DFAS-cone-shearlet-smooth-non-station} with
$\fphi^j$ and $\fPsi_j$ being given as in \eqref{def:fphij} and
\eqref{def:fPsi-j-non-station}, respectively, is a  frequency-based
affine shear tight frame for $L_2(\R^2)$ for all $J\ge J_0$. All
elements in $\FASS_J(\fphi^J;\{\fPsi_j\}_{j=J}^\infty)$ are
compactly supported functions in $C^\infty(\R^2)$. Moreover,
\[
\{\fzeta(\sh_\ell\cdot):
|\ell|<\ell_{\lambda^j}-1\}\subseteq\fPsi_j,\;\; j\ge J_0+1,
\]
and
\[
\{\fzeta_{\sh_\ell\db_\lambda^j;0,\vk},
\fzeta_{\sh_\ell\db_\lambda^j\exchg;0,\vk}: j\ge J,\vk\in\Z^2,
|\ell|<
\ell_{\lambda^j}-1\}\subseteq\FASS_J(\fphi^J;\{\fPsi_j\}_{j=J}^\infty),\;\;
J\ge J_0+1.
\]
\end{theorem}
\begin{proof}
By the property of $\fTheta^{J_0}$  as in
Proposition~\ref{prop:fTheta}, we have
$\fomega_{\lambda,t,\rho}^j(\dn_\lambda^j\xi) =
\fbeta_{\lambda,t,\rho}(\lambda^{-2j}\xi_1)$ for
$\xi\in\supp\fzeta^{j,\ell}(\sh_\ell\db_\lambda^j\cdot)$ with
$|\ell|<\ell_{\lambda^j}-1$, $j\ge J_0+1$, and
$\fomega_{\lambda,t,\rho}^j(\xi) =
\fbeta_{\lambda,t,\rho}(\lambda^{-2j}\xi_2)$ for
$\xi\in\supp\fzeta^{j,\ell}(\sh_\ell\db_\lambda^j\exchg\cdot)$ with
$|\ell|<\ell_{\lambda^j}-1$ and $j\ge J_0+1$. Hence, it is easily
seen that for $j\ge J_0+1$,
\[
\fpsi^{j,\ell}(\xi) =
\fbeta_{\lambda,t,\rho}(\xi_1)\fgamma_\varepsilon(\xi_2/\xi_1)=\fzeta(\xi),
\quad |\ell|<\ell_{\lambda^j}-1.
\]
Define $\widetilde\fPsi_j:=\{\fzeta(\sh_\ell\cdot):
\ell=-\ell_{\lambda^j}+2,\ldots,\ell_{\lambda^j}-2\}\cup\{\fpsi^{j,\ell}(\sh_\ell\cdot):
|\ell|=\ell_{\lambda^j}-1,\ell_{\lambda^j}\}$. Then for $j\ge
J_0+1$, $\{\fzeta(\sh_\ell\cdot): |\ell|<
\ell_{\lambda^j}-1\}\subseteq\widetilde\fPsi_j=\fPsi_j$ and
$\FASS(\fphi^J;\{\fPsi_j\}_{j=J}^\infty)
=\FASS(\fphi^J;\{\widetilde\fPsi_j\}_{j=J}^\infty)$.

By  our construction, \eqref{eq:SimpleCond2} and
\eqref{eq:SimpleCond3} hold. Moreover, all generators are
nonnegative. Noting that
$\supp\falpha_{\lambda,t,\rho}=[-\lambda^{-2}\rho\pi,
\lambda^{-2}\rho\pi]$, $\supp\fbeta_{\lambda,t,\rho}=[-\rho\pi,
-\lambda^{-2}\rho\pi]\cup[\lambda^{-2}\rho\pi, \rho\pi]$, and
$\supp\gamma_\varepsilon =[-1/2-\varepsilon, 1/2+\varepsilon]$,
together with $\rho\le1$ and $0<\varepsilon\le1/2$, we see that
$\supp\fpsi^{j,\ell}\subseteq[-\rho\pi,\rho\pi]^2\subseteq[-\pi,\pi]^2$
for $|\ell|\le\ell_{\lambda^j}-1$. Hence, we have
$\fpsi^{j,\ell}(\xi)\fpsi^{j,\ell}(\xi+2\pi\vk)=0$, a.e.,
$\xi\in\R^2$ and $\vk\in\Z^2\backslash\{0\}$ for
$|\ell|\le\ell_{\lambda^j}-1$. For $\fpsi^{j,-\ell_{\lambda^j}}$ we
have
\[
\supp \fpsi^{j,-\ell_{\lambda^j}} \subseteq\{\xi\in\R^2: \xi_1\in[-\rho\pi,\rho\pi],-1/2-\varepsilon\le\xi_2/\xi_1\le\lambda^j(1+2\varepsilon_0/\lambda^{2j})-\ell_{\lambda^j}\}.
\]
Since $\varepsilon_0\le\lambda^{J_0}(2/\rho-1-2\varepsilon)$, we have,
\[
(\lambda^j(1+2\varepsilon_0/\lambda^{2j})-\ell_{\lambda^j}+1/2+\varepsilon)
\le
(\lambda^j(1+2\varepsilon_0/\lambda^{2j})-(\lambda^j+1/2-\varepsilon)+1+1/2+\varepsilon)
\le \frac{2\varepsilon_0}{\lambda^j}+1+2\varepsilon
\le 2/\rho.
\]
Hence, we  conclude that
$\fpsi^{j,\pm\ell_{\lambda^j}}(\xi)\fpsi^{j,\pm\ell_{\lambda^j}}(\xi+2\pi\vk)=0$,
a.e., $\xi\in\R^2$ and  $\vk\in\Z^2\backslash\{0\}$.

By the definition of $\fphi^j$ and that $\varepsilon_0\le \lambda^{2J_0}(\frac{\lambda^2}{2\rho}-1/2)$, we have
\[
\supp\fphi^j\subseteq[-\lambda^{-2}\rho(1+2\varepsilon_0/\lambda^{2j})\pi,\lambda^{-2}\rho(1+2\varepsilon_0/\lambda^{2j})\pi]^2\subseteq[-\pi,\pi]^2.
\]
 Hence,
we conclude that $\fphi^j(\xi)\fphi^j(\xi+2\pi\vk)=0$ for all
$\vk\in\Z^2\backslash\{0\}$ and for almost every $\xi\in\R^2$.
Therefore, \eqref{eq:SimpleCond1} holds. By the result of
Corollary~\ref{cor:positive},
$\FASS_J(\fphi^J;\{\fPsi_j\}_{j=J}^\infty)$ is a frequency-based
affine shear tight frame for $L_2(\R^2)$ for all $J\ge J_0$.  Since
all generators are compactly supported functions in
$C^\infty(\R^2)$,  all elements in
$\FASS_J(\fphi^J;\{\fPsi_j\}_{j=J}^\infty)$ are compactly supported
functions in $C^\infty(\R^2)$.
\end{proof}
From Theorem~\ref{thm:non-station}  we see that
\[
\fzeta(\sh_{-\ell_{\lambda^j}+2}\db_\lambda^j\xi)=\fbeta_{\lambda,t,\rho}(\lambda^{-2j}\xi_1)\gamma_\varepsilon(\lambda^j\xi_2/\xi_1-\ell_{\lambda^j}+2)
,\quad \xi\in\R^2
\]
 has support satisfying $\xi_2/\xi_1\le \frac{\ell_{\lambda^j}-2+1/2+\varepsilon}{\lambda^j}\rightarrow 1$ as $j\rightarrow \infty$.
In other words, the shear operation reaches arbitrarily close to the seamlines $\{\xi\in\R^2: |\xi_2/\xi_1|=\pm1\}$.

\subsection{Quasi-stationary construction}
 Let us next discuss the quasi-stationary construction. The idea is to use the tensor product of   functions in 1D to obtain rectangular bands for different scales, and then
 a frequency splitting using $\fgamma_\varepsilon$ is applied to produce generators with respect to different shears. More precisely,
let $\lambda>1$ and $0<t,\rho\le1$. Consider
$\fphi(\xi):=\falpha_{\lambda,t,\rho}(\xi_1)\falpha_{\lambda,t,\rho}(\xi_2)$,
$\xi=(\xi_1,\xi_2)\in\R^2$  and define
\begin{equation}\label{def:fomega}
\fomega_{\lambda,t,\rho}(\xi):=\sqrt{|\fphi(\lambda^{-2}\xi)|^2-|\fphi(\xi)|^2},\quad\xi\in\R^2.
\end{equation}
Then $\fomega_{\lambda,t,\rho}\in C^\infty(\R^2)$. In fact, it is
easy to show that if $\fphi(\xi_0)=0$ or $1$, then all the
derivatives of $\fphi$ vanish at $\xi_0$. Now if
$\widetilde\fomega_{\lambda,t,\rho}(\xi):=|\fphi(\lambda^{-2}\xi)|^2-|\fphi(\xi)|^2$
is not vanishing for $\xi=\xi_0$, then it is trivial to see that
$\fomega_{\lambda,t,\rho}
=\sqrt{\widetilde\fomega_{\lambda,t,\rho}}$ is infinitely
differentiable at $\xi=\xi_0$. If
$\widetilde\fomega_{\lambda,t,\rho}(\xi)=0$ at $\xi=\xi_0$, then we
must have $\fphi(\xi_0)=\fphi(\lambda^{-2}\xi_0)=0$ or
$\fphi(\xi_0)=\fphi(\lambda^{-2}\xi_0)=1$. Then, all the derivatives
of $\widetilde\fomega_{\lambda,t,\rho}$ vanish at $\xi_0$. By the
Taylor expansion, we see that $\fomega_{\lambda,t,\rho} =
\sqrt{\widetilde\fomega_{\lambda,t,\rho}}$ must be infinitely
differentiable at $\xi_0$ with all its derivatives at $\xi_0$ being
zero. Therefore, $\fomega_{\lambda,t,\rho}\in C^\infty(\R^2)$.

In view of the construction of $\fphi$, the refinable structure is
clear. We have $\fphi(\lambda^2\xi) = \fa(\xi) \fphi(\xi)$,
$\xi\in\R^2$ with $\fa=\fmu_{\lambda,t,\rho}\otimes
\fmu_{\lambda,t,\rho}$ being the tensor product of the 1D  mask
$\fmu_{\lambda,t,\rho}$ given as in \eqref{def:fmask}. Moreover, we
have $\fomega(\lambda^2\xi) = \fb(\xi)\fphi(\xi)$ with $\fb\in
C^\infty(\T)$ being given by
$\fb(\xi)=(\fg(\xi)-|\fa(\xi)|^2)^{1/2}$ for any smooth function
$\fg\in C^\infty(\T^2)$ such that $\fg\equiv 1$ on the support of
$\fphi$.

\begin{proposition}\label{prop:fGamma}
Define
\begin{equation}\label{def:fGamma}
\begin{aligned}
\fGamma_{j}(\xi)
:=&\sum_{\ell=-\ell_{\lambda^{j}}}^{\ell_{\lambda^{j}}}(|\fgamma_\varepsilon(\lambda^j\xi_2/\xi_1+\ell)|^2
+|\fgamma_\varepsilon(\lambda^j\xi_1/\xi_2+\ell)|^2),\;\; j\in \N_0.
\end{aligned}
\end{equation}
Then    $\fGamma_j\in C^\infty(\R^2\backslash\{0\})$ having the
following properties.
\begin{itemize}
\item[\rm{(i)}]$0< \fGamma_{j}(\xi)\le 2$, $\fGamma_j(\exchg\xi) = \fGamma_j(\xi)$, and $\fGamma_j(t\xi)=\fGamma_j(\xi)$ for all $t\neq0$ and $\xi\neq0$.
\item[\rm{(ii)}] $\fGamma_j$ satisfies
\begin{equation}\label{eq:fGamma==1}
\fGamma_{j}(\xi)\equiv1, \; \xi\in\left\{\xi\in\R^2:
\max\{|\xi_2/\xi_1|,|\xi_1/\xi_2|\}\le
\frac{\lambda^j}{\ell_{\lambda^j}+1/2+\varepsilon}\right\}.
\end{equation}
\end{itemize}
\end{proposition}
\begin{proof}
The proof for $\fGamma_j\in C^\infty(\R^2\backslash\{0\})$ is
similar to that for $\fGamma^j$ in
Proposition~\ref{prop:fGamma-non-station}. By its definition in
\eqref{def:fGamma}, $\fGamma_j(\exchg\cdot) =\fGamma_j$ and
$\fGamma_j(t\cdot) =\fGamma_j$ for $t\neq0$.

By the property of $\gamma_\varepsilon$ as in
\eqref{eq:gamma:station}, it is easily seen that $0<\fGamma_j\le 2$.
Now to see that \eqref{eq:fGamma==1} holds, we  notice that the
seamline element on the  horizontal cone  with respect to
$\ell=-\ell_{\lambda^j}$ has part of the piece overlapping with the
other cone. By the support of $\fgamma_\varepsilon$, for this
seamline element, we have
\[
\lambda^j\xi_2/\xi_1- \ell_{\lambda^j}\le \frac12+\varepsilon,
\]
which implies $\xi_2/\xi_1\le \frac{1/2+\varepsilon+
\ell_{\lambda^j}}{\lambda^j}$. Hence, it only affects elements in
other cone with support satisfying  $\xi_1/\xi_2>
\frac{\lambda^j}{1/2+\varepsilon+ \ell_{\lambda^j}}$. By symmetry,
the same result holds for seamline elements on vertical cone
affecting the horizontal cone. Therefore, \eqref{eq:fGamma==1}
holds.
\end{proof}
%

 Since $0<\fGamma_j\le 2$ and $\fGamma_j$ is in
$C^\infty(\R^2\backslash\{0\})$, we have that $\sqrt{\fGamma_j}$ is
infinitely differentiable for all $\xi\in\R^2\backslash\{0\}$. Let
$\da_\lambda, \db_\lambda, \dm_\lambda, \dn_\lambda, \dd_\lambda$
with $\lambda>1$ be defined as before. Let
$\fPsi_j:=\{\fpsi^{j,\ell}(\sh_\ell\cdot):
\ell=-\ell_{\lambda^j},\ldots,\ell_{\lambda^j}\}$ with
\begin{equation}\label{def:fpsij-station}
\begin{aligned}\fpsi^{j,\ell}(\xi) &:=
\fomega_{\lambda,t,\rho}(\dd_\lambda^{-j}\sh_{-\ell}\xi)\frac{\fgamma_{\varepsilon}(\xi_2/\xi_1)}{\sqrt{\fGamma_j((\sh_\ell\db_\lambda^j)^{-1}\cdot)}}
\\&=\fomega_{\lambda,t,\rho}(\xi_1,
\lambda^{-j}(-\xi_1\ell+\xi_2))\frac{\fgamma_{\varepsilon}(\xi_2/\xi_1)}{\sqrt{\fGamma_j((\sh_\ell\db_\lambda^j)^{-1}\cdot)}},\quad
\xi\in\R^2\backslash\{0\},
\end{aligned}
\end{equation}
and $\fpsi^{j,\ell}(0):=0$,
 which gives
\[
\begin{aligned}
\fpsi^{j,\ell}(\sh_\ell\db_\lambda^j\xi)
=\fomega_{\lambda,t,\rho}(\dn_\lambda^j\xi)\frac{\fgamma_{\varepsilon}(\lambda^j\xi_2/\xi_1+\ell)}{\sqrt{\fGamma_j(\xi)}}.
\end{aligned}
\]
By the properties of $\fGamma_j$ and that the support of
$\fomega_{\lambda,t,\rho}$ is away from the origin, we see that
$\fpsi^{j,\ell}$ are functions in $C^\infty(\R^2)$. We can define
the following system:
\begin{equation}
\begin{aligned}\label{def:DFAS-cone-shearlet-smooth-station}
\FASS_J(\fphi;\{\fPsi_j\}_{j=J}^\infty) := &
\{\fphi_{\dn_\lambda^J;0,\vk}:\vk\in\Z^2\}\cup\{\fh_{\db_\lambda^j;0,\vk},\fh_{\db_\lambda^j\exchg;0,\vk}:
\vk\in\Z^2,\fh\in\fPsi_j \}_{j=J}^\infty
\\=&\{\fphi_{\dn_\lambda^J;0,\vk}:\vk\in\Z^2\}\cup\{\fpsi^{j,\ell}_{\sh_\ell\db_\lambda^j;0,\vk},\fpsi^{j,\ell}_{\sh_\ell\db_\lambda^j\exchg;0,\vk}:
 \vk\in\Z^2, \ell=-\ell_{\lambda^j},\ldots,\ell_{\lambda^j} \}_{j=J}^\infty.
\end{aligned}
\end{equation}
At first glance, such a system does not have  shear structure at all due to that the function $\fomega_{\lambda,t,\rho}$ is not shear-invariant. However, we next shall show that such a system  do have certain affine and shear structure in the sense
 that a sub-system of this system is from shear and dilation of one single generator.
\begin{theorem}\label{thm:cone-shear-station}
 Let $\lambda>1$ and $0<t,\rho\le 1$. Let $\FASS_J(\fphi;\{\fPsi_j\}_{j=J}^\infty)$ be defined as in \eqref{def:DFAS-cone-shearlet-smooth-station}
 with $\fphi=\falpha_{\lambda,t,\rho}\otimes\falpha_{\lambda,t,\rho}$ and $\fpsi^{j,\ell}$ being given by \eqref{def:fpsij-station}. Then
$\FASS_J(\fphi;\{\fPsi_j\}_{j=J}^\infty)$ is a frequency-based
affine shear tight frame for $L_2(\R^2)$ for all $J\ge 0$. All
elements in $\FASS_J(\fphi;\{\fPsi_j\}_{j=J}^\infty)$ are compactly
supported functions in $C^\infty(\R^2)$. Moreover, we have
\[
\{\fzeta(\sh_\ell\cdot):\ell=- r_j,\ldots, r_j\}\subseteq \fPsi_j,
\;\; j\ge J,
\]
where $r_j:=\lfloor\lambda^{j-2}(1-t)\rho-(1/2+\varepsilon)\rfloor$
and $\fzeta(\xi): =
\fbeta_{\lambda,t,\rho}(\xi_1)\fgamma_\varepsilon(\xi_2/\xi_1)$,
$\xi\in\R^2$. In other words,
\[
\{\fzeta_{\sh_\ell\db_\lambda^j;0,\vk},
\fzeta_{\sh_\ell\db_\lambda^j\exchg;0,\vk}: \vk\in\Z^2,
\ell=-r_j,\ldots,r_j\}_{j=J}^
\infty\subseteq\FASS_J(\fphi;\{\fPsi_j\}_{j=J}^\infty).
\]
\end{theorem}
\begin{proof} By our construction, we have
\[
\begin{aligned}
&|\fphi(\dn_\lambda^j\xi)|^2+\sum_{\ell=-\ell_{\lambda^j}}^{\ell_{\lambda^j}}
\Big[|\fpsi^{j,\ell}(\sh_\ell\db_{\lambda}^j\xi)|^2+
|\fpsi^{j,\ell}(\sh_\ell\db_{\lambda}^j\exchg\xi)|^2\Big]
\\=&
|\fphi(\dn_\lambda^j\xi)|^2+\frac{|\fomega_{\lambda,t,\rho}(\dn^j\xi)|^2}{\fGamma_j(\xi)}\sum_{\ell=-\ell_{\lambda^j}}^{\ell_{\lambda^j}}
\Big[|\fgamma_\varepsilon(\lambda^j\xi_2/\xi_1+\ell)|^2+
|\fgamma_\varepsilon(\lambda^j\xi_1/\xi_2+\ell)|^2\Big]
\\=&
|\fphi(\dn_\lambda^j\xi)|^2+|\fomega_{\lambda,t,\rho}(\dn^j\xi)|^2=|\fphi(\dn^{j+1}\xi)|^2,\quad
\xi\in\R^2.
\end{aligned}
\]
Hence, \eqref{eq:SimpleCond2} holds. By the definition of $\fphi$,
\eqref{eq:SimpleCond3} also holds. Note that all generators
$\fpsi^{j,\ell}$ are nonnegative and are defined in
$[-\rho\pi,\rho\pi]^2$ with $\rho\le1$. Hence,
\eqref{eq:SimpleCond1} is true. Now, by
Corollary~\ref{cor:positive}, we conclude that
$\FASS_J(\fphi;\{\fPsi_j\}_{j=J}^\infty)$ is a frequency-based
affine shear tight frame for $L_2(\R^2)$ for all $J\ge 0$. Since all
generators $\fphi, \fpsi^{j,\ell}$ are compactly supported functions
in $C^\infty(\R^2)$, all elements in
$\FASS_J(\fphi;\{\fPsi_j\}_{j=J}^\infty)$ are compactly supported
functions in $C^\infty(\R^2)$.

By the definition of $\fomega_{\lambda,t,\rho}$, it is easy to see that
\[
|\fomega_{\lambda,t,\rho}(\xi_1,\xi_2)|^2 =|\falpha_{\lambda,t,\rho}(\xi_1)\fbeta_{\lambda,t,\rho}(\xi_2)|^2+ |\fbeta_{\lambda,t,\rho}(\xi_1)\falpha_{\lambda,t,\rho}(\xi_2)|^2+|\fbeta_{\lambda,t,\rho}(\xi_1)\fbeta_{\lambda,t,\rho}(\xi_2)|^2.
\]
And for $|\xi_2|\le \lambda^{-2}(1-t)\rho\pi$, we have
\[
\fomega_{\lambda,t,\rho}(\xi_1,\xi_2) = \fbeta_{\lambda,t,\rho}(\xi_1)\falpha_{\lambda,t,\rho}(\xi_2)=\fbeta_{\lambda,t,\rho}(\xi_1).
\]
Consequently, if the support $\{\xi=(\xi_1,\xi_2)\in\R^2:
\xi\in\supp\fpsi^{j,\ell}_{\sh_\ell\db_\lambda^j;0,\vk}\}$ satisfies
$|\xi_2|\le \lambda^{2j-2}(1-t)\rho\pi$, then we have
\[
\begin{aligned}
\fpsi^{j,\ell}_{\sh_\ell\db_\lambda^j;0,\vk}(\xi)
&= \lambda^{-3j/2}\fomega_{\lambda,t,\rho}(\lambda^{-2j}\xi)\fgamma_\varepsilon(\lambda^j\xi_2/\xi_1+\ell)e^{-i\vk\cdot\sh_\ell\db_\lambda^j\xi}\\
&= \lambda^{-3j/2}\fbeta_{\lambda,t,\rho}(\lambda^{-2j}\xi_1)\fgamma_\varepsilon(\lambda^j\xi_2/\xi_1+\ell)e^{-i\vk\cdot\sh_\ell\db_\lambda^j\xi}\\
&=\fzeta_{\sh_\ell\db_\lambda^j;0,\vk}(\xi)
\end{aligned}
\]
with
$\fzeta(\xi):=\fbeta_{\lambda,t,\rho}(\xi_1)\fgamma_\varepsilon(\xi_2/\xi_1)$,
 $\xi\in\R^2$. Now let us find the range of $\ell$ such that the
support constrain holds.

At scale $j$, we have
\[
\supp \fomega_{\lambda,t,\rho}(\lambda^{-2j}\cdot)\subseteq [-\lambda^{2j}\rho\pi,\lambda^{2j}\rho\pi]^2\backslash[-\lambda^{2j-2}(1-t)\rho\pi,\lambda^{2j-2}(1-t)\rho\pi]^2.
\]
Then, the support constrain means that at scale $j$, one needs $|\xi_2/\xi_1|\le \lambda^{-2}(1-t)\rho$. Hence, the support of
 $\fgamma_{\varepsilon}(\lambda^j\xi_2/\xi_1+\ell)$ must satisfy
\[
-\lambda^{-2}(1-t)\rho\le -\lambda^{-j}(1/2+\varepsilon+\ell)\le \xi_2/\xi_1\le \lambda^{-j}(1/2+\varepsilon-\ell)\le \lambda^{-2}(1-t)\rho.
\]
Consequently, we obtain
\[
-\lambda^{j-2}(1-t)\rho+(1/2+\varepsilon)\le\ell\le \lambda^{j-2}(1-t)\rho-(1/2+\varepsilon).
\]
That is, $|\ell|\le \lambda^{j-2}(1-t)\rho-(1/2+\varepsilon)$. In
summary, letting $r_j:=\lfloor
\lambda^{j-2}(1-t)\rho-(1/2+\varepsilon)\rfloor$, we have
\[
\{\fzeta(\sh_\ell\cdot):\ell=- r_j,\ldots, r_j\}\subseteq \fPsi_j, \; j\ge J,
\]
and
\[
\{\fzeta_{\sh_\ell\db_\lambda^j;0,\vk},
\fzeta_{\sh_\ell\db_\lambda^j\exchg;0,\vk}: j\ge J,\vk\in\Z^2,
\ell=-r_j,\ldots,r_j\}\subseteq\FASS_J(\fphi;\{\fPsi_j\}_{j=J}^\infty).
\]
\end{proof}
Note that when $\ell=-r_j$, The support
$\fzeta(\sh_\ell\db_\lambda^j\xi)=\fbeta_{\lambda,t,\rho}(\lambda^{-2j}\xi_1)\gamma_\varepsilon(\lambda^j\xi_2/\xi_1-r_j)$ satisfies
\[
\begin{aligned}
\xi_2/\xi_1&\le \lambda^{-j}(r_j+1/2+\varepsilon)\\&
\le \lambda^{-j}(\lfloor\lambda^{j-2}(1-t)\rho-1/2-\varepsilon\rfloor+1/2+\varepsilon)
\\&
\le \lambda^{-2}(1-t)\rho.
\end{aligned}
\]
Hence, by the symmetry property of $\fGamma_j$, we see that the
shear operation generates a subsystem of
$\FASS_J(\fphi;\{\fPsi_j\}_{j=0}^\infty)$  inside the cone area
$\{\xi\in\R^2: \max\{|\xi_2/\xi_1|,|\xi_1/\xi_2|\}\le
\lambda^{-2}(1-t)\rho\}$.

\subsection{Connections to other directional mutliscale represention systems}
\label{subsec:otherDMRS}
In this subsection, we shall discuss the connections of our affine
shear tight frames to those shearlet systems in
\cite{GKL06,GuoLabate:2007} or shearlet-like systems in
\cite{GuoLabate:2012}.

Define corner pieces
\begin{equation}\label{def:gamma:corner}
\begin{aligned}
\fgamma_\lambda^+(x) &:=
\begin{cases}
\fgamma_\varepsilon(x)& \mbox{if }-1/2-\varepsilon\le x\le -1/2+\varepsilon;\\
1&\mbox{if } -1/2+\varepsilon\le x \le \lambda-\ell_\lambda;\\
0 & \mbox{otherwise},
\end{cases}
\\
\fgamma_\lambda^-(x) &:=
\begin{cases}
\fgamma_\varepsilon(x)&\mbox{if } 1/2-\varepsilon\le x\le 1/2+\varepsilon;\\
1&\mbox{if } -\lambda+\ell_\lambda\le x\le 1/2-\varepsilon;\\
0 & \mbox{otherwise}.
\end{cases}
\end{aligned}
\end{equation}
These are the corner pieces that  shall be used  to achieve
tightness of the system or for gluing two seamline elements together
 smoothly. Let
$\{\falpha_{\lambda,t,\rho},\fbeta_{\lambda,t,\rho},
\fgamma_{\varepsilon},\fgamma_\lambda^\pm\}$ be  defined as before.
Similar to the half pieces for the system  generated by the
characteristic functions as in \eqref{def:shannonPhiPsi},  we define
$\fpsi, \fpsi^{j,\pm\ell_{\lambda^j}}$ by
\[
\fpsi(\xi):=\fbeta_{\lambda,t,\rho}(\xi_1)\fgamma_\varepsilon(\xi_2/\xi_1),\quad
\fpsi^{j,\pm
\ell_{\lambda^j}}(\xi):=\fbeta_{\lambda,t,\rho}(\xi_1)\gamma_{\lambda^j}^\mp(\xi_2/\xi_1),\quad\xi\neq0
\]
and $\fpsi(0):=0$. The scaling function $\fphi$ is defined to be
\begin{equation}\label{def:fphi:cone-cut}
\fphi:=\fphi^h+\fphi^v
\end{equation}
with $\fphi^h(\xi)
=\falpha_{\lambda,t,\rho}(\xi_1)\chi_{\{\xi\in\R^2: |\xi_2/\xi_1|\le
1\}}(\xi)$, $\xi\in\R^2$ and $\fphi^v =
\fphi^h(\exchg\cdot)=\falpha_{\lambda,t,\rho}(\xi_2)\chi_{\{\xi\in\R^2:|\xi_1/\xi_2|\le
1\}}(\xi)$, $\xi\in\R^2$. Now define
\begin{equation}\label{def:fPsij:cone-cut}
\fPsi_j:=\{\fpsi(\sh_\ell\cdot): \ell=-\ell_{\lambda^j}+1,\ldots,\ell_{\lambda^j}-1\}\cup\{\fpsi^{j,\ell}(\sh_\ell\cdot): \ell=\pm \ell_{\lambda^j}\}.
\end{equation}
Note that $\fpsi$ is smooth while the corner pieces $\fpsi^{j,\pm
\ell_{\lambda^j}}$ are not smooth.

We have the following result.
\begin{corollary}\label{cor:DFAS-cone-shearlets}
Let $\da_\lambda, \db_\lambda, \dm_\lambda, \dn_\lambda, \sh_\ell,
\exchg$ be defined as before with $\lambda>1$. Let $0<t,\rho\le1$
and $0< \varepsilon\le 1/2$. Then the system $\FASS_J(\fphi;
\{\fPsi_j\}_{j=J}^\infty)$ defined as in \eqref{def:FAS_J} with
$\fphi, \fPsi_j$ being given by \eqref{def:fphi:cone-cut},
\eqref{def:fPsij:cone-cut}, respectively, is a frequency-based
affine shear tight frame for $L_2(\R^2)$ for all $J\ge 0$.
\end{corollary}
\begin{proof}
By the definition of $\fgamma_\varepsilon$ and
$\fgamma_\lambda^\pm$, for a fixed $j\ge0$, it is easy to show that
\[
\sum_{\ell=-\ell_{\lambda^j}+1}^{\ell_{\lambda^j}-1}
|\fgamma_\varepsilon(\lambda^j\xi_2/\xi_1+\ell)|^2+|\fgamma_{\lambda^j}^\pm(\lambda^j\xi_2/\xi_1\mp\ell_{\lambda^j})|^2
=\chi_{\{|\xi_2/\xi_1|\le1\}}(\xi), \quad \xi\neq0.
\]
Hence, we have
\[
\begin{aligned}
|\fphi^h(\dn_\lambda^j\xi)|^2+\sum_{\fh\in\fPsi_j}|\fh(\db_\lambda^j\xi)|^2
& =
\left(|\falpha_{\lambda,t,\rho}(\lambda^{-2j}\xi_1)|^2+\fbeta_{\lambda,t,\rho}(\lambda^{-2j}\xi_1)|^2\right)\chi_{\{|\xi_2/\xi_1|\le1\}}(\xi)
\\&= |\falpha_{\lambda,t,\rho}(\lambda^{-2j-2}\xi_1)|^2\chi_{\{|\xi_2/\xi_1|\le1\}}(\xi)
\\&=|\fphi^h(\dn_\lambda^{j+1}\xi)|^2,\quad\xi\in\R^2.
\end{aligned}
\]
Similarly, we have
\[
\begin{aligned}
|\fphi^v(\dn_\lambda^j\xi)|^2+\sum_{\fh\in\fPsi_j}|\fh(\db_\lambda^j\exchg\xi)|^2
& =
\left(|\falpha_{\lambda,t,\rho}(\lambda^{-2j}\xi_2)|^2+\fbeta_{\lambda,t,\rho}(\lambda^{-2j}\xi_2)|^2\right)\chi_{\{|\xi_1/\xi_2|\le1\}}(\xi)
\\&= |\falpha_{\lambda,t,\rho}(\lambda^{-2j-2}\xi_1)|^2\chi_{\{|\xi_2/\xi_1|\le1\}}(\xi)
\\&=|\fphi^v(\dn_\lambda^{j+1}\xi)|^2,\quad\xi\in\R^2.
\end{aligned}
\]
Consequently, we have
\[
\begin{aligned}
|\fphi(\dn_\lambda^j\xi)|^2+\sum_{\fh\in\fPsi_j}(|\fh(\db_\lambda^j\xi)|^2+|\fh(\db_\lambda^j\exchg\xi)|^2)
=|\fphi(\dn_\lambda^{j+1}\xi)|^2,\quad a.e.\;\xi\in\R^2.
\end{aligned}
\]
Hence \eqref{eq:SimpleCond2} holds.

Moreover, we have $\fh(\xi)\fh(\xi+2\pi \vk)=0$ for all
$\fh\in\{\fphi\}\cup\{\fPsi_j\}_{j=0}^\infty$ and
$\vk\in\Z^2\backslash\{0\}$. In fact, if $\vk=(k_1,k_2)\in\Z^2$ with
$k_1\neq 0$, then $\fh( \xi)\fh(\xi+2\pi\vk)=0$ due to that
$\falpha_{\lambda,t,\rho},\fbeta_{\lambda,t,\rho}$ are supported on
$[-\rho\pi,\rho\pi]$ with $\rho\le1$. If $k_1= 0$ but $k_2\neq 0$,
then by that $\fgamma_\varepsilon((\xi_2+2\pi
k_2)/\xi_1)\fgamma_\varepsilon(\xi_2/\xi_1)=\fgamma_\varepsilon(\xi_2/\xi_1+2\pi
k_2/\xi_1)\fgamma_\varepsilon(\xi_2/\xi_1)=0$ for
$\xi_1\in[-\rho\pi,\rho\pi]$, we have $\fh(\xi)\fh(\xi+2\pi \vk)=0$
as well. Hence, \eqref{eq:SimpleCond1} is satisfied. Obviously,
\eqref{eq:SimpleCond3} is true by our construction of $\fphi$.

Therefore, by Corollary~\ref{cor:positive},  $\FASS_J(\fphi;
\{\fPsi_j\}_{j=J}^\infty)$ defined as in \eqref{def:FAS} with
$\fphi, \fPsi_j$ being given by \eqref{def:fphi:cone-cut},
\eqref{def:fPsij:cone-cut}, respectively, is a frequency-based
affine shear tight frame for $L_2(\R^2)$ for all $J\ge0$.
\end{proof}

Now, it is easy to show that the cone-adapted shearlet system
constructed in \cite{GuoLabate:2007} is indeed the initial system of
a sequence of frequency-based  affine shear tight frames.  In fact,
let $\lambda = 2$, and $\da_h:=\da_\lambda$,
$\da_v:=\exchg\da_h\exchg$. Let $\varphi, \psi,
\psi^{j,\pm\ell_{\lambda^j}}$ be the inverse Fourier transforms of
$\fphi, \fpsi^h, \fpsi^{j,\pm\ell_{\lambda^j}}$; that is,
$\fphi=\ft\varphi$, $\fpsi=\ft\psi^h$, and
$\fpsi^{j,\pm\ell_{\lambda^j}}=\ft\psi^{j,\pm\ell_{\lambda^j}}$. Let
$\psi^v:=\psi^h(\exchg\cdot)$. It is easy to show that
\[
\ft(2^{3j/2}\psi^h(\sh^\ell\da_h^j\cdot-\vk)) = 2^{-3j/2}\fpsi(\sh_{-\ell}\db_\lambda^j\cdot)e^{-i\sh^{-\ell}\vk\cdot \db_\lambda^j\cdot}=\fpsi_{\sh_{-\ell}\db_\lambda^j;0,\vk}.
\]
Similarly, we have
$\ft(2^{3j/2}\psi^v(\sh_\ell\da_v^j\cdot-\vk))=\fpsi_{\sh_{-\ell}\db_\lambda^j\exchg;
0,\exchg\vk}$. Noting that $\exchg \Z^2=\Z^2$ and the symmetry of
the range of $\ell$ for each scale $j$, we see that the cone-adapted
shearlet system in \eqref{def:cone-adapted-shearlet} with modified
seamline elements is the  frequency-based  affine shear tight frame
$\FASS(\fphi; \{\fPsi_j\}_{j=0}^\infty)$ defined as in
\eqref{def:FAS} with  $\fphi, \fPsi_j$ being given by
\eqref{def:fphi:cone-cut}, \eqref{def:fPsij:cone-cut}, and $\lambda
= 2$. Moreover, it is the initial system of the sequence of
frequency-based affine shear tight frames $\FASS_J(\fphi;
\{\fPsi_j\}_{j=J}^\infty), J\in \N_0$ defined as in
\eqref{def:FAS_J} with   $\fphi, \fPsi_j$ being given by
\eqref{def:fphi:cone-cut}, \eqref{def:fPsij:cone-cut}, respectively.

For the smooth shearlet-like systems constructed in
\cite{GuoLabate:2012}, it is also a special case of the follow
system. Note that $\fgamma_\lambda^+,\fgamma_\lambda^-$ satisfy
\begin{equation}\label{eq:gamma:derivative:corner}
\Big[\frac{d^n}{dx^n}\fgamma_\lambda^\pm(\lambda x\mp
\ell_\lambda)\Big]\Big|_{x=\pm1} = \delta(n)\quad \forall n\in\N_0,
\end{equation}
which  guarantees the smooth gluing of two corner pieces.

Let $\fPsi_j:=\{\fpsi^{j,\ell}(\sh_\ell\cdot):
\ell=-\ell_{\lambda^j},\ldots,\ell_{\lambda^j}\}$ with elements
strictly inside the cone, i.e., $\ell=-\ell_{\lambda^j}+1,\ldots,
\ell_{\lambda^j}-1$  being given by
\begin{equation}\label{def:fphijlGlueInside}
\fpsi^{j,\ell}(\xi) :=
\fomega_{\lambda,t,\rho}(\dd_\lambda^{-j}\sh_{-\ell}\xi)\fgamma_{\varepsilon}(\xi_2/\xi_1)
=\fomega_{\lambda,t,\rho}(\xi_1,
\lambda^{-j}(-\xi_1\ell+\xi_2))\fgamma_{\varepsilon}(\xi_2/\xi_1),\quad
\xi\in\R^2,
\end{equation}
which gives
\[
\begin{aligned}
\fpsi^{j,\ell}(\sh_\ell\db_\lambda^j\xi)=\fomega_{\lambda,t,\rho}(\lambda^{-2j}\xi)\fgamma_{\varepsilon}(\lambda^j\xi_2/\xi_1+\ell);
\end{aligned}
\]
and  those elements on the seamlines, i.e., for  $\ell = \pm \ell_{\lambda^j}$ and $j\ge1$, being given by gluing two pieces along the seamline on two cones together,
\[
\fpsi^{j,\pm\ell_{\lambda^j}}(\sh_{\pm \ell_{\lambda^j}}\db_\lambda^j/2\xi):=
\begin{cases}
\fomega_{\lambda,t,\rho} (\lambda^{-2j}\xi)\fgamma_{\lambda^j}^\mp(\lambda^j\xi_2/\xi_1\pm \ell_{\lambda^j}) & |\xi_2/\xi_1|\le1\\
\fomega_{\lambda,t,\rho}
(\lambda^{-2j}\xi)\fgamma_{\lambda^j}^\mp(\lambda^j\xi_1/\xi_2\pm
\ell_{\lambda^j}) & |\xi_2/\xi_1|\ge 1.
\end{cases}
\]
For $j=0$,
\[
\fpsi^{0,\pm1}(\sh_{\pm 1}\xi):=
\begin{cases}
\fomega_{\lambda,t,\rho} (\xi)\fgamma_{\varepsilon}(\xi_2/\xi_1\pm1) & |\xi_2/\xi_1|\le1\\
\fomega_{\lambda,t,\rho} (\xi)\fgamma_{\varepsilon}(\xi_1/\xi_2\pm
1) & |\xi_2/\xi_1|\ge 1.
\end{cases}
\]
Let $\db_\lambda^{j,\ell}:=\db_\lambda^j$ for $j\ge1$ and $\ell<\ell_{\lambda^j}$,  $\db_\lambda^{j,\pm\ell_{\lambda^j}}:=\db_\lambda^j/2$ for $j\ge1$,
and for $j=0$, $\db_\lambda^{j,\ell}:=I_2$.
Then, we can define the following system
\begin{equation}
\begin{aligned}\label{def:DFAS-cone-shearlet-smooth-station2}
\FASS(\fphi;\{\fPsi_j\}_{j=0}^\infty) = & \{\fphi_{0,\vk}:\vk\in\Z^2\}\cup\\
&\cup\{\fh_{\db_\lambda^{j,\ell};0,\vk},\fh_{\db_\lambda^{j,\ell}\exchg;0,\vk}: \vk\in\Z^2, \fh\in\fPsi_j\backslash\{\fpsi^{j,\ell}(\sh_\ell\cdot): \ell=\pm\ell_{\lambda^j}\}\}_{j=0}^\infty\\
&\cup\{\fh_{\db_\lambda^{j,\ell};0,\vk},\fh_{\db_\lambda^{j,\ell};0,\vk}(\exchg\cdot):
\vk\in\Z^2, \fh\in \{\fpsi^{j,\ell}(\sh_\ell\cdot):
\ell=\pm\ell_{\lambda^j}\}\}_{j=0}^\infty.
\end{aligned}
\end{equation}
\begin{corollary}\label{cor:cone-shear-station2}
$\FASS(\fphi;\{\fPsi_j\}_{j=0}^\infty)$ defined as in
\eqref{def:DFAS-cone-shearlet-smooth-station2} is a frequency-based
affine shear tight frame for $L_2(\R^2)$ and all elements in
$\FASS(\fphi;\{\fPsi_j\}_{j=0}^\infty)$ are compactly supported
functions in $C^\infty(\R^2)$.
\end{corollary}
\begin{proof} By our construction, we have
\[
\begin{aligned}
|\fphi|^2&+\sum_{j=0}^\infty\sum_{\ell=-\ell_{\lambda^j}+1}^{\ell_{\lambda^j}-1}
[|\fpsi^{j,\ell}(\sh_\ell\db_{\lambda}^j\cdot)|^2+
|\fpsi^{j,\ell}(\sh_\ell\db_{\lambda}^j\exchg\cdot)|^2]
\\&+\sum_{j=0}^\infty\sum_{\ell=\pm\ell_{\lambda^j}}|\fpsi^{j,\ell}(\sh_{\ell}\db_\lambda^j/2\cdot)|^2+|\fpsi^{j,\ell}(\sh_{\ell}\db_\lambda^j\exchg/2\cdot)|^2
=1,\quad a.e.\;\xi\in\R^2.
\end{aligned}
\]
Moreover, all generators $\fpsi^{j,\ell}$ are nonnegative and
defined in $[-\pi,\pi]^2$. Note that dilation matrices of the
seamline generators $\fpsi^{j,\pm\ell_{\lambda^j}}$ are
$\db_\lambda^j/2$. A simple adaptation of the proof of
Theorem~\ref{thm:main} gives that
$\FASS(\fphi;\{\fPsi_j\}_{j=0}^\infty)$ is a frequency-based tight
frame. By the definition of $\fgamma, \fgamma^\pm_\lambda$ in
\eqref{def:gamma}, \eqref{def:gamma:corner}, $\fpsi^{j,\ell}$ are
smooth. Consequently, all elements in
$\FASS(\fphi;\{\fPsi_j\}_{j=0}^\infty)$ are smooth.
\end{proof}

When $\lambda =2, t = 1-\lambda^{-2}$ and $\rho=1$,
$\FASS(\fphi;\{\fPsi_j\}_{j=0}^\infty)$ defined as in
\eqref{def:DFAS-cone-shearlet-smooth-station2} is essentially the
system defined as in \cite{GuoLabate:2012}. For their construction,
the shear operations can reach only up to slope (in absolute value)
$\lambda^{-4}=1/16$. Here, our construction is more general and in
our construction, the shear can reach up to slope
$\lambda^{-2}(1-t)\rho$ with any $0< t,\rho\le 1$. Comparing to our
quasi-stationary construction, the gluing procedure is somewhat
unnatural since one can  see that a different dilation matrix
$\db_\lambda^j/2$ needs to be applied to the gluing elements  at
this scale $j$ while  all other generators use the dilation matrix
$\db_\lambda^j$.
On the other hand, all atoms of our affine shear systems, either
under quasi-stationary construction or non-stationary construction,
obey the parabolic rule and more importantly, at all scales $j$, for
each cone, the dilation matrix is fixed as $\db_\lambda^j$ for all
generators.

\section{MRA Structures and  Filter Banks}
\label{sec:filterBank} In this section  we shall study the MRA
structure of  sequences of frequency-based affine shear tight frames
constructed in Section~\ref{sec:construction} and investigate their
underlying filter banks.

The sequence of systems $\FWS_J(\fphi^J;\{\fPsi_j\}_{j=J}^\infty)$,
$J\ge J_0$ defined as in \eqref{def:FAS_J} has two different
dilation matrices $\da_{\lambda}$ and $\exchg\da_\lambda\exchg$ for
two cones. On the one hand, the functions $\fphi^J, J\ge J_0$ with
the dilation matrix $\dm_{\lambda}$   induce an MRA  $\{\V_j\}_
{j\ge J_0}$ with
\[
\V_j := \overline{{\rm  span}}\{\varphi^j(\dm_{\lambda}^j\cdot-\vk):
\vk\in\Z^2\},
\]
where $\varphi^j := \ft^{-1}\fphi^j$. The function space $\V_j$ is
shift-invariant on the lattice $\dn_\lambda^j\Z^2$. On the other
hand, at scale $j$, let $\psi:=\ft^{-1}\fpsi$ and
$\psi^{j,\ell}:=\ft^{-1}\fpsi^{j,\ell}, |\ell|=r_j+1,\ldots,s_j$.
For convention, $\psi^{j,\ell}:=\psi$ for $|\ell|\le r_j$. Then the
wavelet subspaces $\W_j$ are given by
\[
\W_j :=\overline{{\rm
span}}\{\psi^{j,\ell}(\sh^\ell\da_\lambda^j\cdot-\vk),
\psi^{j,\ell}(\sh^\ell\da_\lambda^j\exchg\cdot-\vk):\ell=-s_j,\ldots,s_j,
\vk\in\Z^2\},\;\; j\ge J_0,
\]
and we have $\V_j\subseteq\V_{j+1}$, $\W_j\subseteq\V_{j+1}$, and
$\V_j+\W_j=\V_{j+1}$ provided
$\FWS_J(\fphi^J;\{\fPsi_j\}_{j=J}^\infty)$ is a tight frame for
$L_2(\R^2)$ for all $J\ge J_0$. While $\V_j$ is shift-invariant on
the lattice $\dn_\lambda^j\Z^2$,  $\W_j$ is not shift-invariant on
the lattice $\dn_{\lambda}^j\Z^2$. We shall next discuss other types
of sequences of  frequency-based  affine  systems that can be
regarded as a `shift-invariant' extension of
$\FWS_J(\fphi^J;\{\fPsi_j\}_{j=J}^\infty)$.

\subsection{Non-stationary MRA and filter banks}
Let us first discuss the MRA and filter bank structure for the
non-stationary construction in Section~\ref{sec:construction}.

Let $\feta, \fzeta$ be defined as in \eqref{def:fetaZeta}. Fixed
$J_0\ge 0$, let $\fTheta^{J_0}$, $\fphi^j$, $\fPsi_j$ be defined as
in \eqref{def:fTheta}, \eqref{def:fphij},
\eqref{def:fPsi-j-non-station} respectively. Let
$\FASS_J(\fphi^J;\{\fPsi_j\}_{j=J}^\infty)$ be defined with these
$\fphi^j, \fPsi_j$. By modifying elements from original
$\fpsi^{j,\ell}$ in \eqref{def:fpsijl}, we define
$\widetilde\fPsi_j$ as follows:
\begin{equation}\label{def:fpsijl-framelets}
\begin{aligned}
\widetilde\fPsi_j:=\{\widetilde\fpsi^{j,\ell}:
\ell=-\ell_{\lambda^j},\ldots,\ell_{\lambda^j}\}
\end{aligned}
\end{equation}
with
\begin{equation}
\label{def:fpsijl-framelet-inner}
\widetilde\fpsi^{j,\ell}(\xi):=\fomega_{\lambda,t,\rho}^j(\xi)\frac{\fgamma_{\varepsilon}(\lambda^j\xi_2/\xi_1+\ell)}{\fGamma^j(\xi)},
\quad\xi\neq0,\;\; |\ell|\le \ell_{\lambda^j}-1,
\end{equation}
\begin{equation}
\label{def:fpsijl-framelet-inner}
\widetilde\fpsi^{j,\pm\ell_{\lambda^j}}(\xi):=\fomega_{\lambda,t,\rho}^j(\xi)
\frac{\fgamma^\mp_{\lambda^j,\varepsilon,\varepsilon_0}(\lambda^j\xi_2/\xi_1\pm\ell_{\lambda^j})}
{\fGamma^j(\xi)},\quad\xi\neq0,
\end{equation}
and $\widetilde\fpsi^{j,\ell}(0):=0$. Again, it is trivial that
$\widetilde\fpsi^{j,\ell},
\ell=-\ell_{\lambda^j},\ldots,\ell_{\lambda^j}$ are compactly
supported functions in $C^\infty(\R^2)$.
 We then use the dilation matrix
$\dm_\lambda$ for all generators $\fphi^j$ and
$\widetilde\fpsi^{j,\ell}$. The \emph{frequency-based  affine
system} is then defined to be
\begin{equation}
\begin{aligned}\label{def:cone-framelets-non-station}
\FWS_J(\fphi^J;\{\widetilde\fPsi_j\}_{j=J}^\infty) := &
\{\fphi^J_{\dn_{\lambda}^J;0,\vk}:\vk\in\Z^2\}
\cup\{\fh_{\dn_{\lambda}^j;0,\vk},\fh_{\dn_{\lambda}^j\exchg;0,\vk}: \vk\in\Z^2, \fh\in\widetilde\fPsi_j\}_{j=J}^\infty\\
= & \{\fphi^J_{\dn_{\lambda}^J;0,\vk}:\vk\in\Z^2\}
\cup\{\widetilde\fpsi^{j,\ell}_{\dn_{\lambda}^j;0,\vk},\widetilde\fpsi^{j,\ell}_{\dn_{\lambda}^j\exchg;0,\vk}:
 \vk\in\Z^2, \ell=-\ell_{\lambda^j},\ldots,\ell_{\lambda^j}\}_{j=J}^\infty.
\end{aligned}
\end{equation}
We now discuss connection of \eqref{def:cone-framelets-non-station}
with the directional tight framelets as discussed in
\cite{Han:2012:FWS}.
\begin{theorem}\label{thm:cone-framelets-non-station}
Retaining all the conditions for $\lambda, t, \rho, \varepsilon,
\varepsilon_0, J_0$ as in Theorem~\ref{thm:non-station}. Let
$\FWS_J(\fphi^J;\{\widetilde\fPsi_j\}_{j=J}^\infty)$ be defined as
in \eqref{def:cone-framelets-non-station} with $\fphi^j$ and
$\widetilde\fPsi_j$ being given by \eqref{def:fphij} and
\eqref{def:fpsijl-framelets}, respectively. Then
$\FWS_J(\fphi^J;\{\widetilde\fPsi_j\}_{j=J}^\infty)$ is a
frequency-based affine  tight frame for $L_2(\R^2)$ for all $J\ge
J_0$; that is,
$\{\fphi^J\}\cup\{\widetilde\fPsi_j\}_{j=J}^\infty\subseteq
L_2(\R^2)$ and
\begin{equation}\label{def:tight-framelets}
\begin{aligned}
(2\pi)^2\|\ff\|_2^2=\sum_{\vk \in\Z^2}
|\ip{\ff}{\fphi^J_{\dn_{\lambda}^J;0,\vk}}|^2
&+\sum_{j=J}^\infty\sum_{\fh\in\widetilde\fPsi_j}\sum_{\vk \in\Z^2}
(|\ip{\ff}{\fh_{\dn_\lambda^{j};0,\vk}}|^2+|\ip{\ff}{\fh_{\dn_\lambda^{j}\exchg;0,\vk}}|^2)
\end{aligned}\; \forall \ff\in L_2(\R^2).
\end{equation}
All elements of $\FWS_J(\fphi^J;\{\widetilde\fPsi_j\}_{j=J}^\infty)$
are compactly supported functions in $C^\infty(\R^2)$. Moreover, let
$\FWS_J(\fphi^J;\{\fPsi_j\}_{j=J}^\infty)$ be the frequency-based
affine shear tight frame with $\fPsi_j=\{\fpsi^{j,\ell}:
\ell=-\ell_{\lambda^j},\ldots,\ell_{\lambda^j}\}$ being given in as
\eqref{def:fpsijl}. Then $\FWS_J(\fphi^J;\{\fPsi_j\}_{j=J}^\infty)$
and $\FWS_J(\fphi^J;\{\widetilde\fPsi_j\}_{j=J}^\infty)$ are
connected to each other by the following relations:
\begin{equation}\label{eq:Connection}
\fpsi^{j,\ell}_{\sh_\ell\db_\lambda^j;0,\vk}=\lambda^{j/2}\widetilde\fpsi^{j,\ell}_{\dn_{\lambda}^j;0,\dd_\lambda^{j}\sh^\ell\vk},\quad
\fpsi^{j,\ell}_{\sh_\ell\db_\lambda^j\exchg;0,\vk}=\lambda^{j/2}\widetilde\fpsi^{j,\ell}_{\dn_{\lambda}^j\exchg;0,\dd_\lambda^{j}\sh^\ell\vk}.
\end{equation}
\end{theorem}
\begin{proof}
It is obvious that
\[
|\fphi^j(\xi)|^2+\sum_{\fh\in\widetilde\fPsi_j}(|\fh(\xi)|^2+|\fh(\exchg\xi)|^2) = |\fphi^{j+1}(\dn\xi)|^2,\; a.e.\; \xi\in\R^2.
\]
By our construction, \eqref{eq:SimpleCond1} and
\eqref{eq:SimpleCond3} hold. Now, the conclusion follows from
\cite[Corollary 18]{Han:2012:FWS} that
$\FWS_J(\fphi^J;\{\widetilde\fPsi_j\}_{j=J}^\infty)$ is a
frequency-based affine tight frame for $L_2(\R^2)$ for all $J\ge
J_0$. Obviously, all elements of
$\FWS_J(\fphi^J;\{\widetilde\fPsi_j\}_{j=J}^\infty)$ are compactly
supported functions in $C^\infty(\R^2)$.

Since
\[
\fpsi^{j,\ell}_{\sh_\ell\db_\lambda^j;0,\vk}(\xi)
=\lambda^{3j/2}\fpsi^{j,\ell}(\sh_\ell\db_\lambda^j\xi)e^{-i\vk\cdot
\sh_\ell\db_\lambda^j\xi}=\lambda^{-3j/2}\fomega_{\lambda,t,\rho}^j(\dn_\lambda^j\xi)\frac{\gamma_\varepsilon(\lambda^j\xi_2/\xi_1+\ell)}{\fGamma^j(\xi)}
e^{-i\vk\cdot \sh_\ell\db_\lambda^j\xi}
\]
and
\[
\widetilde\fpsi^{j,\ell}_{\dn_\lambda^j;0,\vk}(\xi)
=\lambda^{-2j}\widetilde\fpsi^{j,\ell}(\dn_\lambda^j\xi)e^{-i\vk\cdot
\dn_\lambda^j\xi}=\lambda^{2j}\fomega_{\lambda,t,\rho}^j(\dn_\lambda^j\xi)\frac{\gamma_\varepsilon(\lambda^j\xi_2/\xi_1+\ell)}{\fGamma^j(\xi)}
e^{-i\vk\cdot \dn_\lambda^j\xi},
\]
by noting that $\sh_\ell\db_\lambda^j =
\sh_\ell\dd_\lambda^j\dn_\lambda^j$, \eqref{eq:Connection} is
obviously true.
\end{proof}
When $\lambda$ is an integer, we have
$\dd_\lambda^j\sh^\ell\Z^2\subseteq\Z^2$. Equation
\eqref{eq:Connection} shows that when $\lambda$ is an integer, the
frequency-based affine shear tight frame
$\FWS_J(\fphi^J;\{\fPsi_j\}_{j=J}^\infty)$ is indeed a subsystem of
the frequency-based affine tight frame
$\FWS_J(\fphi^J;\{\widetilde\fPsi_j\}_{j=J}^\infty)$ through
subsampling. Since both of these two systems share the same
refinable functions $\fphi^j$, $\V_j$ of the MRA for these two
systems are the same. However,  at scale $j$, the wavelet subspace
$\widetilde\W_j$ of
$\FWS_J(\fphi^J;\{\widetilde\fPsi_j\}_{j=J}^\infty)$  generated by
$\widetilde\psi^{j,\ell}:=\ft^{-1}\widetilde\fpsi^{j,\ell}$ are
\[
\widetilde\W_j :=\overline{{\rm
span}}\{\widetilde\psi^{j,\ell}(\dm_{\lambda}^j\cdot-\vk),
\widetilde\psi^{j,\ell}(\dm_{\lambda}^j\exchg\cdot-\vk):\ell=-\ell_{\lambda^j},\ldots,\ell_{\lambda^j},
\vk\in\Z^2\}, \;\;j\ge J_0.
\]
It is trivial to see that  $\widetilde\W_j$
is shift-invariant on the lattice $\dn_{\lambda}^j\Z^2$ and moreover, we have $\W_j\subseteq\widetilde \W_j$.

Let us next study the filter bank structure of the frequency-based
affine tight frame
$\FWS_J(\fphi^J;\{\widetilde\fPsi_j\}_{j=J}^\infty)$. By our
construction and requiring $0<\rho<1$, we can choose
$\varepsilon_0<\lambda^2(\rho_0/\rho-1)/2$ so that
$\supp\fphi^j(\dm_{\lambda}\cdot)\subseteq\supp\fphi^{j+1}\subseteq[-\rho_0\pi,\rho_0\pi]^2$
and
$\supp\widetilde\fpsi^{j,\ell}(\dm_{\lambda}\cdot)\subseteq\supp\fphi^{j+1}\subseteq[-\rho_0\pi,\rho_0\pi]^2$
for some $0<\rho<\rho_0<1$. Let
 $2\pi\Z^2$-periodic functions $\fa^j, \fb^{j,\ell}, j\ge J_0$ be
defined as follows.
\begin{equation}\label{def:masks}
\begin{aligned}
\fa^j(\xi)&:=
\begin{cases}
\frac{\fphi^j(\dm_{\lambda}\xi)}{\fphi^{j+1}(\xi)}&\xi\in\supp\fphi^{j}(\dm_\lambda\cdot)\\
0 & \xi\in [-\pi,\pi]^2\backslash\supp\fphi^{j}(\dm_\lambda\cdot),
\end{cases}\\
\fb^{j,\ell}(\xi)&:=\fb^j(\xi)\frac{\fgamma_\varepsilon(\lambda^j\xi_2/\xi_1+\ell)}{\fGamma^j(\xi)}, \;|\ell|<\ell_{\lambda^j}-1,\\
\fb^{j,\pm\ell_{\lambda^j}}(\xi)&:=\fb^j(\xi)\frac{\fgamma^{\mp}_{\lambda^j,\varepsilon,\varepsilon_0}(\lambda^j\xi_2/\xi_1\pm\ell_{\lambda^j})}{\fGamma^j(\xi)},
\end{aligned}
\end{equation}
where $\fb^j(\xi) = \sqrt{\fg^j(\xi)-|\fa^j(\xi)|^2}$ for some function $\fg^j$ defined on $\T^2$ satisfying $\fg^j\equiv 1$ on the support of $\fphi^{j+1}$.
By \cite[Corollary 18 and Theorem 17]{Han:2012:FWS}, we have the following result.
\begin{corollary}\label{cor:filterBank}
Retaining all the conditions for $\lambda, t, \rho, \varepsilon,
\varepsilon_0, J_0$ as in Theorem~\ref{thm:non-station} with
$\lambda$ being an integer, $0<\rho<1$, and
$\varepsilon_0<\lambda^2(\rho_0/\rho-1)$  such that $\supp\fphi^j$
and $\supp\fpsi^{j,\ell}$ are both inside $[-\rho_0\pi,\rho_0\pi]^2$
for some $0<\rho<\rho_0<1$. Let
$\FWS_J(\fphi^J;\{\widetilde\fPsi_j\}_{j=J}^\infty)$, $J\ge J_0$ be
defined as in \eqref{def:cone-framelets-non-station} with
$\widetilde\fPsi_j$ being given as in \eqref{def:fpsijl-framelets}
and let $\fa^j, \fb^{j,\ell}$ be defined as in \eqref{def:masks}.
Then there exist $\fg^j\in C^\infty(\T^2)$, $j\ge J_0$ such that
 $\fa^j, \fb^{j,\ell}\in C^\infty(\T^2)$ for all $j\ge J_0, \ell=-\ell_{\lambda^j},\ldots, \ell_{\lambda^j}$, and  we have
\begin{equation}\label{def:refinable}
\fphi^j(\dm_{\lambda}\xi)=\fa^j(\xi)\fphi^{j+1}(\xi)
\quad\mbox{and}\quad
\widetilde\fpsi^{j,\ell}(\dm_{\lambda}\xi)=\fb^{j,\ell}(\xi)\fphi^{j+1}(\xi),
\; j\ge J_0
\end{equation}
 for a.e. $\xi\in\R^2$. Moreover, $\{\fa^j;\fb^{j,\ell}, \fb^{j,\ell}(\exchg\cdot):
\ell=-\ell_{\lambda^j},\ldots, \ell_{\lambda^j}\}$ is a filter bank
having the perfect reconstruction property, i.e.,
\begin{equation}\label{def:PR1}
\begin{aligned}
&|\fa^j(\xi)|^2+\sum_{\ell=-\ell_{\lambda^j}}^{\ell_{\lambda^j}}(|\fb^{j,\ell}(\xi)|^2+|\fb^{j,\ell}(\exchg\xi)|^2)=1,  \quad a.e.\;\xi\in\sigma_{\fphi^{j+1}},
\end{aligned}
\end{equation}
and
\begin{equation}\label{def:PR2}
\begin{aligned}
&\overline{\fa^j(\xi)} \fa^j(\xi+2\pi
\omega)+\sum_{\ell=-\ell_{\lambda^j}}^{\ell_{\lambda^{j}}}
\Big[\overline{\fb^{j,\ell}(\xi)}\fb^{j,\ell}(\xi+2\pi
\omega)+\overline{\fb^{j,\ell}(\exchg\xi)}\fb^{j,\ell}(\exchg\xi+2\pi
\omega)\Big]= 0
\end{aligned}
\end{equation}
for $a.e.\;\xi\in\sigma_{\fphi^{j+1}}\cap
(\sigma_{\fphi^{j+1}}-2\pi\omega)$ and for
$\omega\in\Omega_{\dm_{\lambda}}\backslash\{0\}$ with
$\Omega_{\dm_{\lambda}}= [\dm_{\lambda}^{-1}\Z^2]\cap[0,1)^2$ and
$\sigma_{\fphi^j}:=\{\xi\in\R^2: \sum_{\vk\in\Z^2}|\fphi^j(\xi+2\pi
\vk)|^2\neq0\}$.
\end{corollary}

\begin{proof}
By our choice of $\varepsilon, \rho, \varepsilon_0, \lambda, t,
J_0$, $\FWS_J(\fphi^J;\{\widetilde\fPsi_j\}_{j=J}^\infty)$  defined
as in \eqref{def:cone-framelets-non-station} with
$\widetilde\fPsi_j$ being given in \eqref{def:fpsijl-framelets} is a
frequency-based affine tight frame for $L_2(\R^2)$ for all $J\ge
J_0$. By the construction of $\fa^j, \fb^{j,\ell}$, it is easily
seen that \eqref{def:refinable} holds. Now by \cite[Theorem
17]{Han:2012:FWS}, \eqref{def:PR1} and \eqref{def:PR2} hold.

Since $\supp\fphi^j(\dm_\lambda\cdot)$ is strictly inside
$\supp\fphi^{j+1}$, by the smoothness of $\fphi^{j}$ and
$\fphi^{j+1}$, it is trivial  that $\fa^j\in C^\infty(\T^2)$. We
next show that there exist $\fg^j\in C^\infty(\T^2)$ such that
$\fb^{j,\ell}\in C^\infty(\T^2)$. Since  $\supp\fphi^{j+1}$ and
$\supp\fpsi^{j,\ell}$ is inside $[-\rho_0\pi, \rho_0\pi]^2$. Then
one can construct a function $\fg^j\in C^\infty(\T^2)$ such that
$\fg^j(\xi)\equiv1$ for $\xi\in [-\rho_0\pi,\rho_0\pi]^2$ and
$\fg^j(\xi)\equiv0$ for $\xi\in \T^2\backslash
[-\rho_1\pi,\rho_1\pi]^2$ and for some  $\rho_1$ such that
$0<\rho_0<\rho_1<1$. Since
$\supp\fomega_{\lambda,t,\rho}^j(\dm_\lambda\cdot)
\subseteq\supp\fphi^{j+1}$, we have
\[
\begin{aligned}
\fomega_{\lambda,t,\rho}^j(\dm_\lambda^{j+1}\xi)&=(|\fphi^{j+1}(\xi)|^2-|\fphi^j(\dm_\lambda\xi)|^2)^{1/2}
=(|\fphi^{j+1}(\xi)|^2-|\fa^j(\xi)\fphi^{j+1}(\xi)|^2)^{1/2}
\\&=(1-|\fa^j(\xi)|^2)^{1/2}\fphi^{j+1}(\xi)=(\fg^j(\xi)-|\fa^j(\xi)|^2)^{1/2}\fphi^{j+1}(\xi).
\end{aligned}
\]
Obviously, $(\fg^j(\xi)-|\fa^j(\xi)|^2)^{1/2}\in C^\infty(\T^2)$.
Then,
\[\begin{aligned}
\widetilde\fpsi^{j,\ell}(\dm_{\lambda}\xi)
=\fomega_{\lambda,t,\rho}^j(\dm_{\lambda}\xi)\frac{\fgamma_\varepsilon(\lambda^j\xi_2/\xi_1+\ell)}{\fGamma^j(\xi)}=
\fb^{j,\ell}(\xi)\fphi^{j+1}(\xi)
\end{aligned}
\]
with $\fb^{j,\ell}(\xi)=\fb^j
(\xi)\frac{\fgamma_\varepsilon(\lambda^j\xi_2/\xi_1+\ell)}{\fGamma^j(\xi)}$
being a function in $\C^\infty(\T^2)$. Similarly,
\[\begin{aligned}
\widetilde\fpsi^{j,\pm\ell_{\lambda^j}}(\dm_{\lambda}\xi)
=\fomega_{\lambda,t,\rho}^j(\dm_{\lambda}\xi)\frac{\fgamma^\mp_{\lambda^j,\varepsilon,\varepsilon_0}(\lambda^j\xi_2/\xi_1\pm\ell)}{\fGamma^j(\xi)}=
\fb^{j,\pm\ell_{\lambda^j}}(\xi)\fphi^{j+1}(\xi).
\end{aligned}
\]
We are done.
\end{proof}

\subsection{Quasi-stationary MRA and filter banks}
Next, let us discuss the MRA and filter bank structure for the
quasi-stationary construction. For the quasi-stationary case, the
 function $\fphi^j$ for the system
$\FAS_J(\fphi^j;\{\fPsi_j\}_{j=J}^\infty)$ is fixed as
$\fphi^j\equiv \fphi$. Consider the system
$\FAS_J(\fphi;\{\fPsi_j\}_{j=J}^\infty)$ defined as in
\eqref{def:DFAS-cone-shearlet-smooth-station}. It is also from a
`denser' directional framelet system. Recall that for this system
$\fphi(\xi):=\falpha_{\lambda,t,\rho}(\xi_1)\falpha_{\lambda,t,\rho}(\xi_2)$,
$\xi=(\xi_1,\xi_2)\in\R^2$  and
$\fPsi_j:=\{\fpsi^{j,\ell}(\sh_\ell\cdot):
\ell=-\ell_{\lambda^j},\ldots,\ell_{\lambda^j}\}$ with
$\fpsi^{j,\ell}(\xi) =
\fomega_{\lambda,t,\rho}(\dd_\lambda^{-j}\sh_{-\ell}\xi)\frac{\fgamma_{\varepsilon}(\xi_2/\xi_1)}{\sqrt{\fGamma_j((\sh_\ell\db_\lambda^j)^{-1}\xi)}}$,
$\xi\in\R^2$, which gives
\[
\begin{aligned}
\fpsi^{j,\ell}_{\sh_\ell\db_\lambda^j;0,\vk}(\xi)=\lambda^{-3j/2}
\fomega_{\lambda,t,\rho}(\lambda^{-2j}\xi)
\frac{\fgamma_{\varepsilon}(\lambda^j\xi_2/\xi_1+\ell)}{\sqrt{\fGamma_j(\xi)}}e^{-i\vk\cdot\sh_\ell\db_\lambda^j\xi}.
\end{aligned}
\]
Now define a new set
\begin{equation}\label{def:fpsijl-framelets-station}
\widetilde
\fPsi_j:=\{\widetilde\fpsi^{j,\ell},:\ell=-\ell_{\lambda^j},\ldots,\ell_{\lambda^j}\}\mbox{
with } \widetilde\fpsi^{j,\ell}(\xi) :=
\fomega_{\lambda,t,\rho}(\xi)\frac{\fgamma_{\varepsilon}(\lambda^j\xi_2/\xi_1+\ell)}{\sqrt{\fGamma_j(\xi)}},\;
\xi\in\R^2,
\end{equation}
which gives
\[
\begin{aligned}
\widetilde\fpsi^{j,\ell}_{\dn^j;0,\vk}(\xi)=\lambda^{-2j}\fomega_{\lambda,t,\rho}(\lambda^{-2j}\xi)
\frac{\fgamma_{\varepsilon}(\lambda^j\xi_2/\xi_1+\ell)}{\sqrt{\fGamma_j(\xi)}}e^{-i\vk\cdot\dn^j\xi}.
\end{aligned}
\]
We then use a fixed dilation matrix $\dm_\lambda$ for all generators
$\fphi$ and $\fpsi^{j,\ell}$. The  \emph{frequency-based  affine
system} is then defined to be
\begin{equation}
\begin{aligned}\label{def:cone-framelets-station}
\FWS_J(\fphi;\{\widetilde\fPsi_j\}_{j=J}^\infty)
= & \{\fphi_{\dn_{\lambda}^J;0,\vk}:\vk\in\Z^2\}\cup\{\fh_{\dn_{\lambda}^j;0,\vk},\fh_{\dn_{\lambda}^j\exchg;0,\vk}: \vk\in\Z^2, \fh\in\widetilde\fPsi_j\}_{j=J}^\infty\\
= &
\{\fphi_{\dn_{\lambda}^J;0,\vk}:\vk\in\Z^2\}\cup\{\widetilde\fpsi^{j,\ell}_{\dn_{\lambda}^j;0,\vk},\widetilde\fpsi^{j,\ell}_{\dn_{\lambda}^j\exchg;0,\vk}:
\vk\in\Z^2,
\ell=-\ell_{\lambda^j},\ldots,\ell_{\lambda^j}\}_{j=J}^\infty.
\end{aligned}
\end{equation}
%
Now similar to the result in
Theorem~\ref{thm:cone-framelets-non-station} for the non-stationary
construction, we have the following result for the quasi-stationary
construction.
\begin{theorem}\label{thm:cone-framelets-station}
Let $\lambda>1$ and  $0<\varepsilon\le1/2, 0<t,\rho\le1$. Let
$\FWS_J(\fphi;\{\widetilde\fPsi_j\}_{j=J}^\infty)$ be defined as in
\eqref{def:cone-framelets-station} with $\fphi$ and
$\widetilde\fPsi_j$ being given by \eqref{def:fphij} and
\eqref{def:fpsijl-framelets-station}, respectively. Then
$\FWS_J(\fphi;\{\widetilde\fPsi_j\}_{j=J}^\infty)$ is a
frequency-based affine tight frames for $L_2(\R^2)$ for all $J\ge
0$; that is,
$\{\fphi\}\cup\{\widetilde\fPsi_j\}_{j=J}^\infty\subseteq L_2(\R^2)$
and
\begin{equation}\label{def:tight-framelets-station}
\begin{aligned}
(2\pi)^2\|\ff\|_2^2=\sum_{\vk \in\Z^2}
|\ip{\ff}{\fphi_{\dn_{\lambda}^J;0,\vk}}|^2
&+\sum_{j=J}^\infty\sum_{\fh\in\widetilde\fPsi_j}\sum_{\vk \in\Z^2}
(|\ip{\ff}{\fh_{\dn_{\lambda}^{j};0,\vk}}|^2+|\ip{\ff}{\fh_{\dn_{\lambda}^{j}\exchg;0,\vk}}|^2)
\end{aligned}\; \forall \ff\in L_2(\R^2).
\end{equation}
All elements of $\FWS_J(\fphi;\{\widetilde\fPsi_j\}_{j=J}^\infty)$
are compactly supported functions in $C^\infty(\R^2)$. Moreover, let
$\FWS_J(\fphi;\{\fPsi_j\}_{j=J}^\infty)$ be the frequency-based
affine tight frame with $\fPsi_j=\{\fpsi^{j,\ell}:
\ell=-\ell_{\lambda^j},\ldots,\ell_{\lambda^j}\}$ being given in
\eqref{def:fpsij-station}. Then
$\FWS_J(\fphi;\{\fPsi_j\}_{j=J}^\infty)$ and
$\FWS_J(\fphi;\{\widetilde\fPsi_j\}_{j=J}^\infty)$ are connected to
each other by the following relations:
\begin{equation}\label{eq:Connection2}
\fpsi^{j,\ell}_{\sh_\ell\db_\lambda^j;0,\vk}=\lambda^{j/2}\widetilde\fpsi^{j,\ell}_{\dn_{\lambda}^j;0,\dd_\lambda^{j}\sh^\ell\vk},\quad
\fpsi^{j,\ell}_{\sh_\ell\db_\lambda^j\exchg;0,\vk}=\lambda^{j/2}\widetilde\fpsi^{j,\ell}_{\dn_{\lambda}^j\exchg;0,\dd_\lambda^{j}\sh^\ell\vk}.
\end{equation}
\end{theorem}
\begin{proof}
The proof is essentially the same as the proof of Theorem~\ref{thm:cone-framelets-non-station}.
\end{proof}
Similarly,  if $\lambda$ is an integer, then
$\dd_\lambda^j\sh^\ell\Z^2\subseteq\Z^2$, and we can define a
sequence of filter banks. In this case, the low-pass filter $\fa$ of
$2\pi\Z^2$-periodic function for $\fphi$ is fixed as follows
\begin{equation}\label{def:fa-station}
\fa(\xi)=\fmu_{\lambda,t,\rho}(\xi_1)\fmu_{\lambda,t,\rho}(\xi_2),
\quad \xi\in\T^2.
\end{equation}
with $\fmu_{\lambda,t,\rho}$ being given by \eqref{def:fmask}. Note
that $\supp\widetilde\fpsi^{j,\ell}(\dm_\lambda\cdot)\subseteq\supp
\fphi$. Define $2\pi\Z^2$-periodic functions $\fb^{j,\ell}$ for
$\widetilde\fpsi^{j,\ell}$, $j\ge 0$   as follows.
\begin{equation}\label{def:fb-station}
\begin{aligned}
\fb^{j,\ell}(\xi):=\fb(\xi)\frac{\fgamma_{\varepsilon}(\lambda^j\xi_2/\xi_1+\ell)}{\fGamma_j(\xi)}
,\quad |\ell|\le \ell_{\lambda^j}.
\end{aligned}
\end{equation}
with $\fb(\xi):=\sqrt{\fg(\xi)-|\fa(\xi)|^2}$ for some function $\fg$ defined on $\T^2$ satisfying $\fg\equiv 1$ on the support of $\fphi$.
Then, we have the following result.
\begin{corollary}\label{cor:filterBank-station}
Let $\lambda>1$ be an integer. Choose $0<\varepsilon\le 1/2$,
$0<t\le1$, and $0<\rho<1$. Let
$\FWS_J(\fphi;\{\widetilde\fPsi_j\}_{j=J}^\infty)$, $J\ge0$ be
defined as in \eqref{def:cone-framelets-station} with
$\widetilde\fPsi_j$ being given in
\eqref{def:fpsijl-framelets-station} and let $\fa, \fb^{j,\ell}$ be
defined as in \eqref{def:fa-station} and \eqref{def:fb-station},
respectively. Then there exists $\fg\in C^\infty(\T^2)$ such that
 $\fa, \fb^{j,\ell}\in C^\infty(\T^2)$ for all $j\ge 0, \ell=-\ell_{\lambda^j},\ldots,\ell_{\lambda^j}$ and  we have
\begin{equation}\label{def:refinable2}
\fphi(\dm_{\lambda}\xi)=\fa(\xi)\fphi(\xi)
\quad\mbox{and}\quad
\widetilde\fpsi^{j,\ell}(\dm_{\lambda}\xi)=\fb^{j,\ell}(\xi)\fphi(\xi)
\end{equation}
 for a.e. $\xi\in\R^2$.  Moreover, $\{\fa;\fb^{j,\ell}, \fb^{j,\ell}(\exchg\cdot): \ell=-\ell_{\lambda^j},\ldots, \ell_{\lambda^j}\}$ is a filter bank having the perfect reconstruction property,
i.e.,
\begin{equation}\label{def:PR1-station}
\begin{aligned}
&|\fa(\xi)|^2+\sum_{\ell=-\ell_{\lambda^j}}^{\ell_{\lambda^j}}(|\fb^{j,\ell}(\xi)|^2+|\fb^{j,\ell}(\exchg\xi)|^2)=1,  \quad a.e.\;\xi\in\sigma_{\fphi},
\end{aligned}
\end{equation}
and
\begin{equation}\label{def:PR2-station}
\begin{aligned}
&\overline{\fa(\xi)} \fa(\xi+2\pi
\omega)+\sum_{\ell=-\ell_{\lambda^j}}^{\ell_{\lambda^{j}}}
\Big[\overline{\fb^{j,\ell}(\xi)}\fb^{j,\ell}(\xi+2\pi
\omega)+\overline{\fb^{j,\ell}(\exchg\xi)}\fb^{j,\ell}(\exchg\xi+2\pi
\omega)\Big]= 0
\end{aligned}
\end{equation}
for $a.e.\;\xi\in\sigma_{\fphi}\cap (\sigma_{\fphi}-2\pi\omega)$ and
for  $\omega\in\Omega_{\dm_{\lambda}}\backslash\{0\}$ with
$\Omega_{\dm_{\lambda}}= [\dm_{\lambda}^{-1}\Z^2]\cap[0,1)^2$ and
$\sigma_{\fphi}:=\{\xi\in\R^2: \sum_{\vk\in\Z^2}|\fphi(\xi+2\pi
\vk)|^2\neq0\}$.
\end{corollary}
\begin{proof}
By our choice of $\varepsilon, \rho, \lambda, t$,
$\FWS_J(\fphi;\{\widetilde\fPsi_j\}_{j=J}^\infty)$ defined as in
\eqref{def:cone-framelets-station} with $\widetilde\fPsi_j$ being
given in \eqref{def:fpsijl-framelets-station} is a frequency-based
affine tight frame for $L_2(\R^2)$ for all $J\ge 0$. By the
construction of $\fa, \fb^{j,\ell}$, it is easily seen that
\eqref{def:refinable2} holds. Now by \cite[Theorem
17]{Han:2012:FWS}, \eqref{def:PR1-station} and
\eqref{def:PR2-station} hold.

Since $\supp\fphi(\dm\cdot)$ is strictly inside $\supp\fphi$, by the
smoothness of $\fphi$, it is trivial  that $\fa\in C^\infty(\T^2)$.
We next show that there exists $\fg\in C^\infty(\T^2)$ such that
$\fb^{j,\ell}\in C^\infty(\T^2)$. Since $\rho<1$, we have
$\supp\fphi=[-\lambda^{-2}\rho\pi,\lambda^{-2}\rho\pi]^2\subseteq[-\rho\pi,\rho\pi]^2$
and $\supp\fomega_{\lambda,t,\rho}\subseteq[-\rho\pi,\rho\pi]^2$.
Then one can construct a function $\fg\in C^\infty(\T^2)$ such that
$\fg(\xi)\equiv1$ for $\xi\in
[-\lambda^{-2}\rho\pi,\lambda^{-2}\rho\pi]^2$ and $\fg(\xi)\equiv0$
for $\xi$ outside $[-\rho\pi,\rho\pi]^2$. Note that
$\supp\fomega_{\lambda,t,\rho}(\lambda^2\cdot)
\subseteq[-\lambda^{-2}\rho\pi,\lambda^{-2}\rho\pi]$ and we have
\[
\begin{aligned}
\fomega_{\lambda,t,\rho}(\lambda^2\xi)&=(|\fphi(\xi)|^2-|\fphi(\lambda^2\xi)|^2)^{1/2}
=(|\fphi(\xi)|^2-|\fa(\xi)\fphi(\xi)|^2)^{1/2}
\\&=(1-|\fa(\xi)|^2)^{1/2}\fphi(\xi)=(\fg(\xi)-|\fa(\xi)|^2)^{1/2}\fphi(\xi).
\end{aligned}
\]
Obviously, $(\fg(\xi)-|\fa(\xi)|^2)^{1/2}\in C^\infty(\T^2)$.
Then,
\[\begin{aligned}
\widetilde\fpsi^{j,\ell}(\dm_{\lambda}\xi)
=\fomega_{\lambda,t,\rho}(\lambda^2\xi)\frac{\fgamma_\varepsilon(\lambda^j\xi_2/\xi_1+\ell)}{\fGamma_j(\xi)}
=
\fb^{j,\ell}(\xi)\fphi(\xi)
\end{aligned}
\]
with $\fb^{j,\ell}(\xi):=\fb
(\xi)\frac{\fgamma_\varepsilon(\lambda^j\xi_2/\xi_1+\ell)}{\fGamma_j(\xi)}$
being a function in $\C^\infty(\T^2)$.
\end{proof}

\section{Discussion and Extension}
In this paper, we mainly investigate affine shear systems in
$L_2(\R^2)$ for the purpose of simplicity of presentation. Our
characterization and construction can  be easily extended to higher
dimensions. In $\R^d$ with $d\ge2$, the shear operator $\sh^\tau$
with $\tau=(\tau_2,\ldots,\tau_{d})\in\R^{d-1}$ and $\da_\lambda$
are of the form:
\[
\sh^\tau
=
\left[
\begin{matrix}
1 & \tau_2 &\ldots& \tau_{d}\\
0 & 1 & \ldots & 0\\
\vdots & \vdots & \ddots & \vdots\\
0 & 0 &\ldots &1\\
\end{matrix}
\right] \quad\mbox{and}\quad \da_\lambda = \left[
\begin{matrix}
\lambda^2 & 0 &\ldots& 0\\
0 & \lambda & \ldots & 0\\
\vdots & \vdots & \ddots & \vdots\\
0 & 0 &\ldots &\lambda\\
\end{matrix}
\right].
\]
Define $\sh_\tau:=(\sh^\tau)^\tp$ and denote $\exchg_n$ to be the
elementary matrix corresponding to the coordinate exchange between
the first axis and the $n$th one. For example, $\exchg_1 = I_d$ and
$\exchg_2=\diag(\exchg,I_{d-2})$. For $d=2$, we have
$\exchg_2=\exchg$. Let $\fPsi_j$ be given by
\[
\fPsi_j:=\{\fpsi(\sh_\ell\cdot): \ell_n=-r_j^n,\ldots,r_j^n, n =
2,\ldots,d\}\cup \{\fpsi^{j,\ell}(\sh_\ell\cdot):
|\ell_n|=r_j^n+1,\ldots,s_j^n, n = 2,\ldots,d\}
\]
with $\ell=(\ell_2,\ldots,\ell_d)\in\Z^{d-1}$ and $\fpsi,
\fpsi^{j,\ell}$ being functions in $L_2^{loc}(\R^d)$. For the low
frequency part, it corresponds to a function $\fphi^j\in
L_2^{loc}(\R^d)$. Then a frequency-base affine shear system in
$\R^d$ is defined to be
\begin{equation}\label{def:FAS-Rd}
\FASS_J(\fphi^J;\{\fPsi_j\}_{j=J}^\infty)=
\{\fphi^J_{\dn_\lambda^j;0,\vk}: \vk\in\dZ\}
\cup\{\fh_{\db_\lambda^j\exchg_n;0,\vk}: \vk\in\dZ, n=1,\ldots,d,\;
\fh\in\fPsi_j\}_{j=J}^\infty,
\end{equation}
where $\dn_\lambda:=\lambda^{-2}I_d$ and
$\db_\lambda:=(\da_\lambda)^{-\tp}$.

 All the characterizations for
frequency-based affine shear tight frames and sequences of
frequency-based affine shear frames can be carried over to the
$d$-dimensional case for the system defined as in
\eqref{def:FAS-Rd}. Since the essential idea of our smooth
non-stationary construction and smooth quasi-stationary construction
is frequency splitting, our 2D construction thus can be easily
extended to any high dimensions once an $\fomega^j$ is constructed
in a way satisfying
$\fomega^j=(|\fphi^{j+1}(\dn_\lambda\cdot)|^2-|\fphi^j|^2)^{1/2}$
for both the non-stationary and quasi-stationary construction.
Filter banks associated with high-dimensional frequency-based affine
systems can be obtained as well as  their connection to cone-adapted
high-dimensional directional framelets.

Several problems remain open in our study of affine shear tight
frames. For example, the existence and  construction of affine shear
tight frames with compactly supported generators in the spatial
doamin. If we drop the tightness requirement, there are  indeed
compactly supported shearlet frames, e.g., see \cite{Lim:cptsht}. In
view of the connection between affine shear tight frames and
cone-adapted directional framelets, one might want to consider the
existence and construction of cone-adapted directional framelets
with compactly supported generators first. Another problem is the
existence of affine shear tight frames with only one smooth
generator; that is, $\fPsi_j:=\{\fpsi(\sh_\ell\cdot): \ell\in I_j\}$
is from one generator $\fpsi$. Considering that the shear operator
along the seamlines is not consistent for both cones, our conjecture
is that there is even no affine shear tight frame with one single
generator that is continuous. In other words, additional seamline
generators seem to be unavoidable when considering cone-adapted
construction. When $\lambda>1$ is an integer, we know  that an
affine shear tight frame can be regarded as a subsystem of a
directional framelet through sub-sampling, from which an underlying
filter bank exists for the affine shear tight frame. However, when
$\lambda>1$ is not an integer, though an affine shear tight frame is
still related to a cone-adapted directional framelet  via
\eqref{eq:Connection} or \eqref{eq:Connection2}, the lattice
$\dd_\lambda^j\sh^\ell\Z^2$ is  no longer an integer lattice, the
sub-sampling procedure thus fails and we do not know whether there
is still an underlying filter bank for such an affine shear tight
frame.

%
%
%



\begin{thebibliography}{10}
\bibitem{Antoine}
J.-P. Antoine, R. Murenzi, and P. Vanderheynst, Directional wavelets
revisited: Cauchy wavelets and symmetry detection in pattern,
\emph{Appl. Comput. Harmon. Anal.} {\bf 6} (3) (1999) 314 -- 345.

\bibitem{CD}
E. J. Cand\`{e}s and D. L. Donoho, New tight frames of curvelets and
optimal representations of objects with piecewise $C^2$
singularities, \emph{Comm. Pure Appl. Math.} {\bf57} (2)  (2004)
219--266.

\bibitem{ChuiHeStockler}
C. K. Chui, W. He, and J. St\"{o}ckler, Compactly supported tight
and sibling frames with maximum vanishing moments, \emph{Appl.
Comput. Harmon. Anal.} {\bf 13} (3) (2002) 224 -- 262.

\bibitem{DoVetterli}
M. N. Do and M. Vetterli, Contourlets, in G. V. Welland, editor,
Beyond Wavelets, Academic Press, 2008.

\bibitem{Elad.Starck:MCA}
M. Elad, J.-L. Starck,  P. Querre, D. L. Donoho, Simultaneous
cartoon and texture image inpainting using morphological component
analysis (MCA), \emph{Appl. Comput. Harmon. Anal.} \textbf{19} (3)
(2005) 340--358.

%
%
\bibitem{Daub:book}
I.~Daubechies, Ten Lectures on Wavelets, CBMS-NSF Regional
Conference Series in Applied Mathematics, \textbf{61}, SIAM,
Philadelphia, PA, 1992.

\bibitem{DaubHanRonShen}
I.~Daubechies, B.~Han, A. Ron, and Z. Shen, Framelets: MRA-based
constructions of wavelet frames, \emph{Appl. Comput. Harmon. Anal.}
{\bf 14} (1) (2003)  1--46.

\bibitem{Donoho01}
 D. L. Donoho, Sparse components of images and optimal atomic
decomposition, \emph{Constr. Approx.} {\bf17} (2001) 353--382.

%


\bibitem{GLLWW04}
K. Guo, D. Labate, W. Lim, G. Weiss, and E. Wilson, Wavelets with
composite dilations, \emph{Electr. Res. Ann. AMS} {\bf10} (2004)
78--87.

\bibitem{GKL06}
K. Guo, G. Kutyniok, and D. Labate, {Sparse multidimensional
representations using anisotropic dilation and shear operators},
Wavelets and Splines (Athens, GA, 2005), Nashboro Press, Nashville,
TN (2006) 189-201.

\bibitem{GLLWW06:2006}
K. Guo, D. Labate, W. Lim, G. Weiss, and E. Wilson, Wavelets with
composite dilations and their MRA properties, \emph{Appl. Comput.
Harmon. Anal.} {\bf 20 } (2006)  231--249.

\bibitem{GuoLabate:2007}
K. Guo and D. Labate, Optimal sparse multidimensional representation
using shearlets, \emph{SIAM J. Math. Anal.} {\bf 9} (2007) 298--318.

\bibitem{GuoLabate:2009}
K. Guo, and D. Labate, Characterization and analysis of edges using
the continuous shearlet transform, \emph{SIAM J. Imag. Sci.} {\bf 2}
(2009) 959--986.

\bibitem{GuoLabate:2011}
 K. Guo, and D. Labate, Analysis and detection of surface
discontinuities using the 3D continuous shearlet transform,
\emph{Appl. Comput. Harmon. Anal.} {\bf30} (2011) 231--242.

\bibitem{GuoLabate:2012}
K. Guo and D. Labate, The construction of smooth Parseval frames of
shearlets, \emph{Math. Model. Nat. Phenom.}, to appear.

\bibitem{Han:97}
B.~Han, On dual wavelet tight frames, {\em Appl. Comput. Harmon.
Anal.} {\bf 4} (1997) 380--413.

%
%
\bibitem{Han:2012:FWS}
B.~Han, Nonhomogeneous wavelet systems in high dimensions,
\emph{Appl. Comput. Harmon. Anal.} {\bf 32} (2012) 169--196.
%
%
%
\bibitem{HanKutShen}
B. Han, G. Kutyniok, and Z. Shen, Adaptive multiresolution analysis
structures and shearlet systems, \emph{SIAM J.  Num. Anal.} {\bf49}
(2011) 1921--1946.



%

\bibitem{Houska}
R. Houska, The nonexistence of shearlet scaling functions,
\emph{Appl. Comput. Harmon. Anal.}  {\bf 32} (1) (2012) 28--44.

\bibitem{KingKutZhuang}
E. J. King, G. Kutyniok, and X. Zhuang, Analysis of data separation
and recovery problems using clustered sparsity, J. Math. Imaging
Vis., to appear.

\bibitem{Lim:cptsht}
P. Kittipoom, G. Kutyniok, and  W.-Q. Lim, Construction of compactly
supported shearlet frames, \emph{Constr. Appr.}  {\bf35} (1) (2012)
21--72.

\bibitem{KutLabate}
 G. Kutyniok and D. Labate, Resolution of the wavefront set using
continuous shearlets. \emph{Trans. Amer. Math. Soc.} {\bf361} (2009)
2719--2754.


\bibitem{shearlets:book}
G. Kutyniok and D. Labate, Shearlets: Multiscale Analysis for
Multivariate Data, Birkh\"{a}user, 2012.


\bibitem{KutLemvigLim}
 G. Kutyniok, J. Lemvig, and W.-Q Lim. Optimally sparse
approximations of 3D functions by compactly supported shearlet
frames, \emph{SIAM J. Math. Anal.} {\bf 44} (2012) 2962--3017.


\bibitem{KutLim}
G. Kutyniok and W.-Q. Lim, Compactly supported shearlets are
optimally sparse, \emph{J. Approx. Theory} {\bf 163} (11) (2011)
1564--1589.

\bibitem{KutSauer}
 G. Kutyniok and T. Sauer. Adaptive directional subdivision
schemes and shearlet multiresolution analysis, \emph{SIAM J. Math.
Anal.} {\bf41} (2009) 1436--1471.

\bibitem{KutShahZhuang}
G. Kutyniok, M. Sharam, and X. Zhuang, ShearLab: A rational design
of a digital parabolic scaling algorithm, \emph{SIAM J. Imaging
Sci.} {\bf 5} (4) (2012) 1291--1332.

\bibitem{Mallat:book}
S. Mallat, A Wavelet Tour of
Signal Processing, Academic Press, 2008.
%
%
%
%
%


\bibitem{RonShen97}
A. Ron and Z. Shen, Affine systems in $L_2(\R^d)$: The analysis of
the analysis operator, \emph{J. Funct.  Anal.} {\bf 148} (2) (1997)
408--447.


%

\bibitem{DualTreeCWT}
I. W. Selesnick, R. G. Baraniuk, and N. G. Kingsbury, The dual-tree
complex wavelet transform, \emph{IEEE Signal Process. Mag.} {\bf 22}
(6) (2005) 123--151.

\end{thebibliography}
\end{document}